\documentclass[12pt,reqno]{amsart}

\usepackage{graphics}
\usepackage{cite}
\usepackage{amsmath}
\usepackage{amsfonts}
\usepackage{amssymb}
\usepackage{float}
\usepackage{epic}
\usepackage{eepic}
\usepackage{color}
\usepackage{ae}
\usepackage{mathrsfs,array,graphicx}
\usepackage[all,ps]{xy}
\usepackage{epstopdf}
\usepackage{a4wide}
\usepackage{enumerate}
\usepackage{hyperref} \hypersetup{colorlinks,allcolors=[rgb]{0,0,0}}
\usepackage{enumitem}
\usepackage{stmaryrd}

\theoremstyle{theorem}
\newtheorem{thmx}{Theorem}

\numberwithin{equation}{subsection}
\newtheorem{theorem}[equation]{Theorem}
\newtheorem{proposition}[equation]{Proposition}

\newtheorem{corollary}[equation]{Corollary}
\newtheorem{lemma}[equation]{Lemma}
\theoremstyle{definition}
\newtheorem{definition}[equation]{Definition}
\theoremstyle{remark}
\newtheorem{remark}[equation]{Remark}

\newcommand{\cind}{{\textup{c-ind}}}
\newcommand{\fkg}{\mathfrak{g}}
\newcommand{\T}{\mathbb{T}}
\newcommand{\F}{\mathbb{F}}
\newcommand{\ol}{\overline}
\newcommand{\ov}{\overline}
\newcommand{\ra}{\rightarrow}
\newcommand{\hra}{\hookrightarrow}

\newcommand{\A}{\mathbb A}
\newcommand{\ad}{\textup{ad}}

\newcommand{\Aut}{\textup{Aut}}

\newcommand{\C}{\mathbb C}
\newcommand{\Cox}{\textup{Cox}}

\newcommand{\End}{\textup{End}}
\newcommand{\fkp}{\mathfrak p}
\newcommand{\fkm}{\mathfrak m}
\newcommand{\Frob}{\textup{Frob}}
\newcommand{\G}{\mathbb G}
\newcommand{\Gal}{{\textup{Gal}}}
\newcommand{\GL}{{\textup {GL}}}
\newcommand{\Mat}{{\textup {Mat}}}

\newcommand{\Hom}{\textup{Hom}}

\newcommand{\Lie}{\tu{Lie}\,}
\newcommand{\one}{\textbf{\textup{1}}}

\newcommand{\Q}{\mathbb Q}

\newcommand{\Qp}{{\mathbb{Q}_p}}
\newcommand{\R}{\mathbb R}

\newcommand{\Res}{\textup{Res}}
\newcommand{\sep}{\textup{sep}}

\DeclareMathOperator{\vol}{vol}
\newcommand{\Spec}{\textup{spec}}

\newcommand{\tu}[1]{\textup{#1}}
\newcommand{\Z}{\mathbb Z}
\newcommand{\Spl}{\textup{Spl}}
\DeclareMathOperator{\elli}{ell}

\usepackage{tikz}
\usepackage{marginnote}
\usepackage{comment}

\newcommand{\wt}{\widetilde}
\providecommand{\abs}[1]{\left\lvert#1\right\rvert}
\DeclareMathOperator{\pr}{pr}

\def\onetoraldatum/{1-toral datum}
\def\zerotoraldatum/{0-toral datum}
\def\zerotoraldata/{0-toral data}
\def\onetoraldata/{1-toral data}
\newcommand{\Tsp}{T^{\textup{sp}}}
\newcommand{\Wsp}{W^{\textup{sp}}}
\DeclareMathOperator{\Cent}{Cent}  	  % centralizer

\newcommand{\bA}{\mathbb{A}}

\newcommand{\bF}{\mathbb{F}}
\newcommand{\bG}{\mathbb{G}}

\newcommand{\bQ}{\mathbb{Q}}
\newcommand{\bR}{\mathbb{R}}

\newcommand{\bT}{\mathbb{T}}

\newcommand{\bZ}{\mathbb{Z}}

\newcommand{\cA}{\mathcal{A}}

\newcommand{\cF}{\mathcal{F}}

\newcommand{\cO}{\mathcal{O}}
\newcommand{\cP}{\mathcal{P}}

\DeclareMathOperator{\fn}{\mathfrak{n}}
\newcommand{\fh}{\mathfrak{h}}

\newcommand{\fr}{\mathfrak{r}}
\newcommand{\ft}{\mathfrak{t}}

\newcommand{\sA}{\mathscr{A}}
\newcommand{\sB}{\mathscr{B}}

\newcommand{\Zp}{\mathbb{Z}_p}
\newcommand{\Qpbar}{\overline{\mathbb{Q}}_p}

\newcommand{\ZZ}{\mathbb Z}

\newcommand{\QQ}{\mathbb Q}

\newcommand{\ev}{\mathrm{ev}}

\newcommand{\Fbar}{\overline{\mathbb F}_p}

\newcommand{\RR}{\mathbb R}

\newcommand{\mm}{\mathfrak m}

\newcommand{\Fpbar}{\Fbar}

\newcommand{\cont}{\mathrm{cont}}

\newcommand{\Htilde}{\widetilde{H}^0}

\newcommand{\rhobar}{\bar{\rho}}

\begin{document}

\title[Congruences and supercuspidal representations]{Congruences of algebraic automorphic forms and supercuspidal representations}
\author{Jessica Fintzen}
\address{University of Cambridge, Cambridge, UK and Duke University, Durham, NC, USA} 
\email{fintzen@maths.cam.ac.uk and fintzen@math.duke.edu}
\author{Sug Woo Shin}
\address{Department of Mathematics, UC Berkeley, Berkeley, CA 94720, USA $/\!/$ Korea Institute for Advanced Study, 85 Hoegiro,
Dongdaemun-gu, Seoul 130-722, Republic of Korea}
\email{sug.woo.shin@berkeley.edu}
\date{}%\today}

\begin{abstract}
	  Let $G$ be a connected reductive group over a totally real field $F$ which is compact modulo center at archimedean places. We find congruences modulo an arbitrary power of $p$ between the space of arbitrary automorphic forms on $G(\A_F)$ and that of automorphic forms with supercuspidal components at $p$, provided that $p$ is larger than the Coxeter number of the absolute Weyl group of $G$. We illustrate how such congruences can be applied in the construction of Galois representations.
	
	Our proof is based on type theory for representations of $p$-adic groups, generalizing the prototypical case of $\GL_2$ in \cite[\S7]{Scholze-LT} to general reductive groups. We exhibit a plethora of new supercuspidal types consisting of arbitrarily small compact open subgroups and characters thereof. We expect these results of independent interest to have further applications. For example, we extend the result by Emerton--Pa\v{s}k\=unas on density of supercuspidal points from definite unitary groups to general $G$ as above.
\end{abstract}

\makeatletter
\let\@wraptoccontribs\wraptoccontribs
\makeatother

\contrib[with appendix C by]{Vytautas Pa\v{s}k\=unas}
\address{Fakult\"{a}t f\"{u}r Mathematik, Universit\"{a}t Duisburg--Essen, 45117 Essen, Germany}
\email{paskunas@uni-due.de}
\contrib[and appendix D by]{Rapha\"el Beuzart-Plessis}
\address{Aix Marseille Univ, CNRS, Centrale Marseille I2M, Marseille, France}
\email{raphael.beuzart-plessis@univ-amu.fr}

\maketitle

\setlength{\parskip}{0pt}
\setcounter{tocdepth}{2}
\tableofcontents
\setlength{\parskip}{1ex plus 0.5ex minus 0.2ex}

\section*{Introduction}

 \vspace{.1in}

Congruences between automorphic forms have been an essential tool in number theory since Ramanujan's discovery of congruences for the $\tau$-function, for instance in Iwasawa theory and the Langlands program. Over time, several approaches to congruences have been developed via Fourier coefficients, geometry of Shimura varieties, Hida theory, eigenvarieties, cohomology theories, trace formula, and automorphy lifting.

In this paper we construct novel
congruences between automorphic forms in quite a general setting using \emph{type theory} of $p$-adic groups, generalizing the argument in \cite[\S7]{Scholze-LT} for certain quaternionic automorphic forms. More precisely, we produce congruences mod $p^m$ (in the sense of Theorem \ref{thm:B} below) between arbitrary automorphic forms of general reductive groups $G$ over totally real number fields that are compact modulo center at infinity with automorphic forms that are supercuspidal at $p$ under the assumption that $p$ is larger than the Coxeter number of the absolute Weyl group of $G$. In order to obtain these congruences, we prove various results about supercuspidal types that we expect to be helpful for a wide array of applications beyond those explored in this paper.

\textbf{Global results}

In order to describe our global results in more details, let $G$ be a connected reductive group over a totally real field $F$ whose $\R$-points are compact modulo center under every real embedding. Fix an open compact subgroup $U^p\subset G(\A_F^{\infty,p})$. Let $U_p\subset \prod_{w|p} G(F_w)$ be an open compact subgroup, let $A$ denote a commutative ring with unity, and $\psi_p:U_p\ra A^\times$ a smooth character that yields an action of $U_p$ on $A$. Write $M(U^pU_p,A)$ for the space of $A$-valued automorphic forms of level $U^pU_p$ equivariant for the $U_p$-action on $A$ via $\psi_p$.
 See \S\ref{sec:congruence} below for the precise definition of this space and the Hecke algebra $\bT(U^pU_{p},A)$ acting on it. When $A=\Z_p$ and $\psi_p$ is trivial, the corresponding space is denoted by $M(U^pU_p,\Z_p)$.
	Now for each $m\in \Z_{\ge 1}$, put $A_m:=\Z_p[T]/(1+T+T^2+\cdots + T^{p^m-1})$. There is an obvious ring isomorphism $A_m/(T-1)\simeq \Z/p^m\Z$ induced by $T\mapsto 1$.

Our main global theorem is the following, where $\Cox(G)\in \Z_{\ge 1}$ denotes the maximum of the Coxeter numbers of the irreducible subsystems of the absolute root system for $G$. (The table of Coxeter numbers is given above Proposition \ref{prop:0-toral-abundance}. If $G$ is a torus, set $\Cox(G)=1$.).

\begin{thmx}[Theorem \ref{prop:congruence}]\label{thm:B}Assume $p> \Cox(G)$.
	Then there exist
	\begin{itemize}
		\item	  a basis of compact open neighborhoods $\{U_{p, m}\}_{m \geq 1}$ of $1 \in \prod_{w|p} G(F_w)$ such that
$U_{p,m'}$ is normal in $U_{p,m}$ whenever $m'\ge m$, and
		\item a smooth character $\psi_m: U_{p,m} \ra A_m^\times$ for each $m\geq 1$,
	\end{itemize}
	such that we have isomorphisms of $\bZ_p/(p^m)$-modules (where the $U_{p,m}$-action in $M(\,\cdot\,)$ is trivial on the left hand side and through $\psi_m$ on the right hand side)
	\begin{equation} \label{eq:thmB-space-mod-pm}
	M(U^pU_{p,m}, \bZ_p/(p^m)) \simeq (M(U^pU_{p,m}, A_m/(T-1))) \tag{A.i}
	\end{equation}
	that are compatible with the action of $\T^S(U^p U_{p,m},\bZ_p/(p^m))$ on both sides via the
	$\Z_p$-algebra isomorphism
	\begin{equation} \label{eq:thmB-Hecke-mod-pm} \T^S(U^p U_{p,m},\Z_p/(p^m)) \simeq \T^S(U^p U_{p,m},A_m/(T-1)). \tag{A.ii} \end{equation}
	Moreover, every automorphic representation of $G(\bA_F)$ that contributes to
	$$(M(U^pU_{p,m}, A_m)) \otimes_{\bZ_p} \overline\bQ_p$$  is supercuspidal at all places above $p$.
\end{thmx}

When $G$ is a definite unitary group which splits at $p$, Theorem \ref{thm:B} was obtained for all primes $p$ via Bushnell--Kutzko types independently by Kegang Liu. The proof will appear in his forthcoming work on generalizing some of the main constructions in \cite{Scholze-LT} from $\GL_2$ to $\GL_n$.

In fact it is technically convenient to allow self-direct sums on both sides of \eqref{eq:thmB-space-mod-pm}, see Theorem \ref{prop:congruence} below.
We also prove the analogue of the theorem for non-constant coefficients instead of the constant coefficient $\bZ_p$. See Theorem \ref{prop:congruenceV} for the precise statement.
The normal subgroup property of $\{U_{p, m}\}_{m \geq 1}$ in Theorem \ref{thm:B} is not used in the applications in this paper, but might be helpful in some settings, e.g., see \cite[\S4]{EP18}.

Notice that \eqref{eq:thmB-space-mod-pm} is a congruence modulo an arbitrary power of $p$ between the space $M(U^pU_{p,m},{\Z_p})$ that represents automorphic forms of arbitrary level (as one can choose smaller $U^p$ and larger $m$) and the space $M(U^pU_{p,m},A_m)$ representing automorphic forms that are supercuspidal at $p$.
	As such we expect Theorem \ref{thm:B} to be widely applicable, by reducing a question about automorphic representations to the case when a local component is supercuspidal, for instance in the construction of automorphic Galois representations as observed in \cite[Rem.~7.4]{Scholze-LT}. Indeed we illustrate such an application in \S\ref{subsec:app} to reprove the construction of $p$-adic Galois representations associated with regular C-algebraic conjugate self-dual automorphic representations $\Pi$ of $\GL_N$ over a CM field, by reducing to the analogous result of \cite{Clo91,HT01} which assumes that $\Pi$ has a discrete series representation at a finite prime. (Compare with Theorems \ref{thm:HT} and \ref{thm:main-app}.) We achieve this via the congruences of Theorem \ref{thm:B}, assuming  $p>N$ (which should be unnecessary if Liu's result mentioned above is applied). Although this kind of argument is standard (cf.~\cite[1.3]{Tay91}), we supply details as a guide to utilize our theorem in an interesting context.

We also mention a related result of Emerton--Pa\v{s}k\=unas \cite[Thm.~5.1]{EP18} that in the spectrum of a localized Hecke algebra of a definite unitary group, the points arising from automorphic representations with supercuspidal components at $p$ are Zariski dense (when $p$ is a prime such that the unitary group is isomorphic to a general linear group locally at $p$). They start from the notion of ``capture'' \cite[\S2.4]{CDP14}, which can be powered by type theory for $\GL_n$ due to Bushnell--Kutzko. While their theorem and our Theorem \ref{thm:B} do not imply each other, Pa\v{s}k\=unas suggested to us that our local Theorem \ref{thm:A} below should provide sufficient input for their argument to go through for general $G$ as above. We confirm his suggestion to extend their density result.

To explain the statement, we assume that the center of $G$ has the same $\Q$-rank and $\R$-rank as in  \cite[\S5]{EP18}. Define the completed cohomology (cf.~\eqref{eq:completed-cohomology})
$$\tilde H^0(U^p):= \varprojlim\limits_{m\ge 1}\varinjlim\limits_{U_p} M(U^p U_p,\Z_p/(p^m)),$$
where the second limit is over open compact subgroups of $G(F\otimes_{\Q}\Q_p)$.
 The space $\tilde H^0(U^p)$ is acted on by the ``big'' Hecke algebra $\T^S(U^p)$ defined as a projective limit of
 $\T^S(U^p U_p,\Z_p/(p^m))$ over $U_p$ and $m$. Let $\mathfrak m$ be an open maximal ideal of $\T^S(U^p U_p,\Z_p/(p^m))$ for the profinite topology.
 It follows from the definition that classical forms are dense in $\tilde H^0(U^p)$, which consists of $p$-adic automorphic forms, but we show that the density statement still holds when the component at $p$ is required to be supercuspidal.
(See \S\ref{subsec:density} below for undefined notions and the precise formulation.)

\begin{thmx}[cf.~Corollary \ref{cor:capture} and Theorem \ref{thm:zariskidense}]\label{thmx:density}
  Assume $p>\Cox(G)$. Then classical automorphic forms with fixed weight which are supercuspidal at $p$ form a dense subspace in $\tilde H^0(U^p)$. In the spectrum of the $\mathfrak m$-adic completion of $\T^S(U^p)$, such classical automorphic forms are Zariski dense.
\end{thmx}

This theorem has a potential application to a torsion and $p$-adic functoriality result following the outline of \cite[\S5.2]{EP18} when a classical Langlands functoriality from a group $G_1$ to another group $G_2$ is available for automorphic representations on $G_1$ which are supercuspidal at $p$. (We thank Pa\v{s}k\=unas for pointing this out to us.)
For a Jacquet--Langlands-type example, let $G_1$ and $G_2$ be the unit groups of central quaternion algebras over $\Q$ with $G_1$ unramified at $p$ but $G_2$ ramified at $p$, and assume that the set of ramified primes away from $p$ for $G_1$ is contained in that for $G_2$. Then \emph{loc.~cit.}~constructs a transfer from $G_1$ to $G_2$ on the level of completed cohomology, overcoming the local obstruction at $p$ in the classical Jacquet--Langlands that principal series of $G_1(\Q_p)$ do not transfer to $G_2(\Q_p)$. See \cite[3.3.2]{EmertonICM} for a related discussion. (A similar transfer of torsion classes for Shimura curves is obtained in \cite[Cor.~7.3]{Scholze-LT} by a somewhat different argument based on a version of Theorem \ref{thm:B}; this approach should extend to more general groups by using our Theorem \ref{thm:B} and its variants.)

Pa\v{s}k\=unas kindly wrote Appendix \ref{app:Galois-reps} for us in which he shows that the big Hecke algebras are Noetherian in the setup of definite unitary groups. He also constructs automorphic Galois representations for Hecke eigensystems in the completed cohomology only from the analogous result by Clozel \cite{Clo91} via the density result of \cite{EP18}. In particular this gives yet another argument to remove the local hypothesis from \cite{Clo91}, which has the advantage that no restriction on $p$ is required, as it is the case for Bushnell--Kutzko's type theory for $\GL_n$. For general reductive groups, we hope that Theorem \ref{thmx:density} will be similarly useful.

\textbf{Local results}

The data in Theorem \ref{thm:B} are constructed via a new variant of types for representations of $p$-adic groups that we call \textit{omni-supercuspidal types}. The aim of the theory of types is to classify complex smooth irreducible representations of $p$-adic groups up to some natural equivalence in terms of representations of compact open subgroups. Theorems about the existence of $\mathfrak s$-types (types that single out precisely one Bernstein component $\mathfrak s$) lie at the heart of many results in the representation theory of $p$-adic groups and play an important role in the construction of an explicit local Langlands correspondence and the study of its fine structure. The idea of omni-supercuspidal types is that it is harmless for some applications if a type cuts out a potentially large family of supercuspidal representations, not just a single supercuspidal Bernstein component, but that it is important to control the shape of the types better.

Let us elaborate. For our global application, we need to be able to choose the compact open subgroups arbitrarily small in our types. It is also desirable to require the representation of our compact open subgroup detecting supercuspidality to be one-dimensional. Unfortunately the irreducible representations of the compact open subgroups that form $\mathfrak s$-types have neither properties in general. Nevertheless, using the theory of $\mathfrak s$-types, we show that there exists a plethora of omni-supercuspidal types that satisfy the two desiderata and therefore exhibit a much easier structure than $\mathfrak s$-types. For readers who are mainly interested in determining if a certain representation is supercuspidal, our omni-supercuspidal types (Definition \ref{def:omni} and Theorem \ref{thm:existence-omnisctypes}) and the supercuspidal types arising from our intermediate result, Proposition \ref{prop:sc-type}, allow therefore significantly more flexibility and easier detection.

To be more precise, let us introduce some notation. Let $F$ be a finite extension of $\Q_p$, and $G$ a connected reductive group over $F$ with $\dim G\ge 1$.
 A supercuspidal type for $G(F)$ means a pair $(U,\rho)$, where $\rho$ is an irreducible smooth complex representation of an open compact subgroup $U$ of $G(F)$ such that every irreducible smooth complex representation $\pi$ of $G(F)$ for which $\pi|_{U}$ contains $\rho$ is supercuspidal. Note that a supercuspidal type may pick out several Bernstein components.

 We define an \textit{omni-supercuspidal type of level $p^m$} (with $m\in \Z_{\ge 1}$) to be a pair $(U,\lambda)$, where $\lambda$ is a smooth $\Z/p^m\Z$-valued character on an open compact subgroup $U$ of $G(F)$ such that $(U,\chi\circ \lambda)$ is a supercuspidal type for \emph{every} nontrivial character $\chi:\Z/p^m\Z\ra \C^*$.
 The flexibility allowed for $\chi$ is extremely helpful for producing congruences of automorphic forms.
Our main novelty is a \emph{constructive proof} of the following theorem about the existence of such omni-supercuspidal types, where $\textup{Cox}(G)$ is defined as before.

\begin{thmx}[Theorem \ref{thm:existence-omnisctypes}]\label{thm:A}
Suppose $p>\textup{Cox}(G)$. Then there exists an infinite sequence of omni-supercuspidal types $\{(U_m,\lambda_m)\}_{m \in  \Z_{\ge 1}}$ such that $(U_m,\lambda_m)$ has level $p^m$ and $\{U_m\}_{m \in  \Z_{\ge 1}}$ forms a basis of open neighborhoods of $1$ and such that $U_{m'}$ is normal in $U_m$ whenever $m'\ge m$.
\end{thmx}

Our proof provides an explicit description of $U_m$ and $\lambda_m$ using the Moy--Prasad filtration. To give an outline of our approach, suppose for simplicity of exposition that $G$ is an absolutely simple group. (The reduction to this case is carried out in the proof of Theorem \ref{thm:existence-omnisctypes}.) Then $G$ is tamely ramified thanks to the assumption that $p>\Cox(G)$.
Let $T$ be a tamely ramified elliptic maximal torus of $G$, and $\phi:T(F)\ra\C^*$ a smooth character of depth $r\in \R_{>0}$ that is $G$-generic of depth $r$ in the sense of \cite[\S9]{Yu01}. We call such a triple $(T,r,\phi)$ a \emph{0-toral datum}.

 Then our approach is to show (a) that there is an infinite supply of 0-toral data (with same torus $T$) with $r\ra\infty$ and (b) that each 0-toral datum gives rise to an omni-supercuspidal type whose level grows along with $r$. (For general $G$ we consider \onetoraldata/, introduced in Definition \ref{def:1-toral}, as there may not be enough 0-toral data.)

To prove the abundance of 0-toral data (Proposition \ref{prop:0-toral-abundance} below; see Proposition \ref{prop:1-toral-abundance} for a result on 1-toral data), we exhibit a $G$-generic element (of some depth) in the dual of the Lie algebra of $T$.
 After reducing to the case that $G$ is quasi-split, we construct $T$ by giving a favorable Galois 1-cocycle to twist a maximally split maximal torus. To exhibit a $G$-generic element, we use the Moy--Prasad filtration and eventually demonstrate a solution to a certain system of equations; here an additional difficulty comes from a Galois-equivariance condition. While we treat most cases uniformly, we carry out some explicit computations for types $D_{2N+1}$ and $E_6$ that can be found in Appendices \ref{app:Dodd} and \ref{app:E6}.

The second step is to construct omni-supercuspidal types from 0-toral data. By making some additional choices one can enlarge the 0-toral datum to an input for the construction of  Adler (\!\!\cite{Adl98}) and Yu (\!\!\cite{Yu01}).
The construction of Adler and Yu yields then a supercuspidal type $(K,\rho)$ such that the compact induction $\cind^{G(F)}_K \rho$ is irreducible and supercuspidal. Unfortunately, $\rho$ may not be a character, and the groups $K$ (as $r$ and $\phi$ vary) do not form a basis of open neighborhoods because $K\supset T(F)$. However, $\rho$ restricted to the Moy--Prasad filtration subgroup $G_{y,r}\subset K$, where $y$ denotes the point of the Bruhat--Tits building of $G$ corresponding to $T$, is given by a character $\hat \phi$.
We show that $(G_{y,r}, \hat \phi)$ is a supercuspidal type (which in the 0-toral datum case essentially follows from Adler (\!\!\cite{Adl98}), but we also treat the more complicated 1-toral datum case, see Proposition \ref{prop:sc-type}). The character $\hat \phi$ can also be defined on the larger group $G_{y,\frac{r}{2}+}$, and $\hat \phi|_{G_{y,\frac{r}{2}+}}$ factors through a surjective $\Z/p^m\Z$-valued character $\lambda$, where $m$ is proportional to~$r$. Finally we deduce that  $(G(F)_{y,\frac{r}{2}+},\lambda)$ is an omni-supercuspidal type of level $p^m$.

Let us remark on the case when $F$ is a local function field in characteristic $p$. In fact we prove most intermediate results for arbitrary nonarchimedean local fields. In particular Proposition \ref{prop:sc-type} also holds for local function fields. This proposition of independent interest provides a rather small compact open subgroup together with a character (the pair $(G_{y,r}, \hat \phi)$ in the case of 0-toral data) detecting supercuspidality.
It is only in the last crucial step when moving from these one-dimensional supercuspidal types to omni-supercuspidal types that we require the field $F$ to have characteristic zero. The reason is that for function fields $F$ every element of $G(F)_{y,\frac{r}{2}+}/G(F)_{y,r+}$ has order $p$ (as the quotient is a vector space over a finite field) while we require elements of order $p^m$ for our construction of omni-supercuspidal types in order to achieve that $\lambda$ surjects onto $\bZ/p^m\bZ$.

\textbf{From the local results to the global results}

We sketch how to obtain Theorem \ref{thm:B} from Theorem \ref{thm:A} by adopting an idea from \cite[\S7]{Scholze-LT}.\footnote{This is a variant of the older idea to produce congruences of group cohomology via two coefficient modules which contain common factors modulo $p^m$. See \cite[p.5]{TaylorThesis} for instance.}
We use the notation from Theorem \ref{thm:B} and assume for simplicity that $F=\bQ$. From Theorem \ref{thm:A} we obtain a sequence of omni-supercuspidal types  $(U_{p,m},\lambda_{m})_{m \in \bZ_{\geq 1}}$ of level $p^m$ for the $p$-adic group $G \times_\bQ \bQ_p$. We define the character $\psi_{m}$ to be the composite map
$$\psi_{m}:U_{p,m}\stackrel{\lambda_{m}}{\twoheadrightarrow} \Z/p^m\Z \hra A^\times_m,$$
where the second map sends $a$ mod $p^m$ to $T^a$. Then $\psi_{m}$ mod $(T-1)$ is the trivial character on $A_m/(T-1)\simeq \Z_p/(p^m)$, giving the isomorphism \eqref{eq:thmB-space-mod-pm} of Theorem \ref{thm:B}, which can be checked to be Hecke equivariant.
The omni-supercuspidal property of $(U_{p,m},\psi_{m})$ implies that $M(U^pU_{p,m},A_m) \otimes_{\bZ_p} \overline\bQ_p$ is accounted for by automorphic representations with supercuspidal components at $p$. In fact this consideration initially motivated our definition of omni-supercuspidal types.

\textbf{Another approach to omni-supercuspidal types (Appendix D)}

After the paper had been submitted for publication, Beuzart-Plessis discovered another proof of our main local Theorem \ref{thm:A} based on the ideas of \cite{BP16}, where he proves the existence of supercuspidal representations. This is the content of Appendix D. 

Beuzart-Plessis' approach and ours are complementary.
His argument has the advantage that it is simpler and requires no constraints on $p$. Even though his proof is non-constructive, the existence statement of Theorem~\ref{thm:A} is enough to imply Theorems~A and~B (without conditions on $p$). 
In contrast, our proof gives us a large supply of \emph{explicit} omni-supercuspidal types, i.e., we know precisely which supercuspidal representations contain our omni-supercuspidal types and what their Langlands parameters are.
 In particular, our explicit approach leads to a congruence with an automorphic representation whose local component is not only supercuspidal but satisfies other (e.g., endoscopic) properties.  Such an additional control is expected to be useful for global applications (when the group is not a form of $\GL_n$).
  Moreover, our approach yields results in type theory of independent interest.

\section*{Guide for the reader}

The structure of the paper should be clear from the table of contents, but we guide the reader to navigate more easily.
At a first reading, the reader might want to concentrate on 0-toral data (\S\ref{sec:0-toral}) and skip 1-toral data (\S\ref{subsec:1-toral}) as this will significantly reduce notational burden while not sacrificing too much of the main theorems. (See Remark \ref{rem:existence-omnisctypes} and the beginning of \S\ref{subsec:1-toral}.)

One may treat the abundance of 0-toral data for absolutely simple groups (proved in \S\ref{subsec:abundance} together with Appendices \ref{app:Dodd} and \ref{app:E6}) as a black box even though this is the basic engine of our method. The main takeaway is Proposition \ref{prop:0-toral-abundance}. If the reader wants to get a sense of its proof, a good idea might be to focus on the split type A case (Case 2 of the proof of Proposition \ref{prop:0-toral-abundance}). Section \ref{subsec:main-construction} is devoted to the main construction of this paper, namely how to go from 0-toral (and 1-toral) data to omni-supercuspidal types. The key technical input is Proposition \ref{prop:sc-type}. This is a result of independent interest in type theory and implies the subsequent lemmas in its own momentum, paving the way to the main local theorem (Theorem \ref{thm:existence-omnisctypes}).

If the reader is merely interested in global applications, it is possible to go through only basic local definitions and start in \S\ref{sec:congruence}, taking Theorem \ref{thm:existence-omnisctypes} on faith. Sections \ref{subsec:constant-single} and \ref{subsec:non-constant} are concerned with the application of Theorem \ref{thm:existence-omnisctypes}  to build congruences between automorphic forms. While \S\ref{subsec:app} only reproves a known result on construction of Galois representations, we hope that the reader will find the details helpful for their own applications.
Another application is given in \S\ref{subsec:density} to show that the supercuspidal part is dense in the completed cohomology in a suitable sense. In Appendix \ref{app:Galois-reps}, the themes of \S\ref{subsec:app} and \S\ref{subsec:density} intersect: Pa\v{s}k\=unas explains how the density statement can be used to construct automorphic Galois representations via congruences. Appendix D explains another approach to the main local theorem independently from  the main text.

\section*{Notation and Conventions}

We assume some familiarity with \cite{Yu01} in that we do not recall every definition or notion used in \cite{Yu01} (e.g. Bruhat--Tits buildings, Moy--Prasad subgroups). However we do provide precise reference points for various facts we import from the paper.

The symbol for the trivial representation (of any group) is $\one$. Write $\Z_{>0}$ (resp. $\Z_{\ge 0}$) for the set of positive (resp. nonnegative) integers.

Every reductive group is assumed to be \emph{connected} and nontrivial (so that $\dim \ge 1$) in the main text without further comments.

Let $k$ be a field.
Let $\ol k$ (resp. $k^{\textup{sep}}$) denote an algebraic (resp. separable) closure of $k$. When considering algebraic field extensions of $k$, we consider them inside the algebraic closure $\overline k$, which we implicitly fix once and for all. (For instance, this applies when $k$ is the base field $F$ of the main text, which is local in the first two sections and global in the last section.)
Let $X$ be an affine $k$-scheme, and $k'/k$ be a field extension and $k''$ a subfield of $k$. Then we write $X\times_k k'$ or $X_{k'}$ for the base change $X\times_{\Spec (k)} \Spec( k')$, and $\Res_{k/k''} X$ for the Weil restriction of scalars (which is represented by a $k''$-scheme).

Let $F$ be a nonarchimedean local field.  Then we write $\cO_F$ for the ring of integers and $k_F$ for the residue field. We use $v:\ol F \ra \Q\cup\{\infty\}$ to designate the additive $p$-adic valuation map sending uniformizers of $F$ to $1$, and we write $|\cdot|_F:F^\times \ra \R_{>0}^\times$ for the modulus character sending uniformizers to $(\#k_F)^{-1}$. We fix a nontrivial additive character  $\Psi$ on $F$ which is nontrivial on $\cO_F$ but trivial on elements with positive valuations. If $E$ is a finite extension of $F$, then we can extend $\Psi$ to $E$.
We fix such an extension and denote it by $\Psi$ as well.

Let $G$ be a reductive group over $F$ (connected by the aforementioned convention).
We say $G$ is \emph{tamely ramified} (over $F$) if some maximal torus of $G$ splits over a tamely ramified extension of $F$. We write $G_{\ad}$ for the adjoint quotient of $G$,
 and
 $\mathfrak g$ for the Lie algebra of $G$. By $Z(G)$ we mean the center of $G$.
Denote the absolute Weyl group of $G$ by $W=W_G$, and its Coxeter number by $\Cox(G)$. For a (not necessarily maximal) split torus $T \subset G$, we denote by $\Phi(G,T)$ the set of ($F$-rational) roots of $G$ with respect to $T$. For each $\alpha\in \Phi(G,T)$ we write $\check\alpha:\G_m\ra T$ for the corresponding coroot, and $\mathfrak{g}_\alpha$ for the subspace of $\mathfrak{g}$ on which $T$ acts via $\alpha$.

We write $\sB(G,F)$ for the (enlarged) Bruhat--Tits building of $G$ over $F$. When $F'$ is a finite tamely ramified extension of $F$, and $T$ a maximally $F'$-split maximal torus of $G_{F'}$ (defined over $F'$), we write $\sA(T,F')$ for the apartment of $T$ over $F'$. Both $\sB(G,F)$ and $\sA(T,F')$ are embedded in $\sB(G,F')$ so their intersection as in D2 of \S\ref{sec:0-toral} below makes sense.
Given a point $y\in \sB(G,F')$ and $s\in \R_{\ge 0}$ (resp. $s \in \bR$), let $G(F')_{y,s}$ (resp. $\fkg(F')_{y,s}$) denote the Moy--Prasad filtration in $G(F')$ (resp. $\mathfrak g(F')$). All Moy--Prasad filtrations are normalized with respect to the fixed valuation $v$.
Write $G(F')_{y,s+}:=\cup_{r>s} G(F')_{y,r}$ and $\fkg(F')_{y,s+}:=\cup_{r>s} \mathfrak{g}(F')_{y,r}$. Moreover, we denote by $[y]$  the image of the point $y$ in the reduced Bruhat--Tits building, and we write $G(F')_{[y]}$ for the stabilizer of $[y]$ in $G(F')$ (under the action of $G(F')$ on the reduced Bruhat--Tits building).
We often abbreviate $G(F)_{y,s}$, $\mathfrak g(F)_{y,s}$, etc as
$G_{y,s}$, $\mathfrak g_{y,s}$, etc, when $F$ is the base field. Similarly we may abuse notation and write $\mathfrak{g}$ for $\mathfrak{g}(F)$.
We denote by $\mathfrak{g}^*$ the $F$-linear dual of $\mathfrak{g}(F)$, and write $\mathfrak{g}^*_{y,s}$ for its Moy--Prasad filtration submodule at $y$ of depth $s$. More generally, if $V$ is an $F'$-vector space, then $V^*$ denotes its $F'$-linear dual.
Let $K$ be an open and closed subgroup of $G(F')$. For a smooth representation $\rho$ of $K$, we write $\cind^{G(F')}_K \rho$ for the compactly induced representation, defined to be the subspace of the usual induction consisting of smooth functions on $G(F')$ whose supports are compact modulo $K$.

Let $\A$ denote the ring of ad\`eles of $\bQ$. Write $\A_k:=\A\otimes_\Q k$ when $k$ is a finite extension of $\Q$. When $S$ is a set of places of $k$, we denote by $\A_k^S$ the subring of elements in $\A_k$ whose components are zero at the places in $S$. E.g. $\A_k^\infty$ denotes the ring of finite ad\`eles over $k$.

\textbf{Acknowledgments}

Both authors heartily thank the organizers (Francesco Baldassarri, Stefano Morra, Matteo Longo) of the 2019 Padova School on Serre conjectures and the $p$-adic Langlands program, where the collaboration got off the ground, and the organizers (Brandon Levin, Rebecca
Bellovin, Matthew Emerton and David Savitt) of the Modularity and Moduli Spaces workshop in Oaxaca in 2019
for giving them a valuable opportunity to discuss with  Matthew Emerton and Vytautas Pa\v{s}k\=unas.
The authors are grateful to Emerton and Pa\v{s}k\=unas for explaining their paper \cite{EP18}. Special thanks are owed to Pa\v{s}k\=unas for writing an appendix and helping with \S\ref{subsec:density}, especially the proof of Proposition \ref{prop:capture}.
We also express our gratitude to Rapha\"el Beuzart-Plessis for writing a note explaining another approach to omni-supercuspidal types and graciously offering it to be included as an appendix.
SWS thanks Tasho Kaletha for valuable feedback and Arno Kret for ongoing joint work which provided inspiration for this paper.
JF is partially supported by NSF grant DMS-1802234 / DMS-2055230 and a Royal Society University Research Fellowship. SWS is partially supported by NSF grant DMS-1802039 and NSF RTG grant DMS-1646385.

\section{Construction of supercuspidal types}\label{sec:sc-types}

Let $F$ be a nonarchimedean local field with residue field $k_F$. The characteristic of $k_F$ is a prime $p>0$.
Let $G$ be a connected reductive group over $F$.
(Most results in the first two sections hold in this generality, sometimes under the condition that $p$ exceeds the Coxeter number of $G$, except that some key statements in \S\ref{subsec:main-construction} require $\textup{char}(F)=0$.)
To fix the idea, every representation is considered on a $\C$-vector space in the first two sections, but everything goes through without change with an algebraically closed field of characteristic zero (e.g. $\ol\Q_p$) as the coefficient field. In \S\ref{sec:congruence} we will thus freely import results from the earlier sections with arbitrary algebraically closed coefficient fields of characteristic zero.

\begin{definition}
  Let $\rho$ be a smooth finite-dimensional representation of an open compact subgroup $U\subset G(F)$.
  We call such a pair $(U,\rho)$ a \textbf{supercuspidal type} if every irreducible smooth representation $\pi$ of $G(F)$ with
  $\Hom_U(\rho,\pi)\neq \{0\}$ is supercuspidal.
\end{definition}

If $(U,\rho)$ is a supercuspidal type, then it cuts out (possibly several) supercuspidal Bernstein components, thus deserving the name.
In type theory, it is typical to construct a supercuspidal type singling out each individual supercuspidal Bernstein component when possible.
We do not insist on this but instead ask for other properties (cf.~\S\ref{subsec:main-construction} below) with a view towards omni-supercuspidal types. The construction data of \cite{Yu01} are too general for us to have a good control, so we restrict our attention to 0-toral and 1-toral data as they should be still general enough for applications.
 Most results of this section are recollections from  \cite{Adl98,Yu01,Fintzen-exhaustion, Fintzen-Yuworks}.

For the remainder of Section \ref{sec:sc-types} we assume that $G$ splits over a tamely ramified extension of $F$.

\subsection{Construction of supercuspidal types from 0-toral data}\label{sec:0-toral}

\begin{definition}\label{def:0-toral}
  A \textbf{0-toral datum} (for $G$) consists of a triple $(T,r,\phi)$, where
\begin{itemize}
  \item $T\subset G$ is a tamely ramified elliptic maximal torus over $F$,
  \item $r\in \R_{>0}$,\footnote{We disregard $r=0$ as it is enough to consider positive depth for our intended applications.}
  \item $\phi:T(F)\ra\C^*$ is a smooth character of depth $r$. If $G \neq T$, then we require $\phi$ to be $G$-generic in the sense of \cite[\S9]{Yu01} (relative to any, or, equivalently, every point in $\sB(G,F)\cap \sA(T,F')$, where $F'$ denotes a finite tame extension of $F$ over which $T$ is split).
\end{itemize}
\end{definition}

A 0-toral-datum gives rise to the following input for Yu's construction of supercuspidal representations, where we use the notation of \cite[\S3]{Yu01} (see also \cite[\S3]{HM08}):
 \begin{itemize}
  \item[D1:] $G^0=T$ is an elliptic maximal torus of $G^1=G$ over $F$, and $T$ is split over a finite tamely ramified extension $F'/F$.
  \item[D2:] $y\in \sB(G,F)\cap \sA(T,F')$ is a point, which we fix once and for all. (Note that the image $[y]$ of the point $y$ in the reduced Bruhat--Tits building does not depend on the choice of $y$.)
    \item[D3:] $r=r_0=r_1>0$ is a real number.
  \item[D4:] $\rho$ is the trivial representation of $K^0:=T(F)$.
  \item[D5:] $\phi=\phi_0:T(F)\ra \C^*$ is a character that is $G$-generic (relative to $y$) of depth $r$ in the sense of \cite[\S9]{Yu01}. The character $\phi_1$ is trivial.
\end{itemize}
\begin{remark}
Note that if $G=T$, then this datum is strictly speaking not satisfying Condition D1 of \cite[\S3]{Yu01} because Condition D1 requires $G^0 \subsetneq G^1$. However, as Yu also points out in \cite[\S15, p.616]{Yu01}, we can equally well work with this ``generalized'' datum.
\end{remark}

Later we will vary $r$ and $\phi$ for a given $T$. The point $y$ will remain fixed.
Following Yu, we write $K^1:=T(F)G(F)_{y,\frac{r}{2}}$ and ${}^0 K^1:=T(F)_y G(F)_{y,\frac{r}{2}}$. In \cite{Yu01}, Yu constructs an irreducible smooth representation $\rho_1$ of $K^1$, and letting $^0 \rho_1:=\rho_1|_{^0 K^1}$, shows the following.

\begin{theorem}[\S2.5 of \cite{Adl98}; Prop 4.6, Thm 15.1 and Cor 15.3 of \cite{Yu01}]\label{thm:Yu-thm}
 The compactly induced representation
 $$\pi:=\cind^{G(F)}_{K^1} \rho_1,$$
  is irreducible and supercuspidal of depth $r$. Moreover, every irreducible smooth representation $\pi'$ of $G(F)$ with $\Hom_{^0 K^1}({}^0 \rho_1,\pi')\neq \{0\}$ is in the same Bernstein component as $\pi$. In particular, $({}^0 K^1,{}^0 \rho_1)$ is a supercuspidal type.
\end{theorem}

\begin{remark}
	These representations were already among those constructed by Adler in \cite{Adl98}. However, we have used Yu's notation here to be consistent with Section \ref{subsec:1-toral}, in which we introduce a class of representations that is more general than that arising from a 0-toral-data.
\end{remark}

\begin{remark}
	The representations obtained from \zerotoraldata/ are the same as the ones constructed from what DeBacker and Spice call ``toral cuspidal $G$-pairs'' in \cite{DS18}. DeBacker and Spice refer to the resulting representations as ``toral supercuspidal representations'' in their introduction.
\end{remark}

During Yu's construction of $\rho_1$, he constructs a smooth character $\hat\phi:T(F)G(F)_{y,\frac{r}{2}+}\ra \C^*$ (denoted $\hat\phi$ in \cite{Yu01}), which is trivial on $(T,G)(F)_{y,(r+,\frac{r}{2}+)}$. Let $J^1_+:=(T,G)(F)_{y,(r,\frac{r}{2}+)}$. Then we show

\begin{proposition}\label{prop:J1-sc-type}
  Assume $p$ does not divide the order of the absolute Weyl group of $G$. Then
  the pair $(J^1_+,\hat\phi|_{J^1_+})$ is a supercuspidal type.
  \end{proposition}
  \begin{proof}
This is a special case of Proposition \ref{prop:J1-sc-type-2} proved in the next section.
  \end{proof}

\subsection{Construction  of supercuspidal types from \onetoraldata/}\label{subsec:1-toral}

Since non-simple reductive groups might not admit a \zerotoraldatum/ in general, we introduce the slightly more general notion of \onetoraldata/.
A reader who is only interested in the final result about the mere existence of enough omni-supercuspidal types, Theorem \ref{thm:existence-omnisctypes}, is welcome to skip all discussions surrounding \onetoraldata/, see also Remark \ref{rem:existence-omnisctypes}. The motivation for the introduction of \onetoraldata/ is to provide the reader with the opportunity  to choose the compact open subgroup appearing in the omni-supercuspidal type to be a Moy--Prasad filtration subgroup. Moreover, we believe that the intermediate results motivated by \onetoraldata/, in particular Proposition \ref{prop:sc-type}, are of independent interest to representation theorists. 

\begin{definition}\label{def:1-toral}
	A \textbf{\onetoraldatum/}
	 is a tuple $((G^0, \hdots, G^d), (r_0, \hdots, r_{d-1}) , (\phi_0, \hdots, \phi_{d-1}))$, where
	\begin{itemize}
		\item  $G^0$ is an elliptic, maximal torus of $G$ and either
		\begin{itemize}
		\item $G^0=G^1=G$, or
		\item $G^0 \subsetneq G^1 \subsetneq \cdots \subsetneq G^{d-1} \subsetneq G^d=G$ is a sequence of twisted Levi subgroups that split over some tamely ramified extension $F'$ of $F$,
		\end{itemize}
		\item $0 < r_0 < r_1 < \hdots < r_{d-1} <r_0+1$ is a sequence of real numbers,
		\item  $\phi_i$ is a smooth character of $G^i(F)$ of depth $r_i$ that is $G^{i+1}$-generic relative to $y$ of depth $r_i$ (if $G^0 \neq G$) for $0 \leq i \leq d-1$ and some (thus every) $y\in \sB(G,F)\cap \sA(G^0,F')$.
	\end{itemize}
\end{definition}

Since $G^0$ is a torus, we will also  write $T$ instead of $G^0$.

A \onetoraldatum/ gives rise to a tuple
\begin{equation}\label{eq:Yu-tuple}
((G^0, \hdots, G^d), y, (r_0, \hdots, r_{d-1}, r_d=r_{d-1}), \rho , (\phi_0, \hdots, \phi_{d-1}, \phi_d=1))
\end{equation}
as in \cite[\S3]{Yu01} (or as in \cite[\S15]{Yu01} in the case of $G=T$) by setting
	\begin{itemize}
		\item $y$ to be a point in $\sB(G,F)\cap \sA(G^0,F')$, which we fix once and for all (note that $[y]$ does not depend on the choice of $y$ in $\sB(G,F)\cap \sA(G^0,F')$),
		\item $\rho$ to be the trivial representation.
	\end{itemize}

\begin{remark}
	A \zerotoraldatum/ is a \onetoraldatum/ with the additional condition that $d=1$. For the reader interested in our choice of nomenclature: The ``0'' in ``0-toral'' refers to the difference between $r_d$ and $r_0$ being 0, while the ``1'' in ``1-toral'' is motivated by the difference between $r_d$ and $r_0$ being smaller than 1.
\end{remark}

\begin{remark}
	The decision to set $\phi_d=1$ is not a serious restriction (and could be removed if desired). By setting $\phi_d=1$ we only exclude additional twists by characters of $G(F)$ if $G$ is not a torus. This convention has the advantage that the translation between the notion in \cite{Yu01} and the notation in \cite{Fintzen-exhaustion} is easier, i.e. we do not have to distinguish the cases $\phi_d=1$ and $\phi_d\neq 1$.
\end{remark}

For $0 \leq i \leq d-1$, we write $H^i$ for the derived subgroup of $G^i$. We set $H^d:=G^d$ (not the derived subgroup of $G^d$ unless $G^d$ is semisimple).
We abbreviate for $1 \leq i \leq d$,
\begin{eqnarray*}
	G^i_{y,r_{i-1}, \frac{r_{i-1}}{2}+} & := & (G^{i-1},G^i)(F)_{y,(r_{i-1}, \frac{r_{i-1}}{2}+)}, \\
	H^i_{y,r_{i-1}, \frac{r_{i-1}}{2}+} & := & H^i(F) \cap  G^i_{y,r_{i-1}, \frac{r_{i-1}}{2}+}.
\end{eqnarray*}
and define $\fh^i_{y,r_{i-1}, \frac{r_{i-1}}{2}+}$ analogously, where $\mathfrak{g}^i$ and $\fh^i$ denote the Lie algebras of $G^i$ and $H^i$, respectively, and set
\begin{eqnarray*}
	J_+^1 &: = &  H^1_{y,r_{0}, \frac{r_{0}}{2}+}	H^2_{y,r_{1}, \frac{r_{1}}{2}+} \hdots H^{d-1}_{y,r_{d-2}, \frac{r_{d-2}}{2}+}	H^{d}_{y,r_{d-1}, \frac{r_{d-1}}{2}+}, \\
	&=& H^1_{y,r_{0}, \frac{r_{0}}{2}+}	H^2_{y,r_{1}, \frac{r_{1}}{2}+} \hdots H^{d-1}_{y,r_{d-2}, \frac{r_{d-2}}{2}+}	G^{d}_{y,r_{d-1}, \frac{r_{d-1}}{2}+}. \\
	K^d & := & G^0_{[y]} H^1_{y,r_{0}, \frac{r_{0}}{2}}	H^2_{y,r_{1}, \frac{r_{1}}{2}} \hdots H^{d-1}_{y,r_{d-2}, \frac{r_{d-2}}{2}}	H^{d}_{y,r_{d-1}, \frac{r_{d-1}}{2}}, \\
	{^0K}^d & := & G^0_{y} H^1_{y,r_{0}, \frac{r_{0}}{2}}	H^2_{y,r_{1}, \frac{r_{1}}{2}} \hdots H^{d-1}_{y,r_{d-2}, \frac{r_{d-2}}{2}}	H^{d}_{y,r_{d-1}, \frac{r_{d-1}}{2}}. \\
\end{eqnarray*}

From the tuple \eqref{eq:Yu-tuple}, Yu constructs in \cite[\S4]{Yu01} an irreducible representation $\rho_d$ of $K^d$ such that the following theorem holds, letting $^0 \rho_d:=\rho_d|_{^0 K^d}$.
\begin{theorem}[Prop 4.6, Thm 15.1, Cor 15.3 of \cite{Yu01} and Thm 3.1 of \cite{Fintzen-Yuworks}]\label{thm:Yu-thm}
	The compactly induced representation
	$$\pi:=\cind^{G(F)}_{K^d} \rho_d,$$
	is irreducible and supercuspidal of depth $r_d$. Moreover, every irreducible smooth representation $\pi'$ of $G(F)$ with $\Hom_{^0 K^d}({}^0 \rho_d,\pi')\neq \{0\}$ is in the same Bernstein component as $\pi$. In particular, $({}^0 K^d,{}^0 \rho_d)$ is a supercuspidal type.
\end{theorem}

The representation $\rho_d$ restricted to $J_+^1$ is given by the character $\hat \phi = \prod_{0 \leq i \leq d-1} \hat \phi_i|_{J_+^1}$ (times identity), where $\hat \phi_i$ is defined as in \cite[\S4]{Yu01}, i.e. $\hat \phi_i$ is the unique character of $(G^0)_{[y]}(G^{i})_{y,0}G_{y,\frac{r_i}{2}+}$ that satisfies
\begin{itemize}
	\item $\hat \phi_i|_{(G^0)_{[y]}(G^{i})_{y,0}}=\phi_i|_{(G^0)_{[y]}(G^{i})_{y,0}}$, and
	\item $\hat \phi_i|_{G_{y,\frac{r_i}{2}+}} $ factors through
	\begin{eqnarray*}
		G_{y,\frac{r_i}{2}+}/G_{y,r_i+} & \simeq& \mathfrak{g}_{y,\frac{r_i}{2}+}/\mathfrak{g}_{y,r_i+} = (\mathfrak{g}^{i} \oplus \fr^i)_{x,\frac{r_i}{2}+}/(\mathfrak{g}^{i} \oplus \fr^i)_{x,r_i+} \\
		& \ra & (\mathfrak{g}^{i})_{x,\frac{r_i}{2}+}/(\mathfrak{g}^{i})_{x,r_i+} \simeq
		(G^{i})_{x,\frac{r_i}{2}+}/(G^{i})_{x,r_i+},
	\end{eqnarray*}
	on which it is induced by $\phi_i$. Here $\fr^i$ is defined to be $\mathfrak{g} \cap \bigoplus_{\alpha \in \Phi(G_{F'},T_{F'}) \setminus \Phi(G^{i}_{F'},T_{F'})} \mathfrak{g}({F'})_\alpha$ for some maximal torus $T$ of $G^{i}$ that splits over a tame extension ${F'}$ of $F$ with $y \in \sA(T,{F'})$, and the surjection $\mathfrak{g}^{i} \oplus \fr^i \twoheadrightarrow \mathfrak{g}^i$ sends $\fr^i$ to zero.
\end{itemize}

\begin{proposition}\label{prop:J1-sc-type-2}
	Assume $p$ does not divide the order of the absolute Weyl group of $G$. Then
the pair $(J^1_+,\hat\phi|_{J^1_+})$ is a supercuspidal type.
\end{proposition}

\begin{remark}The hypothesis on $p$ may not be optimal, but imposed here to import results from \cite{Fintzen-exhaustion,Kal-reg} which assume it. The condition clearly holds if $p$ is larger than the Coxeter number of the Weyl group.
\end{remark}

\begin{proof}[Proof of Proposition \ref{prop:J1-sc-type-2}] The statement is obvious if $G$ is a torus, hence we assume that $G$ is not a torus for the remainder of the proof.
	Let $(\pi, V)$ be an irreducible smooth representation of $G(F)$ such that $V$ contains a one-dimensional subspace $V'$ on which the subgroup $J^1_+ \subset G(F)$ acts via $\hat \phi$. We need to show that $\pi$ is supercuspidal.
	
	Our strategy consists of deriving from the action of $J^1_+ \subset G(F)$ via $\hat \phi$ on $V'$ a maximal datum for $(\pi, V)$ (in the sense of \cite[Def.~4.6]{Fintzen-exhaustion}). This can be achieved because the character $\hat \phi$ on $J^1_+$ encodes the information of a truncated extended datum in the sense of \cite[Def.~4.1]{Fintzen-exhaustion}. The fact that $G^0$ is a torus, hence has trivial derived group, allows us to complete the truncated extended datum to a maximal datum for $(\pi, V)$ by adding the trivial representation of $(G^0)^{\mathrm{der}}=\{1\}$. Then we can apply \cite[Cor.~8.3]{Fintzen-exhaustion}, which is a criterion to deduce that $(\pi, V)$ is supercuspidal from properties of the previously constructed maximal datum for $(\pi, V)$. Let us provide the details.
	
	Since $p$  does not divide the order of the absolute Weyl group of $G$,  the character $\phi_j$ is trivial on $H^j(F) \cap G^j(F)_{y,0+}$ (\!\!\cite[Lemma~3.5.1]{Kal-reg}).
 Hence  $\hat \phi_j|_{H^{i+1}_{y,r_i,\frac{r_i}{2}+}}$ is trivial for $0 \leq i< j <d$. Moreover,  $\hat \phi_j|_{H^{i+1}_{y,r_i,\frac{r_i}{2}+}}$  is trivial for $d> i > j \geq 0$ by the second bullet point of the definition of $\hat \phi_j$. Thus $\hat \phi|_{H^{i+1}_{y,r_i,\frac{r_i}{2}+}}=\hat \phi_i|_{H^{i+1}_{y,r_i,\frac{r_i}{2}+}}$.
	We let $X_i \in \mathfrak{g}^*_{y,-r_i}$ such that the character $\hat \phi|_{H^{i+1}_{y,r_i,\frac{r_i}{2}+}}=\hat \phi_i|_{H^{i+1}_{y,r_i,\frac{r_i}{2}+}}$ viewed as a character of
	$$H^{i+1}_{y,r_i,\frac{r_i}{2}+}/H^{i+1}_{y,r_i+} \simeq \fh^{i+1}_{y,r_i,\frac{r_i}{2}+}/\fh^{i+1}_{y,r_i+}$$ is given by $\Psi \circ X_i$. Since $\phi_i$ is $G^{i+1}$-generic relative to $y$ of depth $r_i$,
	by Yu's definition of genericity \cite[\S8 and \S9]{Yu01}
	we can choose $X_i$ to have the following extra properties: firstly $X_i\in (\Lie(Z(G^i))(F))^* \subset (\mathfrak{g}^i)^*$ (see \cite[\S8]{Yu01} for the definition of this inclusion), and secondly
	$$ v(X_i(H_{\check \alpha}))=-r_i \quad  \text{for all } \alpha \in \Phi(G^{i+1}_{F^{\sep}},T_{F^{\sep}}) \setminus \Phi(G^{i}_{F^{\sep}},T_{F^{\sep}}),$$
where $H_{\check\alpha}:=d\check\alpha(1)\subset \Lie T(F^{\sep})\subset \mathfrak{g}^i(F^{\sep})$. (Here $d \check \alpha: \bG_a \ra \Lie T$ denotes the Lie algebra morphism arising from $\check \alpha: \bG_m \ra T$, and recall that we write $T=G^0$.)
	Since $X_i \in (\Lie(Z(G^i))(F))^*$, we have that
		$$ X_i(H_{\check \alpha})=0 \quad  \text{for all } \alpha \in  \Phi(G^{i}_{F^{\sep}},T_{F^{\sep}})$$
	 and that the $G^{i+1}$-orbit of $X_i$ is closed.

	 Thus $X_i$ is almost stable and generic of depth $-r_i$ at $y$ as an element of $(\mathfrak{g}^{i+1})^*$ in the sense of \cite[Definition~3.1 and Definition~3.5]{Fintzen-exhaustion}.
 By \cite[Cor.~3.8.]{Fintzen-exhaustion} this implies that $X_i$ is almost strongly stable and generic of depth $-r_i$ at $y$. Moreover, since $\phi_i$ is $G^{i+1}$-generic, we have $G^i=\Cent_{G^{i+1}}(X_i)$.  Hence the tuple $(y, (r_i)_{d-1 \geq i \geq 0}, (X_i)_{d-1 \geq i \geq 0}, (G^i)_{d \geq i \geq 0})$ is a truncated extended datum of length $d$ in the sense of \cite[Def.~4.1]{Fintzen-exhaustion}.\footnote{Beware that the sequence of twisted Levi subgroups is increasing in Yu's data, but decreasing in \cite{Fintzen-exhaustion}.}
 Since $J^1_+$ acts on $V'$ via $\hat \phi$, since $(G^0)^{\mathrm{der}}=T^{\mathrm{der}}=\{1\}$, and since $\sB(G^0,F)$ consists of a single facet, we conclude that $(y, (X_i)_{d-1 \geq i \geq 0}, \one)$ is a maximal datum for $(\pi, V)$ in the sense of \cite[Def.~4.6]{Fintzen-exhaustion}.
 	
	Since $y$ is a facet of minimal dimension of $\sB(G^0,F)=\sB(T,F)$ and $Z(G^0)/Z(G)$ is anisotropic, \cite[Cor.~8.3]{Fintzen-exhaustion} implies that $(\pi, V)$ is supercuspidal.
\end{proof}

\section{Construction of a family of omni-supercuspidal types}\label{sec:omni-supercuspidal}

  We have seen that 0-toral and 1-toral data yield supercuspidal types.
  Upgrading this, we will see in this section that 0-toral data and 1-toral data give rise to what we call omni-supercuspidal types (see Definition \ref{def:omni} below), which eventually lead to interesting congruences of automorphic forms in the global setup. Thus it is crucial to have a family of omni-supercuspidal types of level tending to infinity, whose existence we prove in this section.

\subsection{Abundance of 0-toral data and 1-toral data}\label{subsec:abundance}

  The current subsection is devoted to show that 1-toral data exist in abundance under a mild assumption on $p$. We construct an infinite family of 0-toral data of increasing level for absolutely simple groups as an intermediate step.

We begin with a preliminary lemma that will be useful for explicit constructions later.
\begin{lemma} \label{Lemma-a}
	Let $n \in \bZ_{\geq 2}$ and let  $E$ be the degree $n$ unramified extension of $F$. Let $\sigma$ be a generator of $\Gal(E/F)$. Suppose that $p \nmid n$. Then there exists an element $e$ in $E$ with the following properties:
	\begin{itemize}
		\item  $v(e)=0$
		\item  $\sum_{1\leq i \leq n} \sigma^i(e) = 0$
		\item  The image $\bar e$ of $e$ in the residue field $k_E$ of $E$ is a generator for the field extension  $k_E/k_F$.
	\end{itemize}
\end{lemma}
\begin{proof}
Let $e_0 \in E\setminus F$ be an element of valuation zero whose image $\ol e_0$ in $k_E$ generates the extension of finite fields $k_E/k_F$. If $\sum_{1\leq i \leq n} \sigma^i(e_0) = 0$, then we are done, so suppose $\sum_{1\leq i \leq n} \sigma^i(e_0) \neq 0$. Set $e_1 := \sum_{1\leq i \leq n} \sigma^i(e_0)$ and $e:=n\cdot e_0 - e_1$. Then $\sigma(e_1)=e_1$ and $e_1 \in \cO_F$. Hence  the image $\bar e$ of $e:=n \cdot e_0-e_1$ in $k_E$ is a generator for the field extension  $k_E/k_F$ and therefore $v(e)=0$.
Moreover $\sum_{1 \leq i \leq n}\sigma^i(e)=n\sum_{1 \leq i \leq n} \sigma^i(e_0) - n e_1 =0$, so the proof is finished.
\end{proof}

Recall that we write $\Cox(G)\in \Z_{\ge1}$ for the Coxeter number of the absolute Weyl group $W$ of $G$. Note that if $G$ is an absolutely simple reductive group and $p > \Cox(G)$, then $G$ splits over a tamely ramified extension of $F$.
For the reader's convenience, we recall the table of Coxeter numbers for irreducible root systems of all types:
$$
\begin{array}{|c|c|c|c|c|c|c|c|c|c|}
  \hline
  \textup{type~of~} W & A_s & B_s~(s\ge2) & C_s~(s\ge3) & D_s~(s\ge4) & E_6 & E_7 & E_8 & F_4 & G_2 \\\hline
  \textup{Cox}(G)  & s+1 & 2s & 2s & 2s-2 & 12 & 18 & 30 & 12 & 6 \\
  \hline
\end{array}
$$

\begin{proposition} \label{prop:0-toral-abundance}
	Let $G$ be an absolutely simple reductive group over $F$. Assume that $p > \Cox(G)$. Then there exists a tamely ramified elliptic maximal torus $T$ of $G$ that enjoys the following property: For every $n\in \Z_{\ge1}$, there exist a real number $r$ with $n < r \leq n+1$ and a character $\phi:T(F) \ra \mathbb{C}^*$ of depth $r$ such that $(T, r, \phi)$ is a 0-toral datum.
\end{proposition}

\begin{proof}

	We first claim that it suffices to show that $G$ contains an elliptic maximal torus $T$ such that for every $n \in \bZ_{\geq 1}$, there exists a $G$-generic element $X_n \in \ft^*$ of depth $-r$ in the sense of \cite[\S~8]{Yu01} with $n < r \leq n+1$. Suppose that $T$ is such a torus accommodating such an $X_n$ for every $n$. Let $y \in \sA(T, F') \cap \sB(G, F)$ as in \S\ref{sec:0-toral}. Then we can compose $X_n$ with the fixed additive character $\Psi$ to obtain a generic character $\phi$ of $\ft_{r}/\ft_{r+}=\ft_{y,r}/\ft_{y,r+}\simeq T_{y,r}/T_{y,r+}$, which we can view as a character of $T_{y,r}$. Since the character takes image in the divisible group $\mathbb{C}^*$,
	we can extend it to a character of $T(F)$, which we also denote by $\phi$. Then $\phi$ is a $G$-generic character of $T(F)$ of depth $r$ and $(T, r, \phi)$ is a 0-toral datum.
	
	Let us show the existence of $T$ and $X_n$ as above.
For simplicity we write $X$ for $X_n$ when there is no danger of confusion.
	Since over a non-archimedean local field every anisotropic, maximal torus transfers to all inner forms (see \cite[\S10]{Kot86} and \cite[Lemma~3.2.1]{Kal-reg}), we may assume without loss of generality that $G$ is \emph{quasi-split}. (See \cite[3.2]{Kal-reg} to review what transfer of maximal tori means.) Let $\Tsp$ be a maximal torus of $G$ that is maximally split. Denote by $X_*(\Tsp)$ the group of cocharacters of $\Tsp \times_F F^{\textup{sep}}$. Let $E'$ be the splitting field of $\Tsp$, which is finite Galois over $F$. Then the action of $\Gal(F^{\textup{sep}}/F)$ on $X_*(\Tsp)$ factors through $\Gal(E'/F)$. We denote by $\Wsp=N(\Tsp)/\Tsp$ the Weyl group (scheme) of $\Tsp \subset G$.

  Before we go further, we outline the strategy for the rest of the proof. Firstly, we will choose a finite Galois extension $E/F$ (in $F^{\textup{sep}}$) as well as a 1-cocycle of $\Gal(E/F)$ with values in $\Wsp$.  A careful choice will allow us to ``twist'' $\Tsp$ to an elliptic maximal torus $T$ over $F$. Put $\tilde E:=EE'$. (In fact we will have $\tilde E=E$ except possibly for type $D_{2N}$.)
  Secondly, we choose $r\in (n,n+1]$ in such a way that the remaining steps will work and describe an $\cO_{\tilde E}$-linear functional $X$ on the finite free $\cO_{\tilde E}$-module $\ft(\tilde E)_r$ by fixing a convenient basis for $\ft(\tilde E)_r$. Thirdly, we make explicit the conditions that $X\in \ft^*$ (namely $X$ is $\Gal(\tilde E/F)$-equivariant as a linear functional) and that $X$ is $G$-generic of depth $-r$. This yields a list of constraints for the coordinates of $X$ for the fixed (dual) basis. Finally we exhibit a choice of coordinates satisfying all constraints.

  We are about to divide the proof into two cases. The case of type $D_{2N}$ is to receive a special treatment. Assuming $G$ is not of type $D_{2N}$, we define
  $$\delta\in \Aut_{\Z}(X_*(\Tsp))$$
  encoding the action of $\Gal(E'/F)$ on $X_*(\Tsp)$ as follows.
  If $G$ is split, then $E'=F$ and define $\delta$ to be the identity automorphism.
  If $G$ is not split (still excluding $D_{2N}$), then $[E':F]=2$. Write $\sigma$ for the nontrivial element of $\Gal(E'/F)$, and let $\delta$ stand for the action of $\sigma$ in this case.

 Viewing the absolute Weyl group $W:=\Wsp(F^{\textup{sep}})$ as a subgroup of $\Aut_{\Z}(X_*(\Tsp))$, we distinguish two cases as follows.
	
\vspace{.1in}
	
	\noindent \textbf{Case 1: $G$ is of type $D_{2N}$, or $-1 \in W\delta$}\\
	First suppose that $G$ is not of type $D_{2N}$. Let $E$ be a quadratic extension of $F$ that contains $E'$. Then there exists $w \in \Wsp(E)$ such that $w\delta =-1$ (i.e. multiplication by $-1$ on $X_*(\Tsp)$). Let $\bar f: \Gal(F^{\textup{sep}}/F) \ra W$ be the group homomorphism that factors through $\Gal(E/F)$ and sends the non-trivial element $\sigma$ of $\Gal(E/F)$ to $w$. Then $\bar f$ is a 1-cocycle giving an element of the Galois cohomology $H^1(F, \Wsp)$, because $\sigma \in \Gal(E/F)$ acts on $W$ via conjugation by $\delta$.
	
	If $G$ is of type $D_{2N}$, observe that $\Gal(E'/F)$ is one of the following groups: $\{0\},~ \bZ/2\bZ, ~ \bZ/3\bZ$, or $S_3$. By \cite[Lemma~2.2]{Fintzen-good} there exists a quadratic extension $E$ of $F$ such that $E \cap E'= F$, i.e. $\Gal(EE'/F)=\Gal(E/F) \times \Gal(E'/F)$ canonically. Since the center of  $W$ is $\{\pm 1\}$ in this case (\hspace{1sp}\cite[3.19~Cor., Table~3.1, 6.3~Prop.(d)]{Hum90}), we have $-1 \in \Wsp(F)$. Then the group homomorphism $\bar f: \Gal(F^{\textup{sep}}/F) \ra W$ defined by factoring through $\Gal(E/F)$ and sending the non-trivial element $\sigma$ of $\Gal(E/F)$ to $-1$ is an element of $H^1(F, \Wsp)$.

	By \cite[Main~Theorem~1.1]{Rag04} every element of $H^1(F, \Wsp)$ lifts to an element in $H^1(F, N(\Tsp))$ that is contained in $\ker(H^1(F,N(\Tsp)) \ra H^1(F, G))$. Let us denote by $f$ such a lift of $\bar f$ (in both cases). Then $f$ gives rise to (the conjugacy class of) a maximal torus $T$ in $G$ over $F$. The torus $T$ is split over $EE'$, and the nontrivial element $\sigma$ of $\Gal(E/F) \leq \Gal(EE'/F)$ acts on $X_*(T)$ via $\overline{f}(\sigma)\delta=-1$, where we set $\delta=1$ if $G$ is of type $D_{2N}$.

 Hence $T$ is an elliptic maximal torus of $G$.
	
	To avoid convoluted notation, we separate cases according as the quadratic extension $E/F$ is unramified or ramified. Write $\wt E=EE'$.
	
	Suppose first that $E/F$ is unramified, and let $\varpi_F$ be a uniformizer of $F$. Since $p > \Cox(G)$, hence $p \nmid \abs{W}$, the $\cO_{\wt E}$-module $\ft(\wt E)_{y,n+1}$ is spanned by $\{\varpi_F^{n+1} H_{\check \alpha}\}_{\check \alpha \in \check \Delta}$ as a free module, where $\check \Delta$ is a choice of simple coroots of $T \times_F \wt E$  and $H_{\check \alpha}=d\check\alpha(1)$, where $d\check\alpha$ denotes the map $\bG_a \ra \Lie(T \times_F \wt E)$ induced by the coroot $\check \alpha: \bG_m \ra T$. If $G$ is of type $D_{2N}$ we assume in addition that $\check \Delta$ is preserved under the action of $\Gal(\tilde E/E)$, which is possible by the definition of $T$ and $G$ being quasi-split. (Choose  $\check \Delta$ corresponding to a Borel subgroup of $G_E$ over $E$ containing $T_E$, noting that $T_E\simeq \Tsp_E$.)
	Let $a \in E$ with $v(a)=0$ such that $\sigma(a)=-a$, which exists by Lemma \ref{Lemma-a}. Let $X$ be the $\cO_{\wt E}$-linear functional on $\ft(\wt E)_{y,n+1}$ defined by $X(\varpi_F^{n+1} H_{\check\alpha})=a$ for $\check \alpha \in \check \Delta$. Then
	$$X(\sigma(\varpi_F^{n+1} H_{\check\alpha}))=X(\varpi_F^{n+1} H_{-\check\alpha})=X(-\varpi_F^{n+1} H_{\check\alpha})=-a=\sigma(a)=\sigma(X(\varpi_F^{n+1} H_{\check\alpha})), $$
	 and $X$ is also stable under $\Gal(E'/F)$. Hence $X$ defines an $\cO_F$-linear functional on $\ft_{x,n+1}=(\ft(\wt E)_{x,n+1})^{\Gal(\wt E/F)}$. Note that for every coroot $\check \beta$ (of $G_{\wt E}$ with respect to $T_{\wt E}$), we have $X(H_{\check \beta})=m_{\check \beta} a \varpi_F^{-(n+1)}$ for some non-zero integer $m_{\check \beta}$ with $-\Cox(G) < m_{\check \beta} < \Cox(G)$, because the sum of the coefficients of the highest coroot in terms of simple coroots is $\Cox(G)-1$. Hence, since $p > \Cox(G)$, we have $v(X(H_{\check \beta}))=-(n+1)$ for all coroots $\check \beta$ of $G_{\wt E}$ with respect to $T_{\wt E}$. Thus $X$ is $G$-generic of depth $n+1$ in the sense of \cite[\S~8]{Yu01} by our assumption on $p$ and \cite[Lemma~8.1]{Yu01}.

	It remains to treat the case that $E/F$ is totally ramified. Since $p \neq 2$, the quadratic extension $E/F$ is tamely ramified and we may choose a uniformizer $\varpi_E$ of $E$ such that $\sigma(\varpi_E)=-\varpi_E$. Then, as $p > \Cox(G)$, the $\cO_{\wt E}$-module  $\ft(\wt E)_{y,\frac{2n+1}{2}}$ is generated by $\{\varpi_E^{2n+1}H_{\check\alpha}\}_{\check \alpha \in \check\Delta}$, where $\check \Delta$ is chosen as above. We define an $\cO_{\wt E}$-linear functional $X$ on $\ft(\wt E)_{y,\frac{2n+1}{2}}$ by setting $X(\varpi_E^{{2n+1}}H_{\check \alpha})=1$ for $\check \alpha \in \check \Delta$. Then $X(\sigma(\varpi_E^{{2n+1}}H_{\check \alpha}))=1$, and, if $G$ is of type $D_{2N}$, then $X(\tau(\varpi_E^{{2n+1}}H_{\check \alpha}))=X(\varpi_E^{{2n+1}}H_{\tau(\check\alpha)})=1$ for $\tau\in \Gal(\tilde E/E)$. Thus $X$ descends to an $\cO_F$-linear functional on $\ft_{y,\frac{2n+1}{2}}$. Moreover, $X$ is $G$-generic of depth $n+\frac{1}{2}$, because $p > \Cox(G)$. This concludes Case 1.
	
\vspace{.1in}
	
	\noindent  \textbf{Case 2: $G$ is not of type $D_{2N}$ and  $-1 \notin W\delta$}\\
	We claim that in this case $G$ is a split group of type $A_N$, $D_{2N+1}$ or $E_6$ for some integer $N\ge 1$. (In the $D_{2N+1}$-case, $N\ge 2$.) If $G$ is split, this follows from \cite[3.19~Cor.~and Table~3.1]{Hum90}. If $G$ is non-split and not of type $D_{2N}$, then $G$ is of type $A_N$, $D_{2N+1}$ or $E_6$ and $\delta$ is induced by the non-trivial Dynkin diagram automorphism with respect to some set of simple coroots $\check\Delta$. Let $w_0$ be the longest element in the Weyl group $W$ (for simple reflections corresponding to $\check \Delta$). Then $w_0(\check\Delta)=-\check\Delta$. Since $-1 \notin W$ for type $A_N$, $D_{2N+1}$ and $E_6$, we have $w_0 \neq -1$ and therefore $-w_0$ is a nontrivial automorphism of $X_*(T^{\rm{sp}})$ that preserves $\check\Delta$. Hence $-w_0=\delta$, i.e. $-1=w_0\delta \in W\delta$. Thus none of the non-split groups appears in Case 2.
	
	We treat the three cases separately.
When $G$ is a \textbf{split group of type $D_{2N+1}$ or $E_6$}, we exhibit the existence of an appropriate $G$-generic element $X$ in the Lie algebra of a suitable torus by explicit calculations; the details are deferred to Appendix \ref{app:Dodd} and Appendix \ref{app:E6}.

  The remaining case is when \textbf{$G$ is a split group of type $A_N$}. We denote by $\check \alpha_1, \hdots, \check \alpha_N$ simple coroots such that the simple reflections $s_{\check \alpha_i}$ and $s_{\check \alpha_j}$ commute if and only if $\abs{i-j} \neq 1$. Then the Weyl group $W$ contains an element $w$ of order $N+1$ such that $w(\check \alpha_i)=\check \alpha_{i+1}$ for $1\leq i < N$ and $w(\check \alpha_N)=-\sum_{1 \leq i \leq N} \check \alpha_i$. We denote by $E$ the unramified extension of $F$ of degree $N+1$, and let $\sigma$ be a generator of $\Gal(E/F)$.
 Then the map $f: \Gal(F^{\textup{sep}}/F) \twoheadrightarrow \Gal(E/F) \ra W$ defined by sending $\sigma$ to $w$ is an element of $H^1(F,W)$ that gives rise to (the conjugacy class of) an elliptic maximal torus $T$ in $G$ over $F$. Since $p > \Cox(G)=N+1$, the set $\{\varpi_F^{n+1} H_{\check \alpha_i}\}_{1 \leq i \leq N}$ forms an $\cO_E$-basis for $\ft(E)_{y,n+1}$. Let $a \in E$ be an element of valuation zero such that $\sum_{1 \leq i \leq N+1} \sigma^i(a)=0$ and the image $\bar a$ of $a$ in the residue field $k_E$ is a generator for the field extension $k_E/k_F$ (see Lemma \ref{Lemma-a}). Then the linear functional $X$ on $\ft(E)_{y,n+1}$ defined by $X(\varpi_F^{n+1}H_{\check \alpha_i})=\sigma^{i-1}(a)$ descends to an $\cO_F$-linear functional on $\ft_{y,n+1}$.
 Moreover we claim that
 $$v(a+\sigma(a)+\hdots +\sigma^j(a))=0,\qquad \textup{for}\quad 1 \leq j \leq N-1.$$
  Suppose this is false. Then $\bar a+\sigma(\bar a)+\hdots +\sigma^j(\bar a)=0$, where $\bar a$ denotes the image of $a$ in the residue field $k_E$.
  We apply $\sigma$ to the last equation and subtract the original equation to obtain
   that $\sigma^{j+1}(\bar a)=\bar a$. This contradicts that $\bar a$ is a generator of $k_E/k_F$ because $j+1<N+1$.
	Since all coroots of $A_N$ are of the form $\check \alpha_i + \check \alpha_{i+1} + \hdots +\check \alpha_j$ for $1 \leq i \leq j \leq N$, we deduce that $X$ is $G$-generic of depth $n+1$.

\end{proof}

\begin{corollary} \label{cor:0-toral-abundance-adjoint}
	Let $G$ be an adjoint simple reductive group over $F$ that is tamely ramified. Assume that $p > \Cox(G)$. Then there exists a tamely ramified elliptic maximal torus $T$ of $G$ that enjoys the following property: For every $n\in \Z_{\ge1}$, there exists a real number $r$ with $n < r \leq n+1$ and a character $\phi:T(F) \ra \mathbb{C}^*$ of depth $r$ such that $(T, r, \phi)$ is a 0-toral datum.
\end{corollary}

\begin{proof}
  As in the proof of Proposition \ref{prop:0-toral-abundance}, we may and will show the existence of $G$-generic elements of depth $-r$ in the Lie algebra instead of $G$-generic characters of depth $r$.

  By assumption, we can take $G=\Res_{F'/F} G'$ for a finite tamely ramified extension $F'/F$ and an absolutely simple adjoint group $G'$ over $F'$.
  As usual $F'$ is a subfield of $F^{\sep}$; the inclusion is denoted by $\iota_0$.
  By Proposition \ref{prop:0-toral-abundance} there exists a tamely ramified elliptic maximal torus $T'$ of $G'$ over $F'$ such that for each $n'\in \Z_{\ge 1}$, there exists $X\in \ft'(F')$ which is $G'$-generic of depth $-r'$ with $n'<r'\le n'+1$ when the depth is normalized with respect to the valuation  $v':=e \, v : F' \twoheadrightarrow \bZ \cup \{\infty\}$, where $e$ denotes the ramification index of $F'/F$. Writing $\check\Phi'$ for the set of coroots of $G'_{F^{sep}}$ with respect to $T'_{F^{sep}}$, we have
  \begin{equation}\label{eq:G'-generic}
   v'(X(H_{\check\beta}))=-r',\qquad \forall \check\beta\in \check\Phi'.
  \end{equation}
  Then $T:=\Res_{F'/F} T'$ is an elliptic maximal torus of $G$ over $F$. Under the canonical isomorphism $\ft(F)\simeq \ft'(F')$,
  we will show that $X$ viewed as an element of $\ft(F)$ is $G$-generic of depth $-r'/e$. Clearly this finishes the proof. (For each $n$, consider $n'=en$ so that $n<r'/e\le n+\frac{1}{e}\le n+1$.)

  We need some preparation. Over $F^{\sep}$ we have compatible direct sum decompositions
  $$\begin{array}{ccccc}
    \mathfrak{g}(F^{\sep}) & = & \mathfrak{g}'(F')\otimes_F F^{\sep} & = & \bigoplus\limits_{\iota\in \Hom_F(F', F^{\sep})} \mathfrak{g}'(F')\otimes_{F',\iota} F^{\sep}\\
        \cup  &  & \cup & & \cup \\
       \ft(F^{\sep}) & = & \ft'(F')\otimes_F F^{\sep} & = & \bigoplus\limits_{\iota\in \Hom_F(F', F^{\sep})} \ft'(F')\otimes_{F',\iota} F^{\sep}
  \end{array}$$
  Correspondingly the set $\check\Phi$ of coroots of $G_{F^{\sep}}$ with respect to $T_{F^{\sep}}$ decomposes as
  $$\check\Phi = \coprod_{\iota\in \Hom_F(F',F^{\sep})} \check\Phi'_{\iota},$$
 where $\check\Phi'_{\iota}$ is the set of coroots of $G'\times_{F',\iota} F^{\sep}$ with respect to $T'\times_{F',\iota} F^{\sep}$, so that $\check\Phi'=\check\Phi'_{\iota_0}$.

  We are ready to verify that $X$ is $G$-generic of depth $-r'/e$. Namely let us show
    \begin{equation}\label{eq:G-generic}
    v(X(H_{\check\alpha}))=-r'/e,\qquad \forall \check\alpha\in \check\Phi.
    \end{equation}
  Recall that $v'=ev$. So the condition holds for $\check\beta\in \check\Phi'=\check\Phi'_{\iota_0}$ by \eqref{eq:G'-generic}.
  But $\Gal(F^{\sep}/F)$ acts on $\check\Phi$ in a way compatible with its action on the index set $\Hom_F(F',F^{\sep})$. Thus
  the $\Gal(F^{\sep}/F)$-orbit of $\check\Phi'$ exhausts $\check\Phi$. On the other hand, since the valuation is Galois-invariant,
  $$-r'/e= v(X(H_{\check\beta})) = v(\sigma(X(H_{\check\beta}))) = v(X(H_{\sigma(\check\beta)})),\qquad \sigma\in \Gal(F^{\sep}/F),\quad \check\beta\in \check\Phi'.$$
  This completes the proof of \eqref{eq:G-generic}. We are done.
\end{proof}

\begin{proposition}\label{prop:1-toral-abundance}
	Let $G$ be a tamely ramified reductive group over $F$.
	Assume that $p > \Cox(G)$. Then there exists an elliptic maximal torus $T$ of $G$ that splits over a tamely ramified extension with the following property: For every $n \in \Z_{\ge1}$, there exists a \onetoraldatum/ with $G^0=T$, $n < r_0 \leq r_{d-1} \leq n+1$ and, if $G$ is not a torus, then the characters $\phi_0, \hdots, \phi_{d-1}$ are trivial on the center of $G$.
\end{proposition}

\begin{proof}
	If $G$ is a torus, then $T=G$, and for every $n \in \Z_{\ge1}$ there exists $n < r_0 \leq n+1$ such that $T_{r_0} \neq T_{r_0+}$.
	 (This follows, for example, from the Moy--Prasad isomorphism from the analogous assertion for the Lie algebra of $T$; the latter is easy since an increase of 1 in the index corresponds to multiplication by a uniformizer.)
	 Hence there exists a (non-trivial) depth-$r_0$ character $\phi_0$ of $T$, and $((T,G), (r_0), (\phi_0))$ is a \onetoraldatum/.
	
	So assume for the remainder of the proof that $G$ is not a torus.
Let $G_1, \hdots, G_N$ be the simple factors of the adjoint quotient $G_{\ad}$ of $G$. Then there exists  a surjection $\pr: G \ra \prod_{1 \leq i \leq N} G_i$ whose kernel is the center of $G$. For $1 \leq i \leq N$, let $T_i$ be an elliptic maximal torus of $G_i$, and $\phi_i: T_i(F) \ra \mathbb{C}^*$ a character of depth $r_i$ as provided by Proposition \ref{cor:0-toral-abundance-adjoint}. Set $T:=\pr^{-1}(T_1 \times \hdots \times T_N)$. Since $p > \Cox(G)$, the prime $p$ does not divide the order of $\pi_1(G_{\ad})$, and hence the depth of $\phi_i \circ \pr_i: T(F) \ra \mathbb{C}^*$ is $r_i$, where $\pr_i$ denotes the composition of $\pr$ with the projection $ \prod_{1 \leq m \leq N} G_m \twoheadrightarrow G_i$.
We assume without loss of generality that $r_1 \leq r_2 \leq \hdots \leq r_N$. Let $1 = i_1 < i_2 <   \hdots < i_j \leq N$ be integers such that
$$r_{i_1} = \hdots = r_{i_2-1} < r_{i_2} = \hdots = r_{i_3-1} < r_{i_3} = \hdots < r_{i_{j}} = \hdots = r_N.$$
  For $0 \leq k \leq j-1 $, set $\wt r_k=r_{i_{k+1}}$, and define $G^k :=T\pr^{-1}(\prod_{1 \leq m \leq i_{k+2}-1}G_m) \subset G$ and
   $\phi^T_{k}:=\prod_{i_{k+1} \leq m \leq i_{k+2}-1} (\phi_m \circ \pr_m)|_{G^k}$. Then $$((G^0, \hdots, G^{j-1}, G^j:=G), (\wt r_0, \hdots, \wt r_{j-1}), (\phi^T_0, \hdots, \phi^T_{j-1}))$$
   is a desired \onetoraldatum/.
\end{proof}

  As a corollary of the above proof when $N=1$, we strengthen Corollary \ref{cor:0-toral-abundance-adjoint} to allow non-adjoint groups.

\begin{corollary} \label{cor:0-toral-abundance}
	Let $G$ be a tamely ramified reductive group over $F$ whose adjoint quotient $G_{\ad}$ is simple. Assume that $p > \Cox(G)$. Then there exists a tamely ramified elliptic maximal torus $T$ of $G$ that enjoys the following property: For every $n\in \Z_{\ge1}$, there exists a real number $r$ with $n < r \leq n+1$ and a character $\phi:T(F) \ra \mathbb{C}^*$ of depth $r$ such that $(T, r, \phi)$ is a 0-toral datum.
\end{corollary}

\subsection{Construction of omni-supercuspidal types from 0-toral and 1-toral data}\label{subsec:main-construction}

  In this subsection we introduce the notion of omni-supercuspidal types and construct such types from 0-toral and 1-toral data.
  The results of the last subsection then allow us to deduce the main theorem that there exists a large family of omni-supercuspidal types in an appropriate sense.
   The following is the key definition of this section.

\begin{definition}\label{def:omni}
  Let $m\in \Z_{\ge1}$.
  An \textbf{omni-supercuspidal type} of level $p^m$ is a pair $(U,\lambda)$, where $U$ is an open compact subgroup of $G(F)$ and $\lambda:U\twoheadrightarrow \Z/p^m\Z$ is a smooth surjective group morphism such that $(U,\psi\circ \lambda)$ is a supercuspidal type for every nontrivial character $\psi:\Z/p^m\Z\ra \C^*$.
\end{definition}

\begin{remark}\label{rem:level-lowering}
  Let $m>m'\ge 1$ and write $\textup{pr}_{m,m'}:\Z/p^m\Z\twoheadrightarrow\Z/p^{m'}\Z$ for the canonical projection. If $(U,\lambda)$ is an omni-supercuspidal type of level $p^m$, then $(U,\textup{pr}_{m,m'}\circ \lambda)$ is an omni-supercuspidal type of level $p^{m'}$.
\end{remark}

We assume from here until Corollary \ref{cor:existence-omnisctypes} that $G$ is tamely ramified over $F$. (In Theorem \ref{thm:existence-omnisctypes} the group $G$ need not be tamely ramified.)
Fix a \onetoraldatum/
\begin{equation}\label{eq:one-toral-datum}
(G^0, \hdots, G^d),(r_0, \hdots, r_{d-1}), (\phi_0, \hdots, \phi_{d-1}))
\end{equation}
which yields the following input for Yu's construction
$$ ((G^0, \hdots, G^d), y, (r_0, \hdots, r_{d-1}, r_d=r_{d-1}), \rho=\one , (\phi_0, \hdots, \phi_{d-1}, \phi_d=1)). $$
Recall that we introduced an open compact subgroup $J^1_+$ and a smooth character $\hat \phi$ in \S\ref{subsec:1-toral}. The following allows us more flexibility compared to the supercuspidal type $(J^1_+,\hat\phi|_{J^1_+})$ we had before. Combined with Proposition \ref{prop:1-toral-abundance}, it provides an explicit construction of a character on a rather small compact open subgroup that detects supercuspidality. We expect this result to be of independent interest.

\begin{proposition}\label{prop:sc-type}
		Assume $p$ does not divide the order of the absolute Weyl group of $G$. Then
	for every group $U$ such that
	$$G(F)_{y,r_0} \subset U \subset G(F)_{y,\frac{r_d}{2}+},$$
	 the pair $(U, \hat \phi|_U)$ is a supercuspidal type.
\end{proposition}
\begin{remark}
  We have $G(F)_{y,r_0}\subset G(F)_{y,\frac{r_d}{2}+}$ (so that the proposition is not vacuous) if and only if $r_0>r_d/2$, which is always satisfied if $r_0>1$.
\end{remark}

\begin{remark}\label{rem:sc-type}
  When the 1-toral data come from 0-toral data (so that $d=1$, $r=r_0$), the proposition essentially follows from Adler's work.
  More precisely,
  it suffices to handle the case $U=G(F)_{y,r}$.  This case follows immediately from the case $G(F)_{y,r}\subset U= J^1_+\subset G(F)_{y,\frac{r}{2}+}$ (Proposition \ref{prop:J1-sc-type}) thanks to \cite[2.3.4]{Adl98}\footnote{For the reader who likes to check the proof, ``$X_1 \in \mathfrak{m}_{x,-r/2}^\perp$'' in the proof of \cite[2.3.3]{Adl98} (using Adler's notation) should have been  ``$X_1 \in \mathfrak{m}_{x,-r+}^\perp$'', and ``$g \in G_{x,\frac{r}{2}}$'' should have been ``$g \in G_{x,0+}$''.}.
\end{remark}

\begin{proof}
We may and will assume that $r_0>r_d/2$.
 It suffices to consider the case $U=G(F)_{y,r_0}$. If $G$ is a torus, there is nothing to prove, so we assume that $G$ is not a torus for the remainder of the proof.
Let $(\pi,V)$ be an irreducible smooth representation of $G(F)$ and suppose that $(\pi|_{G(F)_{y,r_0}}, V)$ contains the character $\hat\phi|_{G(F)_{y,r_0}}$. By Proposition \ref{prop:J1-sc-type-2} it is enough to show that then  $(\pi|_{J_+^1}, V)$ contains the character $\hat\phi|_{J_+^1}$. Define
$$ J_{0}=G_{y,r_0} , \enskip
J_i  = 	H^{d-i}_{y,r_0} H^{d-i+1}_{y,r_{d-i}, \frac{r_{d-i}}{2}+} \cdots H^{d-1}_{y,r_{d-2}, \frac{r_{d-2}}{2}+}	H^{d}_{y,r_{d-1}, \frac{r_{d-1}}{2}+} \,(1 \leq i \leq d-1), \enskip  J_d=J_+^1.
$$
We show by induction on $j$ that $(\pi|_{J_j}, V)$ contains the character $\hat\phi|_{J_j}$. For $j=0$ the statement is true by assumption. Thus let $1 \leq j \leq d$ and assume the induction hypothesis that $(\pi|_{J_{j-1}}, V)$ contains the character $\hat\phi|_{J_{j-1}}$. We denote by $V_{j-1}$ the largest subspace of $V$ on which $J_{j-1}$ acts via $\hat \phi$.
Let $X_i \in \mathfrak{g}^*_{y,-r_i}$ be such that $\hat \phi_i|_{H^{i+1}_{y,r_i}}$ viewed as a character of
$$H^{i+1}_{y,r_i}/H^{i+1}_{y,r_i+} \simeq \fh^{i+1}_{y,r_i}/\fh^{i+1}_{y,r_i+}$$
 is given by $\Psi \circ X_i$. Since $\phi_i$ is $G^{i+1}$-generic of depth $r_i$ relative to $y$, we can choose $X_i$ to be almost stable and generic of depth $-r_i$ at $y$ (as an element of $(\mathfrak{g}^{i+1})^*$)  in the sense of \cite[Definition~3.1 and Definition~3.5]{Fintzen-exhaustion}; compare with the proof of Propositions \ref{prop:J1-sc-type-2}.
 By \cite[Cor.~3.8.]{Fintzen-exhaustion} this implies that $X_i$ is almost strongly stable and generic of depth $-r_i$ at $y$. Hence the tuple $(y, (X_i)_{d-1 \geq i \geq 0})$ is a truncated datum in the sense of \cite[\S4]{Fintzen-exhaustion}. Moreover, by \cite[Lemma~3.5.1.]{Kal-reg} the character $\phi_k$ is trivial on $H^{k}_{y,0+}$ and hence
  also trivial on $H^{d-i+1}_{y,0+}$ for $d-i+1 \leq k$. Thus $\hat \phi_k$ is trivial on $H^{d-i+1}_{y,r_{d-i}, \frac{r_{d-i}}{2}+}$ if $k \neq d-i$.
   Therefore the action of  $H^{d-i+1}_{y,r_{d-i}, \frac{r_{d-i}}{2}+}$ on $V_{j-1}$ is given by $\hat \phi_{d-i}$ for $1 \leq i \leq j-1$, which means that the action is  as illustrated below:
 $$
   H^{d-i+1}_{y,r_{d-i}, \frac{r_{d-i}}{2}+} ~ \twoheadrightarrow ~ H^{d-i+1}_{y,r_{d-i}, \frac{r_{d-i}}{2}+}/H^{d-i+1}_{y,r_{d-i}+}
   ~ \simeq ~ \fh^{d-i+1}_{y,r_{d-i}, \frac{r_{d-i}}{2}+}/\fh^{d-i+1}_{y,r_{d-i}+} ~~ \stackrel{\Psi \circ X_{d-i}}{\curvearrowright} ~~ V_{j-1}.
   $$

 We can now apply \cite[Cor.~5.2]{Fintzen-exhaustion} repeatedly as in the proof of \cite[Cor.~5.4]{Fintzen-exhaustion}. More precisely, we use the following assignment of notation, where the left hand side denotes the objects in \cite[Cor.~5.2]{Fintzen-exhaustion}  using (only here) the notation of \cite{Fintzen-exhaustion} and the right hand side denotes the objects defined above:
\begin{eqnarray*}
	&& n:=d,  X_i:=X_{d-i}, \, r_i:=r_{d-i}, \,  j:=d-j, \,  H_i:=H^{d-i+1}, \\
	&&	\, T_j:=T\cap H^{d-j+1}=G^0 \cap  H^{d-j+1}, \, \varphi:= \Psi .
\end{eqnarray*}
Choose $\epsilon\in (0,\frac{r_0}{4})$ such that $H^{d-j+1}_{y,r_{d-j}-2\epsilon}=H^{d-j+1}_{y,r_{d-j}}$.
To avoid confusion we write $d'$ (instead of $d$) for the positive number $d$ that occurs in \cite[Cor.~5.2]{Fintzen-exhaustion}, and we set
$$d' :=\max\left(\frac{r_{d-j}}{2}, r_{d-j}-n_0\epsilon\right), $$
where $n_0$ is the smallest integer in $\bZ_{\geq 3}$ such that $r_{d-j}-n_0\epsilon<r_0$.  Then either $d'+\epsilon \ge r_0$ or $d'+\epsilon \geq  r_{d-j}-2\epsilon$, which implies together with the assumption $H^{d-j+1}_{y,r_{d-j}-2\epsilon}=H^{d-j+1}_{y,r_{d-j}}$
(needed in the latter case that $d'+\epsilon \geq  r_{d-j}-2\epsilon$) that $H^{d-j+1}_{y,r_{d-j}+, d'+ \epsilon+} \subset H^{d-j+1}_{y,r_{d-j}+, r_0}$.
Since $H^{d-j+1}_{y,r_{d-j}+, r_0}$ acts via $\hat \phi_{d-j}$ on $V_{j-1}$, and $\hat \phi_{d-j}$ is trivial when restricted to $H^{d-j+1}_{y,r_{d-j}+, r_0}$ (by construction/definition)
the group $H^{d-j+1}_{y,r_{d-j}+, d'+ \epsilon+}$ acts trivially on $V_{j-1}$ as well.
 Moreover, using that $r_{d-j} > r_0 > \frac{r_{d-j}}{2}$, we see that the commutator subgroup $\left[J_{j-1}, H^{d-j+1}_{y,r_{d-j}, \frac{r_{d-j}}{2}+}\right]$ is contained in
 	$$H^{d-j+1}_{y,r_{d-j}+, r_0+} H^{d-j+2}_{y,r_{d-j+1}+, \frac{r_{d-j+1}}{2}+} \cdots H^{d-1}_{y,r_{d-2}+, \frac{r_{d-2}}{2}+}	H^{d}_{y,r_{d-1}+, \frac{r_{d-1}}{2}+} \, \subset J_{j-1} $$
 	and hence acts trivially on $V_{j-1}$. Therefore  $H^{d-j+1}_{y,r_{d-j}, d'+}$ preserves $V_{j-1}$ and we can find a nonzero subspace $V'$ on which the action of  $H^{d-j+1}_{y,r_{d-j}, d'+}$ factors through $H^{d-j+1}_{y,r_{d-j}, d'+}/H^{d-j+1}_{y,r_{d-j}+}$ and is given by $\Psi \circ(X_{d-j}+C_{n_0})$ for some $C_{n_0} \in (\fh^{d-j+1})^*_{y,-(d'+\epsilon)}$ that is trivial on $\ft^j \oplus  \fh^{d-j}$.
  Moreover,
$$2d'-r_{d-j}+2\epsilon\le  d'-\min\left(\frac{r_{d-j}}{2}, 3\epsilon\right)+2\epsilon<d'<r_0.$$
 Hence, applying \cite[Cor.~5.2]{Fintzen-exhaustion} we obtain a nonzero subspace $V''$ of $V$
on which
\begin{itemize}
\item the action of $H^{d-j}_{y,r_0}$ is given by $\hat \phi$ (\hspace{1sp}\cite[Cor.~5.2(iv)]{Fintzen-exhaustion}),
\item for $1 \leq i \leq j-1$, the action of $H^{d-i+1}_{y,r_{d-i}, \frac{r_{d-i}}{2}+}$
is given by
 $$
   H^{d-i+1}_{y,r_{d-i}, \frac{r_{d-i}}{2}+} ~ \twoheadrightarrow ~ H^{d-i+1}_{y,r_{d-i}, \frac{r_{d-i}}{2}+}/H^{d-i+1}_{y,r_{d-i}+}
   ~ \simeq ~  \fh^{d-i+1}_{y,r_{d-i}, \frac{r_{d-i}}{2}+}/\fh^{d-i+1}_{y,r_{d-i}+} ~~ \stackrel{\Psi \circ X_{d-i}}{\curvearrowright} ~~ V''
   $$
(this implies that $H^{d-i+1}_{y,r_{d-i}, \frac{r_{d-i}}{2}+}$ acts via $\hat \phi$, because $\hat \phi_k$ is trivial on $H^{d-i+1}_{y,r_{d-i}, \frac{r_{d-i}}{2}+}$ if $k \neq d-i$) (\hspace{1sp}\cite[Cor.~5.2(ii)]{Fintzen-exhaustion}),
\item  the action of $H^{d-j+1}_{y,r_{d-j}, d'+}$
is given by
 $$
  H^{d-j+1}_{y,r_{d-j}, d'+} ~ \twoheadrightarrow ~ H^{d-j+1}_{y,r_{d-j}, d'+}/H^{d-j+1}_{y,r_{d-j}+}
   ~ \simeq ~ \fh^{d-j+1}_{y,r_{d-j}, d'+}/\fh^{d-j+1}_{y,r_{d-j}+} ~~ \stackrel{\Psi \circ X_{d-j}}{\curvearrowright} ~~ V''
   $$
   (\hspace{1sp}\cite[Cor.~5.2(iii)]{Fintzen-exhaustion}).
\end{itemize}
If $d'=r_{d-j}/2$, then $J_j$ acts on $V''$ via $\hat\phi$ so we are done.
Otherwise, we achieve $d'=r_{d-j}/2$ by a recursion as follows.
We choose $V''$ to be as large as possible with the above properties, then we use the same reasoning as above with $V''$ in place of $V_{j-1}$, and apply \cite[Cor.~5.2]{Fintzen-exhaustion} repeatedly (at each step replacing $V_{j-1}$ by the newly obtained subspace, with $d'$ playing the role of $d$ in \cite[Cor.~5.2]{Fintzen-exhaustion}) as the value of $d'$ goes through:
$$d'=r_{d-j}-(n_0+1)\epsilon,~ r_{d-j}-(n_0+2)\epsilon,~ \hdots,~ r_{d-j}-(n_0+N_0)\epsilon,~ \frac{r_{d-j}}{2},$$
 where $N_0$ is the largest integer for which $r_{d-j}-(n_0+N_0)\epsilon > \frac{r_{d-j}}{2}$.
  (If $N_0\le 0$, we only consider the case $d'=\frac{r_{d-j}}{2}$.)
 After the final step of recursion, we obtain a nonzero subspace of $V$ on which $J_j$ acts via $\hat \phi|_{J_j}$. This completes the inductive proof.
\end{proof}

Write $e_F$ for the absolute ramification index of $F$.
Let $m$ be an arbitrary positive integer and set
$$n:=2e_Fm-1 .$$
Assume that the 1-toral datum \eqref{eq:one-toral-datum} satisfies
\begin{equation} \label{eq:rd} n < r_0 \leq r_{d-1}=r_d \leq n+1. \end{equation}
Such a 1-toral datum always exists by Proposition \ref{prop:1-toral-abundance} under our current assumption that $G$ is tamely ramified.
To produce an omni-supercuspidal type on the group $G(F)_{y,\frac{r_d}{2}+}$, we want to know:

\begin{lemma}\label{lem:order-p^m}
 Assume that $\textup{char}(F)=0$, that $r_d$ is as in \eqref{eq:rd}, and that $G$ is tamely ramified. Then the image of
  $$\hat\phi: G(F)_{y,\frac{r_d}{2}+}/G(F)_{y,r_d+}\ra \C^*$$
  is a cyclic group of order $p^m$.
\end{lemma}

\begin{remark}
  If $\textup{char}(F)=p$, then the last isomorphism in \eqref{eq:order-pm} in the proof below breaks down, and the lemma is false for $m>1$.
\end{remark}

\begin{proof}
  Recall that we write $T=G^0$. Since $G$ is tamely ramified, we have
  \begin{equation}\label{eq:order-pm}
  \frac{T(F)_{r_d}}{T(F)_{r_d+}} \subset \frac{T(F)_{\frac{r_d}{2}+}}{T(F)_{r_d+}} \subset  \frac{G(F)_{y,\frac{r_d}{2}+}}{G(F)_{y,r_d+}}
  \simeq \frac{\fkg(F)_{y,\frac{r_d}{2}+}}{\fkg(F)_{y,r_d+}}.
  \end{equation}
  Since $r_d\le n+1=2me_F$, and hence $r_d - \frac{r_d}{2}\leq me_F$,
  we see that every element in $\fkg(F)_{y,\frac{r_d}{2}+}/\fkg(F)_{y,r_d+}$ has order dividing $p^m$.
   Therefore the image of $\hat\phi$ is contained in a cyclic subgroup of $\C^*$ of order $p^m$. (Every finite subgroup of $\C^*$ is cyclic.)

  The image of $T(F)_{r_d}/T(F)_{r_d+}$ in $\fkg(F)_{y,\frac{r_d}{2}+}/\fkg(F)_{y,r_d+}$ is a $p$-torsion subgroup contained in $\fkg(F)_{y,(r_d-e_F)+}/\fkg(F)_{y,r_d+}$.
  Now suppose that $\mathrm{im}(\hat \phi)$ is contained in a cyclic subgroup of order $p^{m-1}$.
   Viewing $\hat\phi$ as an additive character on $\fkg(F)_{y,\frac{r_d}{2}+}/\fkg(F)_{y,r_d+}$, we then have $\hat\phi(p^{m-1}X)=0$ for $X\in \fkg(F)_{y,\frac{r_d}{2}+}/\fkg(F)_{y,r_d+}$.
   As $p^{m-1}\fkg(F)_{y,\frac{r_d}{2}+}=\fkg(F)_{y,(\frac{r_d}{2}+(m-1)e_F)+}$, it follows that $\hat \phi$ is trivial on $\fkg(F)_{y,(\frac{r_d}{2}+(m-1)e_F)+}/\fkg(F)_{y,r_d+}$, thus also trivial on $T(F)_{r_d}/T(F)_{r_d+}$ (since $\frac{r_d}{2}+(m-1)e_F<r_d$). This contradicts that $\hat\phi|_{T(F)_{r_d}}=\hat\phi_{d-1}|_{T(F)_{r_{d-1}}}\neq 1$.
  We conclude that the image of $\hat\phi$ is exactly a cyclic group of order $p^m$.
\end{proof}

  The lemma allows us to factor $\hat\phi$ as
  \begin{equation}\label{eq:factoring-phi}
   G(F)_{y,\frac{r_d}{2}+}/ G(F)_{y,r_d+} \stackrel{\lambda}{\twoheadrightarrow} \Z/p^{m}\Z \stackrel{\exp}{\hra} \C^*
  \end{equation}
  for some $\lambda$ (assuming $\textup{char}(F)=0$). Here $\exp$ stands for $n\mapsto \exp(2\pi i n/p^{m})$.
  When $\lambda$ is composed with \emph{any} other nontrivial character of $\Z/p^{m}\Z$, the composite map has the form $\hat{\phi}^i$ with $i \not\equiv 0$ (mod $p^{m}$). The following lemma shows that $\hat\phi^i$ is still a supercuspidal type.

\begin{lemma}\label{lem:0-toral-sc-type}
  Assume that $\textup{char}(F)=0$,	that $p$ does not divide the order of the absolute Weyl group of $G$, and that $G$ is tamely ramified. Let $r_0$ and $r_d$ be as in \eqref{eq:rd}.
  Then $(G(F)_{y,\frac{r_d}{2}+},\lambda)$ is an omni-supercuspidal type of level $p^m$, i.e.
  $(G(F)_{y,\frac{r_d}{2}+},\hat\phi^i)$ is a supercuspidal type for all integers $i$ with $i\not\equiv 0$ (mod $p^{m}$).
\end{lemma}

\begin{proof} Since the proof is trivial for $G=T$, we assume that $G$ is not a torus.
  Recall that $\hat\phi=\prod_{0 \leq j \leq d-1} \hat \phi_j$, where $\phi_j$ is a $G^{j+1}$-generic character of $G^j(F)$ of depth $r_j$. We may assume that $0<i<p^m$, and hence
  \begin{equation}\label{eq:v(i)}
  0\le v(i)\le (m-1)e_F = \frac{n+1}{2} - e_F \le \frac{n-1}{2}.
  \end{equation}
  We claim that $\phi_j^i$ is a $G^{j+1}$-generic character of $G^j(F)$ of depth $r_j-v(i)$.
  To see this, let $X_j \in (\mathfrak{g}^j)^*_{y,-r_j}$ be a $G^{j+1}$-generic element of depth $r_j$ such that the character $\phi_j|_{G^j_{y,r_i}}$ viewed as a character of $G^{j}_{y,r_j}/G^{j}_{y,r_j+} \simeq \mathfrak{g}^{j}_{y,r_j}/\mathfrak{g}^{j}_{y,r_j+}$ is given by $\Psi \circ X_j$.
    Note that the $i$-th power map sends $G^j_{y,r_j-v(i)}$ into $G^j_{y,r_j}$ and $G^j_{y,(r_j-v(i))+}$ into $G^j_{y,r_j+}$.
    The resulting map
    $$G^j_{y,r_j-v(i)}/G^j_{y,(r_j-v(i))+} \ra G^j_{y,r_j}/G^j_{y,r_j+}$$
   corresponds to the map
   $$\mathfrak{g}^j_{y,r_j-v(i)}/\mathfrak{g}^j_{y,(r_j-v(i))+} \xrightarrow{\cdot i} \mathfrak{g}^j_{y,r_j}/\mathfrak{g}^j_{y,r_j+}$$
    induced by multiplication by $i$. (This can be seen from the binomial expansion and the fact that $r_j-v(i)>1$.)
   Hence $\phi_j^i|_{G^j_{y,r_i-v(i)}}$ factors through $G^{j}_{y,r_j-v(i)}/G^{j}_{y,r_j-v(i)+} \simeq \mathfrak{g}^{j}_{y,r_j-v(i)}/\mathfrak{g}^{j}_{y,r_j-v(i)+}$ on which it is given by $\Psi \circ (i \cdot X_j)$. The claim has been verified.

  It follows from the claim that
  $$((G^0, \hdots, G^d), (r_0-v(i), \hdots, r_{d-1}-v(i)), (\phi_0^i, \hdots, \phi_{d-1}^i))$$
   is a \onetoraldatum/.  It is implied by \eqref{eq:v(i)} that
  $$ r_0 - v(i) > n - \frac{n-1}{2} = \frac{n+1}{2} \ge \frac{r_d}{2} .$$
  Thus $G(F)_{y,r_0-v(i)}\subset G(F)_{y,\frac{r_d}{2}+}\subset G(F)_{\frac{r_d-v(i)}{2}+}$.
   Applying Proposition \ref{prop:sc-type} to the new \onetoraldatum/ above (thus the role of $r_j$ in the proposition is played by $r_j-v(i)$) and observing that $(\hat \phi)^i |_{G(F)_{y,\frac{r_d}{2}+}}=\widehat{\phi^i}|_{G(F)_{y,\frac{r_d}{2}+}}$, we deduce that $(G(F)_{y,\frac{r_d}{2}+},\hat\phi^i)$ is a supercuspidal type.
\end{proof}

  The upshot is the following.
  \begin{corollary}\label{cor:existence-omnisctypes}
  	Assume $\textup{char}(F)=0$, $p>\Cox(G)$ and $G$ is tamely ramified.
  	Then there exists a sequence
  	$\{(U_m,\lambda_m)\}_{m\ge 1}$
  	such that
  	\begin{enumerate}
  		\item each $(U_m,\lambda_m)$ is an omni-supercuspidal type of level $p^m$,
  		\item $U_1\supset U_2\supset \cdots$, and $\{U_m\}_{m\ge 1}$ forms a basis of open neighborhoods of $1$,
     \item $U_{m'}$ is normal in $U_m$ whenever $m'\ge m$.
\item If $G$ is not a torus, then $\lambda_m|_{U_m\cap Z(G)(F)}$ is trivial.
  	\end{enumerate}
  \end{corollary}

\begin{proof}[Proof of Corollary \ref{cor:existence-omnisctypes}]
	For $m \geq 1$, set $n:=2e_Fm-1$ (as above). Then by Proposition \ref{prop:1-toral-abundance} there exists a 1-toral datum $((G^0, \hdots, G^d, (r_0, \hdots, r_{d-1}), (\phi_0, \hdots, \phi_{d-1}))$ with $n < r_0 \leq r_{d-1} \leq n+1$ and $\phi_0|_{Z(G)(F)}, \hdots, \phi_{d-1}|_{Z(G)(F)}$ trivial. We may and will choose the same $G^0$ and $y$ for every $m$.
	Set $U_m:=G(F)_{y,\frac{r_d}{2}+}$ and let $\lambda_m$ be as defined in Equation \eqref{eq:factoring-phi}. Then  $(U_m,\lambda_m)$ is an omni-supercuspidal type of level $p^m$ by Lemma \ref{lem:0-toral-sc-type}. Moreover, $U_1\supset U_2\supset \cdots$, and $\{U_m\}_{m\ge 1}$ forms a basis of open neighborhoods at $1$. Property (3) follows from $U_m$ being Moy--Prasad subgroups at the same point $y$ and Property (4) follows from  $\phi_0|_{Z(G)(F)}, \hdots, \phi_{d-1}|_{Z(G)(F)}$ being trivial.
\end{proof}

 We now drop the assumption that $G$ splits over a tamely ramified extension of $F$, i.e. $G$ is an arbitrary (connected) reductive group defined over $F$.

\begin{theorem}\label{thm:existence-omnisctypes}
  Assume $\textup{char}(F)=0$ and $p>\Cox(G)$.
  Then there exists a sequence
   $\{(U_m,\lambda_m)\}_{m\ge 1}$
such that
   \begin{enumerate}
   \item each $(U_m,\lambda_m)$ is an omni-supercuspidal type of level $p^m$,
   \item $U_1\supset U_2\supset \cdots$, and $\{U_m\}_{m\ge 1}$ forms a basis of open neighborhoods of $1$,
   \item $U_{m'}$ is normal in $U_m$ whenever $m'\ge m$,
\item If $G$ is not a torus, then $\lambda_m|_{U_m\cap Z(G)(F)}$ is trivial.
   \end{enumerate}
\end{theorem}

\begin{remark}
	Condition (4) is imposed to ensure that some global assertions in Section \ref{sec:congruence} hold when the reductive group over a totally real field is slightly more general than those considered in \cite{Gro99} (see Proposition 1.4 therein), see Remark \ref{rem:Hecke-quotient}. 
\end{remark} 

\begin{remark}\label{rem:existence-omnisctypes}
	In contrast to the proof of Corollary \ref{cor:existence-omnisctypes}, the following proof only relies on the omni-supercuspidal types constructed from 0-toral data and not the more general omni-supercuspidal types constructed from 1-toral data. We have nevertheless included the construction of supercuspidal types from 1-toral data in our discussion above as this construction might be useful for other applications.
	
	Thus we are presenting the reader with two approaches to prove the theorem. One is to prove the preceding results only in the simpler case of 0-toral data, and then deduce the theorem via multiple reduction steps. The other is to establish the intermediate results in the generality of 1-toral data. Then the final theorem (Corollary \ref{cor:existence-omnisctypes}) is immediate, as far as $G$ is further assumed to be tamely ramified.
\end{remark}

\begin{proof}[Proof of Theorem \ref{thm:existence-omnisctypes}]
	It is enough to find a sequence  $\{(U_m,\lambda_m)\}_{m\ge m_0}$ as in the theorem statement for some fixed $m_0$, since we can decrease level of an omni-supercuspidal type without changing the open compact subgroup as explained in Remark \ref{rem:level-lowering}.

The basic idea of proof  is to reduce to the case where $G$ is either a torus or an (absolutely) simple adjoint group, and then handle the two base cases. Below we use $\psi_m$ to denote an arbitrary nontrivial character of $\Z/p^m\Z$.

\medskip

\noindent\textbf{Step 1. Proof when $G$ is a torus}.

Let $G=T$ be a torus over $F$ (possibly wildly ramified). We consider the filtration subgroup $T_{0+}$, which is a pro-$p$ group. Since $T(F)$ is dense in $T$ (see e.g. \cite[III.8.13]{Borel}), the group $T(F)$ contains infinitely many elements. If $T$ is anisotropic, then the compact group $T_{0+}$ has finite index in the compact group $T(F)$, and hence $T_{0+}$ has infinitely many elements. If $T$ is not anisotropic, then $T$ contains a copy of the multiplicative group $\bG_m$ and therefore $T_{0+}$ contains infinitely many elements as well. (An alternative way to see that $T(F)$ has a pro-$p$ subgroup with infinitely many elements is by considering the exponential map from a small neighborhood of $0$ in $\Lie T(F)$ to $T(F)$.)

Take $U_1=T_{0+}$ and $U_{i+1}=\{ u^p: u\in U_i\}$ recursively. Then $\{U_i\}_{i\ge 1}$ satisfies (2).
For $1\le i<j$, as a finite abelian $p$-group, $U_i/U_j$ is a product of cyclic groups of $p$-power orders.
We claim that at least one of the cyclic groups has exact order $p^{j-i}$. Suppose not, then we would have $U_{j-1}=U_j$, which would in turn imply $U_j=U_{j+1}=\cdots$. Since $U_1$ is pro-$p$, we deduced that $U_j=U_{j+1}=\cdots=\{1\}$, which contradicts the infinitude of $U_1$. Therefore there exists a projection $\lambda_m: U_m/U_{2m}\twoheadrightarrow \Z/p^m\Z$ for $m\ge 1$. Then condition (1) obviously holds for $(U_m,\lambda_m)$.

\medskip

\noindent\textbf{Step 2. Proof when $G$ is a simple simply connected group.}

In this case $G=\Res_{F'/F}G'$ for a finite extension $F'/F$ and an absolutely simple simply connected group $G'$ over $F'$. Since $G(F)=G'(F')$, we reduce to the case when $G$ is absolutely simple.
  By Proposition \ref{prop:0-toral-abundance}, there is a tamely ramified elliptic maximal torus $T\subset G$ along with
  a sequence of 0-toral data $\{(T,r_m,\phi_m)\}_{m\ge 1}$ with $2e_Fm-1<r_m\le 2e_Fm$ for all $m$. Fix $y$ as in the paragraph below Definition \ref{def:0-toral}.
  For each $m$, set $U_m:=G(F)_{y,r_m/2+}$ and let $\lambda_m:U_m\twoheadrightarrow \Z/p^m\Z$ be defined by Equation \eqref{eq:factoring-phi}. Then Condition (1) holds thanks to Lemma \ref{lem:0-toral-sc-type}, and Condition (2) is obviously satisfied.
Condition (3) follows because the groups are Moy--Prasad filtration subgroups for the same point $y$. Since $Z(G)(F)$ is finite, there exists an integer $m_0$ such that $U_{m_0} \cap Z(G)(F) = \{1\}$, hence (4) holds for $\{(U_m,\lambda_m)\}_{m\ge m_0}$.

  \medskip

\noindent\textbf{Step 3. Proof when $G$ is a simply connected group.}

Then $G=G_1\times \cdots \times G_N$, where $G_i$ are simple simply connected groups. By Step 2, we can find a sequence of omni-supercuspidal types $\{(U_{i,m},\lambda_{i,m})\}_{m\ge 1}$ satisfying (1), (2), (3) and (4) for each $G_i$, for all $i$.
Take $U_m:=U_{1,m}\times \cdots \times U_{N,m}$ for each $m\ge 1$ and define $\lambda_m:U_m\twoheadrightarrow \Z/p^m\Z$ by $\lambda_m(u_1,...,u_m)=\sum_i \lambda_{i,m}(u_i)$. Clearly (2), (3) {and (4)} hold for $\{(U_m,\lambda_m)\}$.
To verify (1), suppose that $\pi$ is an irreducible smooth representation of $G(F)$ containing $\psi_m\circ\lambda_m$ as a $U_m$-subrepresentation for some arbitrary nontrivial character $\psi_m$ of $\Z/p^m\Z$.
By \cite[Thm.~1]{Fla79}, $\pi\simeq \otimes_{i=1}^N \pi_i$ for irreducible, smooth representations $\pi_i$ of $G_i(F)$. Since $\lambda_m$ pulls back to $\lambda_i$ along the natural inclusion $G_i\hra G$, we see that $\pi_i|_{U_{i,m}}$ contains $\psi_m\circ\lambda_{i,m}$. By Step 2, $\pi_i$ is supercuspidal for every $1\le i\le N$. We conclude that $\pi$ is also supercuspidal.

\medskip

\noindent\textbf{Step 4. Proof of the general case.}

If $G$ is a torus, we are done by Step 1, so we assume that $G$ is not a torus for the remainder of the proof.
Let $Z^0$ denote the maximal torus of $Z$, and $G_{\textup{sc}}$ the simply connected cover of the derived subgroup of $G$. The multiplication map $f: Z^0 \times G_{\textup{sc}}\ra G$ has finite kernel and cokernel (either as algebraic groups or as topological groups of $F$-points). This implies that $f$ induces an isomorphism from a small enough open subgroup in $Z^0(F)\times G_{\textup{sc}}(F)$ onto an open subgroup of $G(F)$.
By Steps 1 and 3, we have omni-supercuspidal types $(U'_m,\lambda'_m)$ for $Z^0(F)$ and $(U''_m,\lambda''_m)$ for $G_{\textup{sc}}(F)$ as in the theorem, for all $m\ge 1$.
There exists $m_0$ such that $\ker(f)\cap (U'_{m'}\times U''_{m''})=\{1\}$  for all $m',m''\ge m_0$. Take $U_m:=f(U'_m\times U''_m)$ and define $\lambda_m:U_m\twoheadrightarrow \Z/p^m\Z$ by $\lambda_m(f(u', u''))=\lambda''(u'')$ for $u'\in U'_m$, $u''\in U''_m$ and $m\ge m_0$. This is well-defined.

Condition (2) is satisfied by $(U_m,\lambda_m)$ as just defined, since $\{U'_m\times U''_m\}_{\ge 1}$ forms a basis of open neighborhoods of $1$ in $Z^0(F)\times G_{\textup{sc}}(F)$.
Conditions (3) and (4) are obviously true by construction.

 To check (1), suppose that an irreducible smooth representation $\pi$ of $G(F)$ contains $\psi_m\circ\lambda_m$ as a $U_m$-representation. Let $N$ be the unipotent radical of an $F$-rational proper parabolic subgroup of $G$. We need to show that $\pi_{N(F)}=0$.

By \cite[Lem.~6.2]{Xu-lifting}, $\pi|_{{f(Z^0(G(F))\times G_{\textup{sc}}(F))}}$ decomposes as a finite direct sum $\oplus_{i=1}^n \pi_i$ of irreducible $f(Z^0(G(F))\times G_{\textup{sc}}(F))$-representations. Without loss of generality, we may assume that $\pi_1|_{U_m}$ contains $\psi_m\circ\lambda_m$. For $i>1$, there exists $g_i\in G(F)$ such that $\pi_i\simeq {^{g_i^{-1}}\pi_1}$, where $^{g_i^{-1}}\pi_1(\gamma)=\pi_1(g_i \gamma g_i^{-1})$ for all $\gamma\in f(Z^0(F)\times G_{\textup{sc}}(F))$ (e.g. see the proof of Lemma 6.2 and Corollary 6.3 in \cite{Xu-lifting}).
For $g\in G(F)$, let $N'_g$ be the unipotent radical of a proper $F$-rational parabolic subgroup of $G_{\textup{sc}}$ such that  $f$ induces an isomorphism from $1 \times N'_g$ onto $g^{-1}N g$ (and also on the $F$-points).  By Step 3, the irreducible $G_{\textup{sc}}(F)$-representation ${\pi_1}|_{f(G_{\textup{sc}}(F))}\circ f$ is supercuspidal and so $(\pi_1 \circ f)_{N'_g(F)}=0$.
Thus we obtain as vector spaces
$$\pi_{N(F)}=\bigoplus_{i=1}^n (\pi_i)_{N(F)}
=\bigoplus_{i=1}^n (\pi_1)_{g_i N(F) g_i^{-1}} = \bigoplus_{i=1}^n (\pi_1 \circ f)_{N'_{g_i}(F)} = 0.$$
\end{proof}

\section{Congruence to automorphic forms with supercuspidal components}\label{sec:congruence}

  Now we switch to a global setup for algebraic automorphic forms as studied in \cite{Gro99}. The following notation will be used.

\begin{itemize}
\item $G$ is a reductive group over a totally real field $F$ such that $G(F\otimes_\Q \R)$ is compact modulo center. Since the statements in this section are trivial to show if $G$ is a torus, because all smooth irreducible representations of $p$-adic tori are supercuspidal, we assume through that $G$ is not a torus.
\item $v$ is a place of $F$ above $p$; $\fkp_v$ is the maximal ideal of the ring of integers $\cO_{F_v}$ in the completion $F_v$ of $F$ with respect to $v$.
\item $e_v:=e(F_v/\Q_p)$ is the absolute ramification index of $F_v$ (so $\fkp_v^{e_v}=(p)$ as ideals in $\cO_{F_v}$).
\item $G_v:=G\times_F F_v$, $G_p:=(\Res_{F/\Q} G)\times_\Q \Q_p$, $G_\infty:=(\Res_{F/\Q} G)\times_\Q\R$.
\item $S$ is a finite set of places of $F$ containing all $p$-adic and infinite places as well as all ramified places for $G$; we fix a reductive model for $G$ over $\Spec \cO_F\backslash \{\mbox{finite~places~in}~S\}$ (still denoted by $G$ for convenience), and we use this model to identify $G(\bA_F)=\prod'_v G(F_v)$ as a restricted product.
\item  $U^p:=\prod_{w\nmid p\infty} U_w$ is a (fixed) compact open subgroup of $G(\bA_F^{\infty, p})$ such that $U_w=G(\cO_{F_w})$ hyperspecial for $w \notin S$, and we write $U^{S}:=\prod_{w\notin S} U_w$.
\end{itemize}

If $U_p$
is a compact open subgroup of $G_p(\bQ_p)=\prod_{w|p} G(F_w)$, and $\Lambda$ a finitely generated $\bZ_p$-module with a continuous action of $U_p$, then we write $U:=U^pU_p$, and
$$M(U,\Lambda):= \left\{
\begin{array}{c}
  \mbox{cont.~functions}~ f:G(F)\backslash G(\A_F)/ U^p G_\infty(\R)^\circ\ra \Lambda,\vspace{.02in}\\
  \mbox{such~that}~~ f(gu_p)=u_p^{-1} f(g),~~\mbox{for}~g\in G(\A_F), ~~u_p\in U_p
\end{array}
  \right\}.$$
  Here $G_\infty(\R)^\circ$ denotes the connected component of $G_\infty(\R)$ that contains the identity. We write $\bZ_p^N$ for the free $\bZ_p$-module of rank $N$ with trivial $U_p$-action, and we drop the exponent if $N=1$. Hence $M(U, \bZ_p)$ is the space of $\bZ_p$-valued functions on the finite set $G(F)\backslash G(\A_F)/ U^pU_p G_\infty(\R)^\circ$. Note that $M(U, \Lambda)$ is a $\mathbb T^S:=\Z[U^S\backslash G(\A_F^S)/U^S]$-module under the usual double coset action. It is routine to check that
  the association of $\mathbb T^S$-modules
  $$ \Lambda \mapsto M(U,\Lambda)$$
  is functorial in $\Lambda$. In particular, if $\Lambda'\simeq \Lambda^{\oplus N}$ as $\Z_p[U_p]$-modules with $N\in \Z_{\ge 1}$, then
  $M(U,\Lambda')\simeq M(U,\Lambda^{\oplus N}) = M(U,\Lambda)^{\oplus N}$ as $\mathbb T^S$-modules, and $\mathbb T^S$ acts on the last direct sum by the diagonal action.
   We define
\begin{equation} \label{eq:TS} \mathbb T^S(U,\Lambda)\subset \End_{\bZ_p}(M(U,\Lambda))
\end{equation}
to be the $\bZ_p$-subalgebra generated by the image of $\bT^S$. (Thus $\mathbb T^S(U,\Lambda)$ is commutative.)
If $\Lambda'\simeq \Lambda^{\oplus N}$ then we have an induced isomorphism $\mathbb T^S(U,\Lambda')\simeq \mathbb T^S(U,\Lambda)$, acting equivariantly on
$M(U,\Lambda')\simeq M(U,\Lambda)^{\oplus N}$.
This observation will be used a few times in \S\ref{subsec:constant-single} and \S\ref{subsec:non-constant}.

We have an obvious action of $\pi_0(G_\infty(\R))$ by right translation on the spaces of automorphic forms considered here, cf.~Proposition 8.6 and the paragraph below (4.1) in \cite{Gro99}, commuting with the Hecke algebra actions. Every isomorphism between spaces of automorphic forms below is compatible with the $\pi_0(G_\infty(\R))$-action. That said, we will not mention $\pi_0(G_\infty(\R))$-actions again.

\subsection{Constant coefficients}\label{subsec:constant-single}
Our goal is to define congruences between arbitrary automorphic forms with constant coefficients and automorphic forms that are supercuspidal at $p$. In order to define the coefficients of the latter space, we denote by $A_m$ the $\bZ_p$-algebra $\bZ_p[T]/(1+T+\hdots+T^{p^m-1})$ for $m$ a positive integer.
Then we have a canonical $\Z_p$-algebra isomorphism $A_m/(T-1)\simeq \Z_p/(p^m)$.

\begin{theorem}\label{prop:congruence} Assume $p> \Cox(G)$ and let $N$ be a positive integer.
	 Then there exist
\begin{itemize}
\item	  a basis of compact open neighborhoods $\{U_{p, m}\}_{m \geq 1}$ of $1 \in G_p(\bQ_p)=\prod_{w|p} G(F_w)$ such that $U_{p,m'}$ is normal in $U_{p,m}$ whenever $m'\ge m$ and
\item a smooth character $\psi_m: U_{p,m} \ra A_m^\times$ for each $m\geq 1$
\end{itemize}
 such that we have isomorphisms of $\bZ_p/(p^m)$-modules (where the $U_{p,m}$-action in $M(\,\cdot\,)$ is trivial on the left hand side and through $\psi_m$ on the right hand side)
	\begin{equation} \label{eq:space-mod-pm}
	M(U^pU_{p,m}, \bZ_p^N/(p^m)) \simeq (M(U^pU_{p,m}, A_m/(T-1)))^{\oplus N}
	\end{equation}
	that are compatible with the action of $\T^S(U^p U_{p,m},\bZ_p^N/(p^m))$ on the left hand side and the diagonal action of $\T^S(U^p U_{p,m},A_m/(T-1))$ on the right hand side via the
	$\Z_p$-algebra isomorphism
	\begin{equation} \label{eq:Hecke-mod-pm} \T^S(U^p U_{p,m},\Z_p^N/(p^m)) \simeq \T^S(U^p U_{p,m},A_m/(T-1)). \end{equation}
	Moreover, every automorphic representation of $G(\bA_F)$ that contributes to
	$$(M(U^pU_{p,m}, A_m))^{\oplus N} \otimes_{\bZ_p} \overline\bQ_p$$  is supercuspidal at all places above $p$.
\end{theorem}
\begin{remark}
	Theorem \ref{prop:congruence}  also holds if we replace $U^p$ by a compact open subgroup $U^v=\prod_{w\nmid v\infty} U_w$ of $G(\bA_F^{\infty, v})$ such that $U_w$ is hyperspecial for $w \notin S$, and $U_{p,m}$ by compact open subgroups $U_{v,m}$ of $G(F_v)$ for a place $v$ above $p$. The conclusion in this case includes that every automorphic representation $G(\bA_F)$ that contributes to 	$(M(U^vU_{v,m}, A_m))^{\oplus N} \otimes_{\bZ_p} \overline\bQ_p$  is supercuspidal at $v$. The proof works in a completely analogous way.
\end{remark}

\begin{remark}\label{rem:use-appD-1}
	The assumption that $p>\Cox(G)$ in Theorem \ref{prop:congruence} and in all other results in Section \ref{sec:congruence} can be removed by appealing to Appendix D, which proves Theorem \ref{thm:existence-omnisctypes} for all primes $p$.
\end{remark}

  Before presenting a proof, let us comment on the meaning of the proposition.
  As $m$ grows to infinity, the space $M(U^p U_{p,m},\Q_p) $, or more precisely its extension of scalars to $\ol\Q_p$, exhausts all automorphic forms on $G(\A_F)$ with constant coefficients, if we also allow $U^p$ to shrink arbitrarily.
  Thus the left hand side of \eqref{eq:space-mod-pm} represents arbitrary automorphic forms on $G$ with $\bZ_p/(p^m)$-coefficients. Loosely speaking, Theorem \ref{prop:congruence}  can be thought of as a congruence modulo a power of $p$ between arbitrary automorphic forms and those with supercuspidal components at $p$.

  A caveat is that the congruence here is between spaces of automorphic forms.
   It does not follow from our result that for an individual automorphic representation $\pi$, there exists $\pi(m)$ which is supercuspidal at $p$ such that ``$\pi\equiv \pi(m)$ mod $p^m$'' in terms of Hecke eigenvalues outside $S$. To see this, let $c_\pi:\T^S(U^p U_{p,m},\Z_p)\ra \ol \Z_p$ be the $\Z_p$-algebra morphism accounting for $\pi$. Suppose that $\T^S(U^p U_{p,m},\Z_p)/(p^m)$ is isomorphic to $\T^S(U^S U_{v,m},A_m)/(T-1)$ (which we do not know; see Remark \ref{rem:Hecke-quotient}). Taking $c_\pi$ mod $p^m$, we obtain a morphism $\T^S(U^S U_{v,m},A_m)/(T-1)\ra \ol \Z_p/(p^m)$, but now the problem is that it is unclear whether the latter lifts to a morphism $\T^S(U^S U_{v,m},A_m)\ra \ol \Z_p$. (When $m=1$, this is often possible by the Deligne--Serre lifting lemma \cite[Lemma~6.11]{Deligne-Serre}, for instance.) However, we often do not need such a lift for applications, see e.g. \S\ref{subsec:app} below and \cite[\S7]{Scholze-LT}. For these applications a statement like Theorem \ref{prop:congruence}  suffices.

\begin{proof}[Proof of Theorem \ref{prop:congruence} ]
 We let $\{(U_{p,m},\lambda_{m})\}_{m\geq 1}$ be omni-supercuspidal types of level $p^m$ for $G_p$ as in Theorem \ref{thm:existence-omnisctypes}, i.e. such that $U_{p,1}\supset U_{p,2}\supset \hdots$, the groups $\{U_{p,m}\}_{m\ge 1}$ form a basis of open neighborhoods of $1$ and $U_{p,m'}$ is normal in $U_{p,m}$ whenever $m'\ge m$.

  For each $m \geq 1$, we have the following commutative diagram of maps, where $\zeta_{p^m}$ denotes a primitive $p^m$-th root of unity in $\overline\bQ_p$.

 $$\xymatrix{
 	& & A_m\otimes_{\Z_p} \ol\Q_p \simeq \ol\Q_p^{p^m-1} & & & \{\zeta_{p^m}^{ai}\}_{i=1}^{p^m-1}\\
 	\Z/p^m \Z \ar@{^(->}[r] & A_m^\times \ar@{^(->}[r]  \ar@{->>}[d] & A_m \ar@{->>}[d]^-{\textrm{mod}~T-1} \ar@{^(->}[u] & a~\textrm{mod}~p^m \ar@{|->}[r] & T^a \ar@{|->}[d] \ar@{|->}[r] & T^a \ar@{|->}[d] \ar@{|->}[u]\\
 	& (\Z/p^m\Z)^\times \ar@{^(->}[r]   &  \Z/p^m \Z & & 1  \ar@{|->}[r] & 1
 }$$

 We define the smooth character $\psi_{m}:U_{p,m}\ra A_m^\times$ as the composite
 $$
 U_{p,m} \xrightarrow{\lambda_{m}}  ~ \Z/p^m\Z~  \hookrightarrow ~A_m^\times,
 $$
 and we let $u \in U_{p,m}$ act on $A_m$ by multiplication by $\psi_{m}(u)$.
  Since the resulting action of $U_{p,m}$ on  $A_m/(T-1)$ is trivial, we have canonical isomorphisms
 \begin{equation}\label{eq:Zp-Am-isom}
 M(U^p U_{p,m},\Z_p^N/(p^m)) \simeq  M(U^p U_{p,m},\Z_p/(p^m))^{\oplus N} \simeq  \left(M(U^p U_{p,m},A_m/(T-1))\right)^{\oplus N}
 \end{equation}
 as modules over $\Z_p/(p^m) \simeq A_m/(T-1)$. Moreover, these isomorphisms are $\bT^S$-equivariant, and observing that the action of $\bT^S$ on $\left(M(U^p U_{p,m},A_m/(T-1))\right)^{\oplus N} $ is given via the diagonal action of $\T^S(U^p U_{p,m},A_m/(T-1))$, we obtain that the isomorphisms \eqref{eq:Zp-Am-isom} are compatible with the action of
 $$\T^S(U^p U_{p,m},\Z_p^N/(p^m)) \simeq \T^S(U^p U_{p,m},A_m/(T-1)).$$

 Note that

 \begin{equation}\label{eq:M(psi-m)[1/p]}
 M(U^p U_{p,m},A_m)\otimes_{\Z_p} \ol\Q_p\simeq \bigoplus_{\chi:\Z/p^m\Z \ra \ol\Q_p^\times\atop \chi\neq 1} M(U^p U_{p,m},(\ov \bQ_p)_{\chi\circ \lambda_{m}}),
 \end{equation}
 where $(\ov \bQ_p)_{\chi\circ \lambda_{m}}$ denotes the free rank-1 $\ov \bQ_p$-module on which $u \in U_{p,m}$ acts by multiplication by $\chi\circ \lambda_{m}$. Since $(U_{p,m}, \lambda_{m})$ is omni-supercuspidal, every automorphic representation of $G(\bA_F)$ that contributes to
$(M(U^pU_{p,m}, A_m))^{\oplus N} \otimes_{\bZ_p} \overline\bQ_p$  is supercuspidal at all places above~$p$.

\end{proof}

\begin{remark}\label{rem:Hecke-quotient}
In the setup of Theorem \ref{prop:congruence}, we have $\T^S$-equivariant isomorphisms
\begin{equation} \label{eq:M-quotient}
M(U^pU_{p,m}, \bZ_p^N)/(p^m) \simeq M(U^pU_{p,m}, \bZ_p^N/(p^m)) \stackrel{\eqref{eq:space-mod-pm}}{\simeq} (M(U^pU_{p,m}, A_m/(T-1)))^{\oplus N} . 
\end{equation}
We claim that the natural map (induced by the map $A_m\ra A_m/(T-1)$ of coefficient modules)
$M(U^pU_{p,m}, A_m)/(T-1) \ra M(U^pU_{p,m}, A_m/(T-1))$
is a $\T^S$-equivariantly isomorphism for sufficiently small $U_{p,m}$. (Unlike the first isomorphism in \eqref{eq:M-quotient}, the isomorphicity is not obvious since the $U_{p,m}$-action on $A_m$ is not trivial.)
If the center of $\Res_{F/\Q}G$ has the same $\Q$-rank and $\R$-rank (that is, if $G$ satisfies the equivalent conditions of \cite[Proposition 1.4]{Gro99}; in particular every arithmetic subgroup of $G(F)$ is finite) then the claim is proved by the argument of \cite[Proposition 9.2]{Gro99}, with $A_m$, $A_m/(T-1)$, and $U_{p,m}$ playing the roles of $L_p$, $\ol L$, and $K_p$ there, respectively; the point is that the arithmetic subgroup $\Delta_\alpha$ in that argument becomes trivial when $K_p=U_{p,m}$ is small enough. Now, without the assumption on the center, we can still mimic the argument since $U_{p,m}\cap Z(G)(F)$ acts trivially on $A_m$ by condition (4) of Theorem \ref{thm:existence-omnisctypes}. This means that $U_{p,m}$ acts on $A_m$ through its image under the adjoint map $G(F\otimes_\Q \Q_p)\ra G^{\textup{ad}}(F\otimes \Q_p)$. Since $G^{\textup{ad}}$ satisfies the conditions of \cite[Proposition 1.4]{Gro99}, if $U_{p,m}$ is small enough, then all $\Delta_\alpha$ act trivially on $A_m$ in the the argument of \cite[Proposition 9.2]{Gro99}, even if $\Delta_\alpha$ need not be a finite group, so the proof there still goes through to establish the claim. Now that the claim is true, \eqref{eq:M-quotient} yields a natural $\T^S$-equivariant isomorphism (cf.~\eqref{eq:space-mod-pm})
$$M(U^pU_{p,m}, \bZ_p^N)/(p^m) \simeq (M(U^pU_{p,m}, A_m)/(T-1))^{\oplus N} . $$
 However we cannot take the quotients outside the Hecke algebras in \eqref{eq:Hecke-mod-pm}. The abstract situation is as follows. Let $M$ be a finite free $\Z_p$-module, $\alpha\in \End_{\Z_p}(M)$. Write $\ol M:=M/(p^m)$ and $\ol\alpha\in \End_{\Z_p}(\ol M)$ for the image of $\alpha$. Put $T:=\Z_p[\alpha]$ and $\ol T:=\Z/(p^m)[\ol\alpha]$ for the subalgebra of $\End_{\Z_p}(M)$ (resp. $\End_{\Z_p}(\ol M)$) generated over $\Z_p$. Then the obvious map $T/(p^m)\ra \ol T$ need not be an isomorphism: consider arbitrary $\alpha$ such that $\ol\alpha$ is the multiplicative unity.
 Despite the apparent defect, it is readily checked that
 \begin{equation}\label{eq:T-as-inv-lim}
\T^S(U^p U_{p,m_0},\Z_p^N)=\varprojlim\limits_{m\ge m_0}\T^S(U^p U_{p,m_0},\Z_p^N/(p^m))
\end{equation}
 for each integer $m_0\ge 1$, with compatible actions on the spaces of functions with coefficients in $\Z_p^N$ and $\Z_p^N/(p^m)$, respectively. If the center of $\Res_{F/\Q}G$ has the same $\Q$-rank and $\R$-rank, then the analogous statements also hold with non-constant coefficient $V_{\bZ_p}$ that we will introduce in \S\ref{subsec:non-constant} below.\footnote{Without the assumption on the center, the argument for the preceding claim still applies if $U_{p,m}\cap Z(G)(F)$ acts trivially on $V_{\bZ_p}$.}
 More precisely, we have $M(U^pU_{p,m}, V_{\bZ_p}^N)/(p^m) \simeq M(U^pU_{p,m}, V_{\bZ_p}^N/(p^m))$  as $\T^S$-modules by the same reasoning as for the claim above for $U_{p,m}$ sufficiently small. This leads to the non-constant coefficient analogue of \eqref{eq:T-as-inv-lim}. For applications, \eqref{eq:T-as-inv-lim} and its non-constant analogue are often enough, cf.~proof of Theorem \ref{thm:main-app} below (where the center of $\Res_{F/\Q}G$ has $\Q$-rank and $\R$-rank $0$).
\end{remark}

\subsection{Non-constant coefficient}\label{subsec:non-constant}

Let $L$ be a finite extension of $F$ such that $L/\Q$ is Galois and $G\times_{F} L$ is split.
Then $(\Res_{F/\Q} G)\times_{\Q} L$ is split and we denote by $\tilde {\mathbf G}$ a split reductive group  over $\cO_L$
such that $\tilde{\mathbf G}_L \simeq (\Res_{F/\Q} G)\times_{\Q} L$. Let $V$ be a linear algebraic representation of $\tilde {\mathbf G}$ over $\cO_L$.
This yields an algebraic representation of $\Res_{\cO_L/\Z} \tilde {\mathbf G}$
on $V_{\Z}$ (i.e. $V$ viewed as a $\bZ$-module) over $\Z$.
In particular, {$(\Res_{\cO_L/\Z}\tilde {\mathbf G})(\Z_p)$} acts continuously on $V_{\Z_p}=V_\bZ \otimes_\Z \Z_p$ for the $p$-adic topology. Observe that
$$(\Res_{\cO_L/\Z}\tilde {\mathbf G})(\Q_p)=\tilde{ \mathbf G}(L\otimes_\Q \Q_p) \simeq (\Res_{F/\Q} G)(L\otimes_\Q \Q_p)
= G_p(L\otimes_\Q \Q_p),$$
and the latter naturally contains $G_p(\Q_p)$ as a closed subgroup.
Therefore every sufficiently small open subgroup $U_p$ of $G_p(\Q_p)$ (as long as it maps into $(\Res_{\cO_L/\Z}\tilde {\mathbf G})(\Z_p)$ under the above isomorphism) acts on $V_{\Z_p}$. In that case, using this action, we obtain a space of automorphic forms $ M(U^pU_{p}, V_{\bZ_p}/(p^m)) $ as defined earlier.

\begin{theorem}\label{prop:congruenceV} Assume $p> \Cox(G)$.
	Then there exists a basis of compact open neighborhoods $\{U_{p, m}\}_{m \geq 1}$ of $1 \in G_p(\bQ_p)=\prod_{w|p} G(F_w)$ with $U_{p, m'}$ normal in $U_{p, m}$ for $m'\ge m$ as well as smooth actions of $U_{p,m}$ on $A_m$ arising from multiplication by a character $\psi_m: U_{p,m} \ra A_m^\times$ for $m \geq 1$ (factoring through $\Z/p^m \Z$), such that we have isomorphisms of $\bZ_p/(p^m)$-modules
	\begin{equation} \label{eq:M-non-const} M(U^pU_{p,m}, V_{\bZ_p}/(p^m)) \simeq (M(U^pU_{p,m}, A_m/(T-1)))^{\oplus \dim_{\bZ_p} V}  \end{equation}
	that are compatible with the action of $\T^S(U^p U_{p,m},V/(p^m))$ on the left hand side and the diagonal action of $\T^S(U^p U_{p,m},A_m/(T-1))$ on the right hand side via the
	$\Z_p$-algebra isomorphism
	\begin{equation} \label{eq:Hecke-mod-pm-2} \T^S(U^p U_{p,m},V_{\bZ_p}/(p^m)) \simeq \T^S(U^p U_{p,m},A_m/(T-1)). \end{equation}
	Moreover, every automorphic representation of $G(\bA_F)$ that contributes to
	$$(M(U^pU_{p,m}, A_m))^{\oplus \dim_{\bZ_p} V_{\bZ_p}} \otimes_{\bZ_p} \overline\bQ_p$$  is supercuspidal at all places above $p$.
\end{theorem}

\begin{remark}\label{rem:model-over-L}
	If $V'_{\C}$ is an algebraic representation of $(\Res_{F/\Q} G)\times_{\Q}\C$ over $\C$,
 then we can choose a model of the representation $V'_\mathbb{C}$ over $\cO_L$ (e.g. the sum of the Weyl modules as in \cite[p.~183]{Jantzen} corresponding to the irreducible components of $V'_\mathbb{C}$), i.e. a representation $V$ of $\tilde {\mathbf G}$ over $\cO_L$ whose base change to $\mathbb{C}$ is the representation $V'_\mathbb{C}$  of $(\Res_{F/\Q} G)\times_{\Q} \mathbb{C}$. This way we obtain the setup needed at the start of this subsection.
\end{remark}

\begin{remark} \label{rem:different-incarnations}
 The space $M(U^pU_{p,m}, V_{\ol\Q_p})$, where $V_{\ol\Q_p}=V_{\bZ_p}\otimes_{\bZ_p}{\ol\Q_p}$, can be described in terms of classical automorphic forms on $G$.
  Here we fix field embeddings $\ol\Q\hra \ol\Q_p$ and $\ol\Q\hra \C$ and a compatible isomorphism $\iota:\ol\Q_p\simeq \C$, and extend scalars using these embeddings to define $V_{\ol\Q_p}$ and $V_\C$. Thus $V_\C$ is an algebraic representation of $(\Res_{L/\Q} \tilde{\mathbf G}_L)\times_\Q \C = (\Res_{F/\Q} G)\times_\Q (L\otimes_\Q \C)$, which we can restrict to a representation of $(\Res_{F/\Q} G)\times_\Q \C$ via the obvious embedding $\C\hra L\otimes_\Q \C$. The resulting representation can be viewed either as an algebraic representation of $(\Res_{F/\Q} G)\times_\Q \C$,
  or as a continuous representation of $G_\infty(\C)\supset G_\infty(\bR)$.
  For simplicity, suppose that the center of $G$ is anisotropic over $F$ and and that $G_\infty(\bR)$ is connected (in addition to being compact modulo center). Write $\mathcal A_G$ for the space of $L^2$-automorphic forms on $G(F)\backslash G(\A_F)$ (as in 2.1.2 of \cite{Sor13} but without the need to fix a central character; note that his $G$ is our $\Res_{F/\Q}G$). Then Lemma 2 of \emph{loc.~cit.}~gives a $\T^S_\C$-equivariant isomorphism
  $$ \iota M(U^pU_{p,m}, V_{\ol\Q_p}) \simeq \Hom_{G_\infty(\bR)}(V^\vee_\C,\mathcal A_G)^{U^p U_{p,m}}.$$

\end{remark}

\begin{proof}[Proof of Theorem \ref{prop:congruenceV}]
	Recall that an open subgroup of $G_p(\Q_p)$ acts continuously on $V_{\Z_p}$ for the $p$-adic topology.	Hence, for every integer $m\ge 1$, there exists an open subgroup of $G_p(\Q_p)$ that acts trivially on (the finite set) $V_{\Z_p}/p^m V_{\Z_p}$.	
	
	Now let $\{(U_{p,n},\lambda_n)\}_{n\ge 1}$ be a sequence of omni-supercuspidal types for $G_p(\Q_p)=G(F\otimes_\Q \Q_p)$ as in Theorem \ref{thm:existence-omnisctypes}. By the preceding paragraph, there is an increasing sequence $n_1<n_2<\cdots$ such that $U_{p,n_m}$ stabilizes $V_{\Z_p}$ and acts trivially on $V_{\Z_p}/p^m V_{\Z_p}$ for every $m$.
	Let $\mathrm{pr}_{i,j}:\Z/p^i \Z\twoheadrightarrow \Z/p^j\Z$ denote the canonical surjection when $i\ge j\ge 1$. Then
	$$(U'_{p,m},\lambda'_{p,m}):=(U_{p,n_m},\mathrm{pr}_{n_m,m}\circ \lambda_{n_m}), \qquad m\ge 1,$$
	is an omni-supercuspidal type of level $p^m$. Moreover, by construction, we have
	$$M(U^p U'_{p,m},V_{\Z_p}/(p^m)) \simeq  M(U^p U'_{p,m},\Z_p/(p^m) )^{\oplus \dim_{\Z_p} V_{\bZ_p}},$$
	where the action of $U'_{p,m}$ is trivial on the right hand side and induced by that on $V_{\Z_p}$ on the left hand side.
	Now we can proceed as in the proof of Theorem \ref{prop:congruence} : We let $U'_{p,m}$ act on $A_m$ via the character  $\psi'_m:U'_{p,m}\stackrel{\lambda'_{p,m}}{\twoheadrightarrow} \Z/p^m\Z\hra A^\times_m$ to obtain a $\Z_p/p^m$-linear isomorphism
	\begin{equation}
	M(U^p U'_{p,m},V_{\Z_p}/(p^m)) \simeq  M(U^p U'_{p,m},A_m/(T-1))^{\oplus \dim_{\Z_p} V_{\bZ_p}}.
	\end{equation}
	As at the start of \S\ref{sec:congruence}, we obtain a $\Z_p/p^m$-algebra isomorphism
	\begin{equation}
	\T^S(U^p U'_{p,m},V_{\Z_p}/(p^m)) \simeq \T^S(U^p U'_{p,m},A_m/(T-1)),\qquad m\ge 1,
	\end{equation}
	which is compatible with \eqref{eq:M-non-const} via the respective Hecke algebra actions on both sides.	
\end{proof}	

\begin{remark}
	In fact the argument of this section still goes through and produces some congruence without the assumption that $(U_{p,m},\lambda_{p,m})$ is omni-supercuspidal, which plays a role only in applications. However the outcome is less interesting without a careful choice of the pair $(U_{p,m},\lambda_{p,m})$.
\end{remark}

\subsection{An application to Galois representations}\label{subsec:app}

  We illustrate how to employ Theorems \ref{prop:congruence} and \ref{prop:congruenceV}  to construct Galois representations from automorphic representations in a suitable context. The idea is to reduce to the case when automorphic representations have supercuspidal components. In fact, Remark 7.4 of \cite{Scholze-LT} reads \textit{``...and it seems reasonable to expect that one could do a similar argument in the compact unitary case, providing an alternative to the construction of Galois representations of Shin \cite{Shi11} and Chenevier--Harris \cite{CH13}, by reducing directly to the representations constructed by Harris--Taylor.''}\,\footnote{We copied the sentence except that the bibliographic items have been adapted. Remark 7.4 of \cite{Scholze-LT} also mentions the case of Hilbert modular forms but we chose to concentrate on the more complicated case treated here.} Confirmation of this is the goal of this section.

Let $n\in \Z_{\ge 2}$. Recall that $F$ is a totally real field, and let $E$ be a totally imaginary quadratic extension of $F$ with complex conjugation $c\in \Gal(E/F)$. For a finite place $w$ of $E$, we write $\textup{LL}_w$ for the unramified local Langlands correspondence for $\GL_n(E_w)$, from irreducible unramified representations of $\GL_n(E_w)$ to continuous semisimple unramified $n$-dimensional representations of the Weil group $W_{E_w}$ of $E_w$ (with coefficients in $\C$, up to isomorphisms).

Let $\Pi$ be a regular C-algebraic cuspidal automorphic representation of $\GL_n(\A_E)$ that is conjugate self-dual, i.e. its contragredient $\Pi^\vee$ is isomorphic to the complex conjugate $\Pi^c$. Let $S_E$ be the set of all infinite places and the finite places of $E$ where $\Pi$ is ramified, and $S$ all the places of $F$ below $S_E$.
Fix a prime $p$ and an isomorphism $\iota:\ol\Q_p\simeq \C$.
A well known theorem by Clozel (\!\!\cite[Thm 1.1]{Clo91}, based on \cite{Kot92b}) states the following. At each place $w$ of $E$, write $|\det_w|$ for the determinant map on $\GL_n(E_w)$ composed with the absolute value on $E_w$ which is normalized to send a uniformizer to the inverse of the residue field cardinality.

\begin{theorem}[\!\!\textup{\cite{Clo91}}]\label{thm:HT}
   Suppose that there exists a finite place $v$ of $E$ where $\Pi_v$ is a discrete series representation. Then there exists a continuous representation
   $$\rho_{\Pi,\iota}: \Gal(\ol E/E)\ra \GL_n(\ol \Q_p)$$
   which is unramified outside $S_E$
   such that $\textup{LL}_w(\Pi_w)\otimes |\det_w|^{(1-n)/2}$ is isomorphic to the semisimplification of $\iota\rho_{\Pi,\iota}|_{W_{E_w}}$ for all $w\notin S_E$.
\end{theorem}
\begin{remark}
  Harris and Taylor showed in \cite{HT01} (see Thm.~VII.1.9 therein) that $\Pi_w$ and  $\rho_{\Pi,\iota}|_{W_{E_w}}$ still correspond under the local Langlands correspondence at ramified primes not above $p$.
   The latter result was later also obtained in \cite{Sch-LLC} by a different approach.
\end{remark}

We remove the assumption that  $\Pi_v$ is a discrete series representation from the above theorem when $p$ is not too small using congruences, to obtain Theorem \ref{thm:main-app}. The condition on $p$ is due to Proposition \ref{prop:0-toral-abundance}.

\begin{theorem}\label{thm:main-app}
  Suppose that $p>n$.
  Then
   for every regular C-algebraic conjugate self-dual cuspidal automorphic representation $\Pi$ of $\GL_n(\A_E)$,
    there exists a continuous representation
   $$\rho_{\Pi,\iota}: \Gal(\ol E/E)\ra \GL_n(\ol \Q_p)$$
   which is unramified outside $S_E$
   such that $\textup{LL}_w(\Pi_w)\otimes |\det_w|^{(1-n)/2}$ is isomorphic to the semisimplification of $\iota\rho_{\Pi,\iota}|_{W_{E_w}}$ for all $w\notin S_E$.
\end{theorem}

\begin{remark}\label{rem:use-appD-1b}
The assumption that $p>n$ can be removed by appealing to Appendix D, which proves Theorem \ref{thm:existence-omnisctypes} for all primes $p$, see also Remark \ref{rem:use-appD-1}.
\end{remark}

\begin{remark}
  This theorem is not new. A stronger statement on the existence and local-global compatibility for $\rho_{\Pi,\iota}$ has been known by \cite{CHL11b,Shi11,CH13} without restriction on~$p$. The local-global compatibility was further strengthened by \cite{Car12, Car14,BLGGT14-LGC}. Much of \cite{Shi11} was reproved in \cite{SS13} by a simpler method. The point is that the proof here is still simpler as there is no eigenvariety as in \cite{CH13} and no elaborate geometric and endoscopic arguments as in \cite{Shi11}, as far as Theorem \ref{thm:HT} is taken for granted. (The geometry and harmonic analysis involved in Theorem \ref{thm:HT} are less complicated than those of \cite{Shi11,SS13}.)
\end{remark}

\begin{proof}
  We will freely use the notions and base change theorems of \cite{Lab11} for unitary groups to go between automorphic representations on unitary groups and general linear groups.

  Let $\textup{Spl}^S_{E}$ (resp. $\textup{Spl}^S_{F}$) be the set of places of $E$ (resp. $F$) outside $S_E$ (resp. $S$) that are split in $E/F$. By a standard reduction step using automorphic base change over quadratic extensions as in the proof of \cite[Thm.~VII.1.9]{HT01} (also see the proof of \cite[Prop.~7.4]{Shi11}), it suffices to show the compatibility for $\rho_{\Pi,\iota}$ only at $w\in \textup{Spl}^S_{E}$.

  By a patching argument (executed as in the proof of \cite[Thm.~3.1.2]{CH13}), we further reduce to the case where
  \begin{itemize}
    \item $[F:\Q]$ is even,
    \item every place of $F$ above $p$ is split in $E$,
    \item every finite place of $F$ is unramified in $E$.
  \end{itemize}
  Since $[F:\Q]$ is even, there exists a unitary group $G$
  over $F$ which is an outer form of $\GL_n$ with respect to the quadratic extension $E/F$ such that $G$ is quasi-split at all finite places and anisotropic at all infinite places. We may view $\Pi$ as a representation of $(\Res_{E/F}(G\otimes_F E))(\A_F)$, which is isomorphic to $\GL_n(\A_E)$; in particular $\Pi_x$ stands for the component of $\Pi$ at a place $x$ of $F$. Since $\Pi$ is conjugate self-dual, by \cite[Thm.~5.4,~5.9]{Lab11}, there exists an automorphic representation $\pi$ on $G(\A_F)$ such that $\Pi_x$ is the unramified base change of $\pi_x$ at all places $x$ of $F$ outside $S$.

    We set up some more notation. For $w \in \textup{Spl}^S_{E}$, we denote by $\ov w \in \textup{Spl}^S_{F}$ the restriction of $w$ to $F$.
 Recall that we denote by $U^p=\prod_{x\nmid p\infty} U_x$ a (fixed) compact open subgroup of $G(\bA_F^{\infty, p})$ such that $U_x$ is hyperspecial for $x \notin S$.
 For $x \in \textup{Spl}^S_{F}$ and for each $w$ above $x$, we fix isomorphisms $i_w:G(F_x)\simeq \GL_n(E_w)$ carrying $U_x$ onto $\GL_n(\cO_{E_w})$ such that $\pi_x\simeq \Pi_w$ via $i_w$.
 We write $\varpi_w$ for a uniformizer of $E_w$ and write $T^{(i)}_w$ for the following double coset
    $$T^{(i)}_w:=\left[\GL_n(\cO_{E_w}) \begin{pmatrix} \varpi_w \textrm{I}_i & 0\\ 0 & \textrm{I}_{n-i} \end{pmatrix}\GL_n(\cO_{E_w}) \right],$$
    which we might also view as a double coset of  $U^S\backslash G(\A_F^S)/U^S$ by requiring all other factors to be the trivial double coset.
   We also denote by $T^{(i)}_w$ the double coset operator corresponding to this double coset acting  on appropriate spaces that can be deduced from the context.
    Given an irreducible unramified representation $\sigma_w$ of $\GL_n(E_w)$ (or $G(F_x)$), we write $T^{(i)}_w(\sigma_w)$ for the eigenvalue of $T^{(i)}_w$ on the one-dimensional space $\sigma_w^{\GL_n(\cO_{E_w})}$.
    For convenience, we introduce the following \emph{variant of the big Hecke algebra}: define $\T^S_{\Spl}$ to be the $\Z$-subalgebra of $\Z[U^S\backslash G(\A_F^S)/U^S]$ generated by $T_{w}^{(i)}$ for $w\in \textup{Spl}^S_E$ (excluding $w$ not split over $F$) and $1\le i\le n$. Replacing $\T^S$ with $\T^S_{\Spl}$, we define other Hecke algebras to be the image of $\T^S_{\Spl}$ in the endomorphism algebras of appropriate spaces of automorphic forms.
    Note that Theorems \ref{prop:congruence} and \ref{prop:congruenceV} are still valid with $\T^S_{\Spl}$ in place of $\T^S$: indeed we retain the same isomorphism between the same spaces of automorphic forms, and the $\T^S_{\Spl}$-equivariance is simply weaker than the $\T^S$-equivariance.

     We choose $U=U_pU^p\subset G(\A_F^{\infty})$ sufficiently small such that $\pi^{\infty}$ has nonzero $U$-fixed vectors, while keeping $U_x$ hyperspecial for $x \notin S$.     Since $G_\infty(\R)$ is compact, we see that $\pi_\infty^\vee$ comes from an irreducible algebraic representation $V'_{\C}$ of  $G_\infty\times_\R \C=(\Res_{F/\Q}G)\times_\Q \C$.
     As in Remark \ref{rem:model-over-L}, there exists a finite Galois extension $L/\Q$ in $\C$ containing $F$ such that $(\Res_{F/\Q}G)\times_\Q L$ is split, thus $(\Res_{F/\Q}G)\times_\Q L\simeq \tilde{\mathbf G}_L$ for a split group $\tilde{\mathbf G}$ over $\cO_L$, and such that there is an algebraic representation $V$ of $\tilde{\mathbf G}$ over $\cO_L$ giving a model for $V'_{\C}$.
According to \S\ref{subsec:non-constant}, this leads to an action of $U_p$ on the corresponding free $\bZ_p$-module $V_{\bZ_p}$ such that $\pi$ contributes to $M(U, V_{\bZ_p})$, see Remark \ref{rem:different-incarnations}. (Note that $\pi_\infty$ is a direct summand of $(V\otimes_{\Z} \C)^\vee =(V_{\bZ_p}\otimes_{\Z_p} \C)^\vee$ as a $G_\infty(\R)$-representation.)
    We let $\{U_{p,m}\}_{m \geq 1}$ be compact open subgroups of $U \subset G_p(\bQ_p)$ with an action of $U_{p,m}$ on $A_m$ arising from multiplication by a character $\psi_m: U_{p,m} \ra A_m^\times$ for $m \geq 1$, factoring through $\lambda_m: U_{p,m} \ra \Z/p^m\Z$, as in Theorem \ref{prop:congruenceV}. This means we have
  \begin{equation}\label{eq:cong-modpm-modT-1}
\T_{\Spl}^S(U^p U_{p,m},V_{\bZ_p}/(p^m)) \simeq \T_{\Spl}^S(U^p U_{p,m},A_m/(T-1)),\qquad m\ge 1,
\end{equation}
  Let $\cA^S(U^p U_{p,m},A_m)$ be the set of irreducible $\T^S_{\Spl,\ol\Q_p}$-modules appearing as a constituent of $M(U^p U_{p,m},A_m)\otimes_{\Z_p}\ol\Q_p$. This set is identified with the set of $\underline{\sigma}^S = \{\underline{\sigma}_x\}_{x\in \Spl_F^S}$, where $\underline{\sigma}_x$ is an irreducible unramified representation of $G(F_x)$, such that there exists an automorphic representation $\sigma$ of $G(\A_F)$ satisfying
  \begin{itemize}
    \item $\sigma_x\simeq \underline{\sigma}_x$ for $x\in \Spl_F^S$,
    \item $(\sigma^{\infty,p})^{U^p}\neq \{0\}$,
    \item $\sigma_\infty=\one$, i.e. the archimedean components of $\sigma$ are trivial, and
    \item $\Hom_{U_{p,m}}(\psi\circ \lambda_m, \sigma_p)\neq 0$ for some nontrivial character $\psi:\Z/p^m\Z\ra \C^*$.
  \end{itemize}
  The last condition implies that $\sigma_p$ is a supercuspidal representation of $G_p(\bQ_p)$.
  For each $\underline{\sigma}^S\in \cA^S(U^p U_{p,m},A_m)$, we choose a $\sigma$ as two sentences above. By base change theorems \cite[Cor.~5.3, Thm.~5.9]{Lab11} we obtain a cuspidal conjugate self-dual automorphic representation $\Sigma$ of $\GL_n(\A_E)$ that is regular and C-algebraic such that $\Sigma_x$ is the unramified base change of $\sigma_x$ at each place $x$ of $F$ outside $S$ and also that $\Sigma_p$ is supercuspidal.
  It follows from Theorem \ref{thm:HT} that there exists a continuous semisimple representation
  $$\rho_{\Sigma,\iota}: \Gal(\ol E/E)\ra \GL_n(\ol \Q_p)$$
   which is unramified outside $S_E$ such that $\textup{LL}_w(\Sigma_w)$ is isomorphic to the semisimplification of $\iota\rho_{\Sigma,\iota}|_{W_{E_w}}$ for all $w\notin S_E$. By the Chebotarev density and Brauer--Nesbitt theorems, $\rho_{\Sigma,\iota}$ is independent of the choice of $\sigma$ up to isomorphism. Thus we write $\rho_{\underline{\sigma}^S,\iota}:=\rho_{\Sigma,\iota}$. Let $E_S$ denote the maximal extension of $E$ in $\ol E$ unramified outside $S_E$. Then $\rho_{\underline{\sigma}^S,\iota}$ factors through $\Gal(E_S/E)$. At $w\in \textup{Spl}^S_{E}$, let $N(w)\in \Z_{\ge 1}$ denote the absolute norm of the finite prime $w$. Using $t$ as an auxiliary variable, the compatibility at $w$ of Theorem \ref{thm:HT} means that (cf.~\cite[Prop.~3.4.2.~part~2]{CHT08})
 \begin{equation}\label{eq:char(Frob)}
  \det(1+t \rho_{\underline{\sigma}^S,\iota}(\Frob_w))=\sum_{i=0}^n t^i N(w)^{i(i-1)/2} T_w^{(i)}(\underline{\sigma}_{\ol w}).
  \end{equation}
  On the other hand,
  $$\T_{\Spl}^S(U^p U_{p,m},A_m)\hra \T_{\Spl}^S(U^p U_{p,m},A_m)\otimes_{\Z_p} \ol\Q_p \simeq \prod_{\underline{\sigma}^S\in \cA^S(U^p U_{p,m},A_m)} \ol\Q_p,$$
  where for each $w \in \Spl_E^S$ and each $1 \leq i \leq n$, the image of $T_{w}^{(i)} \in \T_{\Spl}^S(U^p U_{p,m},A_m)$ in the $\underline{\sigma}^S$-component is  the scalar in $\ol\Q_p$ by which $T_{w}^{(i)}$ acts on the $x$-component of $\underline{\sigma}^S$ (viewed as a representation of $\GL_n(E_w)$ via $i_w$). Let $\Mat_{n \times n}(\cdot)$ denote the $n\times n$-matrix algebra over the specified coefficient ring.
  For $m \geq 1$, write
  $$\rho_m:\Gal(E_S/E) \ra  \prod_{\underline{\sigma}^S\in \cA^S(U^p U_{p,m},A_m)} \GL_n\left( \ol\Q_p \right)   \hookrightarrow \Mat_{n \times n} \left( \prod_{\underline{\sigma}^S\in \cA^S(U^p U_{p,m},A_m)} \ol\Q_p \right), $$

   where
    the $\underline{\sigma}^S$-part of the first map is $\rho_{\underline{\sigma}^S,\iota}$. We can extend this map linearly to a map
    $$\rho_m^B:B[\Gal(E_S/E)] \ra  \Mat_{n \times n}(B) , $$
for every $(\prod_{\underline{\sigma}^S\in \cA^S(U^p U_{p,m},A_m)} \ol\Q_p)$-algebra $B$.
By composing $\rho_m^B$ with the determinant, we produce a continuous $n$-dimensional determinant map  in the sense of Chenevier \cite{Chenevier-det}
  $$ d_m:\ol\Q_p[\Gal(E_S/E)]\ra  \prod_{\underline{\sigma}^S\in \cA^S(U^p U_{p,m},A_m)} \ol\Q_p $$
(which consists of maps $d_m^B$ for all $(\prod_{\underline{\sigma}^S\in \cA^S(U^p U_{p,m},A_m)} \ol\Q_p)$-algebra $B$, but we usually omit the index $B$ from the notation)  such that
 \begin{equation}\label{eq:dm-LGC}
 \det(1+t\rho_m(\gamma))=d_m(1+t\gamma),\qquad \gamma\in \Gal(E_S/E)
 \end{equation}

 We deduce from  \eqref{eq:char(Frob)} that all the coefficients of the characteristic polynomial $\chi(\Frob_w,t):=d_m(t-\Frob_w)$ in the sense of Chenevier are contained in $\T_{\Spl}^S(U^p U_{p,m},A_m)$ for all $w\in \textup{Spl}^S_E$. The same holds for all elements of $\Gal(E_S/E)$ because the union of Frobenius conjugacy classes over $\textup{Spl}^S_{E}$ is dense in $\Gal(E_S/E)$ by the Chebotarev density theorem. Hence it follows from \cite[Corollary~1.14]{Chenevier-det} that
 $d_m$ is the scalar extension of a $\T_{\Spl}^S(U^p U_{p,m},A_m)$-valued $n$-dimensional continuous determinant (in Chenevier's sense):
   $$ \bZ_p[\Gal(E_S/E)]\ra \T_{\Spl}^S(U^p U_{p,m},A_m)$$
   satisfying \eqref{eq:dm-LGC}. To save notation, we still write $d_m$ for the latter. Thus
   $$d_m(1+t \Frob_w)=\sum_{i=0}^n t^i N(w)^{i(i-1)/2} T_w^{(i)},\qquad w \in \textup{Spl}^S_{E}.$$
   Via \eqref{eq:cong-modpm-modT-1} we obtain a continuous $n$-dimensional $\T_{\Spl}^S(U^p U_{p,m},V_{\bZ_p}/(p^m))$-valued determinant
   $$D_m:(\bZ_p/(p^m))[\Gal(E_S/E)]\ra  \T_{\Spl}^S(U^p U_{p,m},V_{\bZ_p}/(p^m))$$
   such that
   	\begin{equation} \label{eq:Dm}
   	D_m(1+t\Frob_w) = \sum_{i=0}^n t^i N(w)^{i(i-1)/2} T_w^{(i)}
   	\end{equation} (equality taken inside $\T_{\Spl}^S(U^p U_{p,m},V_{\bZ_p}/(p^m))$) for $w \in \textup{Spl}^S_{E}$. We have the restriction map
   $$ \mathrm{res}_{m,1}: \T_{\Spl}^S(U^p U_{p,m},V_{\bZ_p}/(p^m)) \ra \T_{\Spl}^S(U^p U_{p,1},V_{\bZ_p}/(p^m))$$
	and for $m'\leq m$ the projection map
   $$ \mathrm{pr}_{m,m'}: \T_{\Spl}^S(U^p U_{p,1},V_{\bZ_p}/(p^m)) \ra \T_{\Spl}^S(U^p U_{p,1},V_{\bZ_p}/(p^{m'})). $$
   From \eqref{eq:Dm} and the density of Frobenii, we deduce that for $m\geq {m'} \geq 1$, we have
   $$\mathrm{pr}_{m,{m'}} \circ \mathrm{res}_{m,1} \circ D_m (1 + t \gamma) = \mathrm{res}_{{m'},1} \circ D_{m'} (1 + t \gamma) \quad \gamma \in \Gal(E_S/E).$$
By Amitsur's formula \cite[(1.5)]{Chenevier-det} and the properties of the determinant, we see that the $n$-dimensional determinant $\mathrm{res}_{m,1} \circ D_m$ is uniquely determined by its values on $1+t\Gal(E_S/E)$. Hence we obtain an equality of $\T_{\Spl}^S(U^p U_{p,1},{V_{\bZ_p}}/(p^{m'}))$-valued $n$-dimensional continuous determinants
   $$\mathrm{pr}_{m,{m'}} \circ \mathrm{res}_{m,1} \circ D_m = \mathrm{res}_{{m'},1} \circ D_{m'} $$
Taking the inverse limit over $\mathrm{res}_{m,1}\circ D_m$ for $m \geq 1$ (\!\!\cite[Lem.~3.2]{Chenevier-det}) and using from Remark \ref{rem:Hecke-quotient} that $\T_{\Spl}^S(U^p U_{p,1},{V_{\bZ_p}})= \varprojlim\limits_{m}\T_{\Spl}^S(U^p U_{p,1},{V_{\bZ_p}}/(p^m))$,
   we obtain a $\T_{\Spl}^S(U^p U_{p,1},{V_{\bZ_p}})$-valued $n$-dimensional continuous determinant
   $$D:\bZ_p[\Gal(E_S/E)]\ra  \T_{\Spl}^S(U^p U_{p,1},V_{\bZ_p})$$
   with $D(1+t\Frob_w)= \sum_{i=0}^n t^i N(w)^{i(i-1)/2} T_w^{(i)}$. Since $\pi$ contributes to $M(U^pU_{p,1},V_{\bZ_p})$ (as $\pi$ contributes to $M(U,V_{\bZ_p})$ and $U_{p,1}\subset U_p$), it gives rise to a $\bZ_p$-algebra morphism
   $$c_\pi: \T_{\Spl}^S(U^p U_{p,1},V_{\bZ_p})\ra \ol \Q_p,\qquad T_w^{(i)}\mapsto T_w^{(i)}(\pi_{\ol w}),\quad \forall w\in \textup{Spl}^S_E. $$
   The composition $c_\pi\circ D$ yields a continuous $n$-dimensional $\ol\Q_p$-valued determinant. It follows from \cite[Thm.~A, Ex.~2.34]{Chenevier-det} that $c_\pi\circ D$ arises from a continuous representation $\rho_\pi:\Gal(E_S/E)\ra \GL_n(\ol\Q_p)$ in the sense that
   $$c_\pi(D(1+t\gamma))=\det(1+t\rho_\pi(\gamma)),\qquad \gamma\in \Gal(E_S/E).$$
   Therefore we conclude that for $w\in \textup{Spl}^S_{E}$,
   $$\det(1+t\rho_\pi(\Frob_w))=c_\pi(D(1+t \Frob_w))= \sum_{i=0}^n t^i N(w)^{i(i-1)/2} T_w^{(i)}(\pi_{\ol w}).$$
   That is, $\textup{LL}_w(\pi_{\ol w})\otimes |\det_w|^{(1-n)/2}$ is isomorphic to the semisimplification of $\iota\rho_{\pi}|_{W_{E_w}}$.
   The proof of the theorem is complete by setting $\rho_{\Pi,\iota}:=\rho_\pi$.
\end{proof}

\subsection{Density of supercuspidal points in the Hecke algebra}\label{subsec:density}

As another application of our main local theorem, we show that the supercuspidal locus is Zariski dense in the spectrum of the Hecke algebra of $p$-adically completed (co)homology following Emerton--Pa\v{s}k\=unas \cite{EP18}. Using Bushnell--Kutzko's study of types for $\GL_n$, they proved the result for a global definite unitary group which is isomorphic to a general linear group at $p$. Their machinery is quite general, enabling us to extend their result to general reductive groups which are compact modulo center at $\infty$ once we combine it with our local construction.

We retain the same notation as at the start of \S\ref{sec:congruence}. We may and will assume that $F=\Q$ by replacing $G$ with $\Res_{F/\Q}G$ as this does not sacrifice the quality of the theorem.
 As in \emph{loc.~cit.}~we assume that the central torus $Z(G)^0$ has the same $\Q$-rank and $\R$-rank. Let $L$ be a finite extension of $\Q_p$ with ring of integers $\cO$. (We have renewed the use of $L$ here to be consistent with \cite{EP18}. In \S\ref{subsec:non-constant} the letter $L$ denoted a certain number field which we will denote by $\mathbb L$ below.) This will be our coefficient field for the involved representations. Fix an algebraic closure $\ol L$ of $L$ and a uniformizer $\varpi\in \cO$.
So far in this section, we worked with Hecke algebras as $\Z_p$-algebras acting on the space of automorphic forms as $\Z_p$-modules, but everything carries over verbatim with $\cO$ and $\varpi$ in place of $\Z_p$ and $p$. This extension is not strictly necessary but sometimes convenient as irreducible algebraic representations of $G_{\ol \Q_p}$ need not be defined over $\Q_p$.
If $U_p$ is a compact open subgroup of $G(\bQ_p)$, then we define the completed group algebra of $U_p$ over $\cO$ to be $\cO[[U_p]]:=\varprojlim_{U'_p} \cO[U_p/U'_p]$, where the limit is taken over open normal subgroups $U'_p\subset U_p$. The topology on $\cO[[U_p]]$ is given by the projective limit (with the usual topology on $\cO[U_p/U'_p]$ as a finite free $\cO$-module). Whenever we work with $\cO[[U_p]]$-modules, we work in the category of compact linear-topological $\cO[[U_p]]$-modules (resp. $\cO$-modules) and
denote the Hom space by $\Hom^{\textup{cont}}_{\cO[[U_p]]}(\cdot,\cdot)$ (resp. $\Hom^{\textup{cont}}_{\cO}(\cdot,\cdot)$).

Let $U^p=\prod_{w\nmid p,\infty} U_w$ be an open compact subgroup of $G(\A^{\infty,p})$ such that $U_w$ is hyperspecial for all $w$ away from a finite set of places $S$.
Define $Y(U^pU_p):=G(\Q)\backslash G(\A)/U^p U_p G(\R)^\circ$.
Consider the completed homology
$$ \tilde H_0(U^p):=\varprojlim\limits_{U_p} H_0(Y(U^pU_p),\cO),$$
where $U_p$ runs over open compact subgroups of $G(\Q_p)$. Then $\tilde H_0(U^p)$ is a finitely generated $\cO[[U_p]]$-module that is $\cO$-torsion free for any compact open subgroup $U_p$ of $G(\bQ_p)$. If $U^p$ or $U_p$ is sufficiently small, for instance if $U^p U_p$ is a neat subgroup (in the sense of \cite[\S0]{PinkThesis}), then $U_p$ acts on points of $G(\Q)\backslash G(\A)/U^p G(\R)^\circ$ with trivial stabilizers, and $\tilde H_0(U^p)$ is free over $\cO[[U_p]]$. We topologize $\tilde H_0(U^p)$ as a finitely generated $\cO[[U_p]]$-module, using the topology of $\cO[[U_p]]$. This is equivalent to the inverse limit topology where the topology on $ H_0(Y(U^pU_p),\cO)$ arises from the topology of $\cO$.

We also define the completed cohomology
\begin{equation}\label{eq:completed-cohomology}
\tilde H^0(U^p):=\varprojlim\limits_{s\ge 1} \varinjlim\limits_{U_p} H^0(Y(U^pU_p),\cO/\varpi^s),
\end{equation}
which is complete for the $\varpi$-adic topology, or equivalently the inverse limit topology over $s$ of the discrete topology on the direct limit.
In our earlier notation, $H^0(Y(U^pU_p),\cO/\varpi^s)=M(U^p U_p,\cO/\varpi^s)$.
We have a canonical isomorphism $\tilde H^0(U^p)=\Hom^{\textup{cont}}_{\cO}(\tilde H_0(U^p),\cO)$ as topological $\cO$-modules, where the topology on the latter is given by the supremum norm.

We define the ``big'' Hecke algebra
\begin{equation}\label{eq:Hecke-complete}
\T^S(U^p):=\varprojlim\limits_{U_p,~s} \T^S(U^pU_p,\cO/\varpi^s),
\end{equation}
recalling that $ \T^S(U^pU_p,\cO/\varpi^s)$ was introduced around \eqref{eq:TS}, where the $\Z_p$-algebra setup earlier extends to the $\cO$-algebra setup in the evident manner.
We equip $\T^S(U^p)$ with the profinite topology via \eqref{eq:Hecke-complete}, where the finite set $\T^S(U^pU_p,\cO/\varpi^s)$ is equipped with the discrete topology.
For each compact open subgroup $U_p$ of $G(\bQ_p)$ and a locally algebraic representation $V$ of $U_p$ over $L$, we have $\T^S(U^p)$-equivariant isomorphisms
$$\Hom_{\cO[[U_p]]} (\tilde H_0(U^p),V^*) = \Hom_{U_p}(V,\tilde H^0(U^p)_L)= M(U^p U_p, V^*),$$
as explained in \cite[5.1]{EP18}. The $\T^S(U^p)$-module structure is semisimple as it is the case on the space of algebraic automorphic forms.

Let $\{V_i\}_{i\in I}$ be a family of continuous representations of $U_p$ on finite dimensional $L$-vector spaces.
We recall from \cite[Def.~2.6]{CDP14} (see Lemmas 2.7 and 2.10 therein for equivalent characterizations):
\begin{definition}\label{def:capture}
  Let $M$ be a compact linear-topological $\cO[[U_p]]$-module.
  We say that $\{V_i\}$ \textbf{captures} $M$ if there is no nontrivial (i.e. other than $M=Q$) quotient $M\twoheadrightarrow Q$ inducing an isomorphism
$$\Hom^{\textup{cont}}_{\cO[[U_p]]}(Q,V_i^*)\simeq \Hom^{\textup{cont}}_{\cO[[U_p]]}(M,V_i^*),\qquad \forall i\in I.$$
\end{definition}

\begin{proposition}\label{prop:capture}
  Assume $p>\Cox(G)$.
  Then there exist
\begin{itemize}
  \item an open compact pro-$p$ subgroup $U_p$ of $G(\bQ_p)$ and
  \item a countable family of smooth representations $\{V_i\}_{i\in I}$ of $U_p$ on finite dimensional $L$-vector spaces
\end{itemize}
such that the following hold:
\begin{itemize}
  \item $(U_p,V_i\otimes_L \ol L)$ is a supercuspidal type for every $i$,
  \item $\{V_i\}_{i\in I}$ captures $\cO[[U_p]]$.
\end{itemize}
  Moreover $U_p$ can be chosen to be arbitrarily small.
\end{proposition}

\begin{remark}
  This proves a result similar to \cite[Cor.~4.2]{EP18} but slightly stronger in that our proposition implies the conclusion there via Lemmas 2.8 and 2.10 of \emph{loc.~cit.}
  Note that the proof below still produces $\{V_i\}_{i\in I}$ capturing $\cO[[U_p]]$ as far as $U_{p,m}$ is a sequence of pro-$p$ groups forming a neighborhood basis of 1.
  Indeed the proof readily adapts to other types such as principal series types.
\end{remark}

\begin{remark}
  As can be seen in the proof, we construct $V_i$ to be irreducible representations of $U_p$ over $L$ which may become reducible over $\ol L$.
\end{remark}

\begin{proof}
 Let $\{(U_{p,m},\lambda_m)\}_{m\ge1}$ be a sequence for $G(\Q_p)$ as in Theorem \ref{thm:existence-omnisctypes}, where we may assume that $U_p:=U_{p,1}$ is a pro-$p$ group and arbitrarily small (see the first paragraph in the proof of Theorem \ref{thm:existence-omnisctypes}).

 Let $L_m\subset \ol L$ denote the totally ramified extension of $L$ generated by $p^m$-th roots of unity, with $\cO_m$ denoting its ring of integers and  $\varpi_{m}$ a uniformizer in $\cO_m$.
 We compose $\lambda_m:U_{p,m}\ra \Z/p^m \Z$ with a fixed character $\Z/p^m\Z\hra \cO_m^\times$ to define a smooth character $\psi^\circ_m: U_{p,m}\ra \cO_m^\times$. Put $\psi_m:=\psi^\circ_m\otimes_{\cO_m} L_m$.
 The dual representations, i.e., inverse characters, corresponding to $\psi^\circ_m$ and $\psi_m$ are denoted by $\psi^{\circ *}_m$ and $\psi^*_m$.
   For an embedding $\sigma:L_m\hra \ol L$ over $L$, write $\psi^*_{m,\sigma}:=\psi^*_m\otimes_{L_m,\sigma} \ol L$ (viewed as a one-dimensional representation over $\ol L$).
  We will think of $\psi^\circ_m, \psi^{\circ *}_m$ (resp.~$\psi_m,\psi^*_m$) as representations on a free $\cO$-module (resp.~$L$-module) of rank $[L_m:L]$. The notation is consistent as $\psi^*_m$ is then the dual representation of $\psi_m$ over $L$.
 The multiplication by $\varpi_m$ (on the underlying $\cO_m$-module) yields a $U_{p,m}$-equivariant map $[\varpi_m]:\psi^{\circ *}_m \ra \psi^{\circ *}_m$. Since $p^m$-th roots of unity become trivial in $\cO_m/(\varpi_m)$, the $U_{p,m}$-action on the cokernel $\psi^{\circ *}_m/(\varpi_m)$ is trivial.

 Take $\{V_i\}$ to be the family of irreducible $L$-subrepresentations of
 $\textup{Ind}^{U_p}_{U_{p,m}}\psi_m$ (which is semisimple) for all $m\ge 1$.
 Since $(U_{p,m},\lambda_m)$ is omni-supercuspidal, and since $\psi_m\otimes_L \ol{L}\simeq \oplus_{\sigma\in \Hom_L(L_m,\ol L)} \psi_m\otimes_{L_m,\sigma} \ol L$ as a $U_{p,m}$-representation, the pair
 $(U_{p,m},\psi_m\otimes_L \ol{L})$ is a supercuspidal type. Via Frobenius reciprocity, it follows that $(U_p,V_i\otimes_L \ol{L})$ is also a supercuspidal type.

 It remains to prove that $\{V_i\}$ captures $\cO[[U_p]]$.
 Put $M:=\cO[[U_p]]$ and let $Q$ be the smallest quotient of $M$ such that $\Hom^{\textup{cont}}_{\cO[[U_p]]}(Q,V_i^*)\simeq \Hom^{\textup{cont}}_{\cO[[U_p]]}(M,V_i^*), \forall i\in I.$ This implies that $Q$ captures $\{V_i\}$.
 We need to show that $M \stackrel{\sim}{\ra} Q$.

 For each $m$, we have $\textup{Ind}^{U_p}_{U_{p,m}} \psi^*_m$ isomorphic to the direct sum of $V^*_i$'s for suitable indices $i$. So
 $M\twoheadrightarrow Q$ induces an isomorphism
 $$\Hom^{\textup{cont}}_{\cO[[U_p]]}(Q, \textup{Ind}^{U_p}_{U_{p,m}}\psi^*_m) \stackrel{\sim}{\ra} \Hom^{\textup{cont}}_{\cO[[U_p]]}(M, \textup{Ind}^{U_p}_{U_{p,m}}\psi^*_m),\qquad m\ge 1.$$
 Via Frobenius reciprocity,
 $$\Hom^{\textup{cont}}_{\cO[[U_{p,m}]]}(Q, \psi^*_m)\stackrel{\sim}{\ra} \Hom^{\textup{cont}}_{\cO[[U_{p,m}]]}(M, \psi^*_m),\qquad m\ge 1.$$
 The isomorphism continues to hold when $\psi^*_m$ is replaced with $\psi^{\circ *}_m$.
  Indeed, the injectivity is obvious as the map is induced by the surjection $M\twoheadrightarrow Q$. For the surjectivity, notice that the cokernel is a finitely generated $\cO$-module that is torsion free and vanishes after taking $\otimes_{\cO} L$. Now we consider the following commutative diagram with exact rows and vertical maps induced by $M\twoheadrightarrow Q$.
  $$\xymatrix{
    \Hom^{\textup{cont}}_{\cO[[U_{p,m}]]}(Q, \psi_m^{\circ *}) \ar[r]^-{[\varpi_m]} \ar[d]^-{\sim} & \Hom^{\textup{cont}}_{\cO[[U_{p,m}]]}(Q,\psi_m^{\circ *}) \ar[r] \ar[d]^-{\sim}  & \Hom^{\textup{cont}}_{\cO[[U_{p,m}]]}(Q, \psi^{\circ *}_m/(\varpi_m))  \ar[d]^-{\alpha}\\
    \Hom^{\textup{cont}}_{\cO[[U_{p,m}]]}(M, \psi_m^{\circ *}) \ar[r]^-{[\varpi_m]} & \Hom^{\textup{cont}}_{\cO[[U_{p,m}]]}(M,\psi_m^{\circ *}) \ar[r]^-{\beta} & \Hom^{\textup{cont}}_{\cO[[U_{p,m}]]}(M, \psi^{\circ *}_m/(\varpi_m))
  }$$
  The map $\alpha$ is injective by surjectivity of $M\twoheadrightarrow Q$.
  To see that $\alpha$ is surjective, it is enough to check that $\beta$ is surjective; this is true by projectivity of $M$ (as a compact linear-topological $\cO[[U_{p,m}]]$-module).
   Hence $\alpha$ gives an isomorphism
  $$\Hom^{\textup{cont}}_{\cO[[U_{p,m}]]}(Q, \psi^{\circ *}_m /(\varpi_m)) \stackrel{\sim}{\ra} \Hom^{\textup{cont}}_{\cO[[U_{p,m}]]}(M,\psi^{\circ *}_m /(\varpi_m))$$
  as vector spaces over $\cO/(\varpi)$. The Hom spaces do not change if $Q$ and $M$ are replaced with $Q/(\varpi)$ and $M/(\varpi)$.
   Since $U_{p,m}$ acts trivially on $\psi^{\circ *}_m/(\varpi_m)$, we deduce that
 $$ (M/(\varpi))_{U_{p,m}}\stackrel{\sim}{\ra}  (Q/(\varpi))_{U_{p,m}},\qquad m\ge 1,$$
   where the subscripts signify the $U_{p,m}$-coinvariants. Since $\{U_{p,m}\}$ is a neighborhood basis of 1, we deduce that $M/(\varpi)\stackrel{\sim}{\ra} Q/(\varpi)$.
   As $Q$ is $\cO$-torsion free by \cite[Lem.~2.4]{EP18}, this implies that $(\ker(M\twoheadrightarrow Q))/(\varpi)=\{0\}$.
   The topological Nakayama's lemma for compact $\cO[[U_{p,m}]]$-modules \cite[Lem.~5.2.18]{CohomologyOfNumberFields} implies that $\ker(M\twoheadrightarrow Q)=\{0\}$. Hence $M\stackrel{\sim}{\ra} Q$, and $\{V_i \}$ captures $M$.
\end{proof}

\begin{corollary}\label{cor:capture}
    Assume $p>\Cox(G)$, and let $U_p$ and $\{V_i\}$ be as in Proposition \ref{prop:capture}.
    If $U^p$ or $U_p$ is sufficiently small (e.g., if $U^p U_p$ is neat), then for each continuous finite dimensional representation $W$ of $G(\Q_p)$ over $L$, the family $\{V_i\otimes_L W\}$
    captures $\tilde H_0(U^p)$ and
  the evaluation morphism
  $$ \bigoplus_i \Hom_{U_p}(V_i\otimes W,\tilde H^0 (U^p){_L})\otimes (V_i\otimes W) \longrightarrow \tilde H^0 (U^p)_L$$
  has dense image.
\end{corollary}

\begin{proof}
  As remarked earlier, if $U^p U_p$ is neat, then $\tilde H_0 (U^p)$ is a free $\cO[[U_p]]$-module.
  Therefore Lemmas 2.8 and 2.9 of \cite{EP18} imply that $\{V_i\otimes_L W\}$
    captures $\tilde H_0(U_p)$.
    The assertion about density is simply an equivalent characterization of capture in \cite[Lem.~2.10]{CDP14} (or \cite[Lem.~2.3]{EP18}),
    noting that $\tilde H^0 (U^p)_L = \Hom^{\textup{cont}}_{\cO}(\tilde H_0(U^p),L)$.
\end{proof}

\begin{remark}
 To apply the corollary when $U^p$ is fixed (and $p>\Cox(G)$), we choose small enough $U_p$ such that $U^p U_p$ is neat. This is possible as $U_p$ can be made arbitrarily small in Proposition \ref{prop:capture}.
\end{remark}

Here is an informal discussion of the corollary. Recall that $ \Hom_{U_p}(V_i\otimes W,\tilde H^0 (U^p)_L)= M(U^p U_p, (V_i\otimes W)^*)$.
If $W$ is the restriction of an irreducible algebraic representation of $G_{\mathbb L}$ (we are writing $\mathbb L$ for the number field $L$ in \S\ref{subsec:non-constant}),
this is a space of algebraic automorphic forms such that every automorphic representation $\pi$ that contributes to it  (as in Remark \ref{rem:different-incarnations}) has the properties that $\pi_p$ is supercuspidal (since $\pi_p$ contains the type $(U_p,V_i)$) and that $\pi|_{G(\R)^0}=W$.
Thus the corollary roughly asserts that automorphic forms which are supercuspidal at $p$ and have ``weight'' $W$ at $\infty$ form a dense subspace in the completed cohomology.

 Now we formulate a density statement in terms of Hecke algebras.
 Fix an open maximal ideal $\mathfrak{m}\subset \T^S(U^p)$ and consider the $\mathfrak m$-adic localization $\T^S(U^p)_{\mathfrak m}$, which is a direct factor of $\T^S(U^p)$ as a topological ring. (See \cite[5.1]{EP18}, cf.~\eqref{Ch_rem} and Remark C.5.)
 Fix a representation $W$ of $G(\Q_p)$ coming from an irreducible algebraic representation as in the preceding paragraph.
 Let $$\Sigma(W)_{\textup{sc}}\subset \textup{Spec}\, \T^S(U^p)_{\mathfrak m}[1/p]$$
  denote the subset of closed points such that the corresponding
 morphism $\T^S(U^p)_{\mathfrak m}[1/p]\ra \ol L$ (up to the $\Gal(\ol L/L)$-action) comes from an eigen-character of $\T^S(U^p)$ in
 $M(U^p U_p, W^*)$ for some $U_p$, and such that the eigenspace gives rise to an automorphic representation of $G(\A)$ whose component at $p$ is supercuspidal.

\begin{theorem} \label{thm:zariskidense}
  If $p>\Cox(G)$, the subset $\Sigma(W)_{\textup{sc}} \subset \textup{Spec}\, \T^S(U^p)_{\fkm}$ is Zariski dense.
\end{theorem}

\begin{remark}\label{rem:use-appD-2}
The assumption that $p>\Cox(G)$ can be removed as in Remark \ref{rem:use-appD-1} by using Appendix D.
\end{remark}

\begin{proof}
  Essentially the same argument as in the proof of \cite[Thm.~5.1]{EP18} works, so we only sketch the proof.
  Choose a sufficiently small $U_p$ and $\{V_i\}$ as in Proposition \ref{prop:capture}.
  By Corollary \ref{cor:capture}, $\{W\otimes V_i\}$ captures the $\mathfrak m$-adic localization $\tilde H_0(U^p)_{\mathfrak m}$, which is a direct summand of $\tilde H_0(U^p)$. As in \emph{loc.~cit.}~we obtain
  that
  \begin{eqnarray*}
  \Hom^{\textup{cont}}_{\cO[[U_p]]}(\tilde H_0(U^p)_{\mathfrak m},(W\otimes V_i)^*)&\simeq& \Hom^{\textup{cont}}_{\cO[[U_p]]}(W\otimes V_i,\tilde H^0(U^p)_{\mathfrak m}\otimes L) \\
  & \simeq &\Hom^{\textup{cont}}_{\cO[[U_p]]}(W\otimes V_i,(\tilde H^0(U^p)_{\mathfrak m}\otimes L)^{\textup{alg}}),
  \end{eqnarray*}
  where $(\cdot)^{\textup{alg}}$ designates the subspace of locally algebraic vectors for the $U_p$-action. Thus $\T^S(U^p)_{\fkm}[1/p]$ acts semisimply on $(\tilde H^0(U^p)_{\mathfrak m}\otimes L)^{\textup{alg}}$ as it is the case on the space of algebraic automorphic forms. Moreover the support of $\Hom^{\textup{cont}}_{\cO[[U_p]]}(W\otimes V_i,(\tilde H^0(U^p)_{\mathfrak m}\otimes L)^{\textup{alg}})$ in the maximal spectrum of $\T^S(U^p)_{\fkm}[1/p]$ is contained in $\Sigma(W)_{\textup{sc}}$ as explained in \emph{loc.~cit.} (The supercuspidality at $p$ comes from the fact that $(U_p,V_i)$ is a supercuspidal type.)
  Finally the Zariski density of $\Sigma(W)_{\textup{sc}}$ follows from
  Proposition 2.11 of \cite{EP18} based on Remarks 2.12 and 2.13 therein with $R=\T^S(U^p)_{\fkm}$. (We need Remark 2.13 as $\T^S(U^p)_{\fkm}$ is not known to be Noetherian in general; refer to the discussion in \S5.1 of \emph{loc.~cit.} Sometimes the Noetherian property can be proved, as in Appendix \ref{app:Galois-reps}.)
\end{proof}

\appendix
\numberwithin{equation}{section}

\section{Calculations for $D_{2N+1}$} \label{app:Dodd}
  The purpose of this appendix is to prove Proposition \ref{prop:0-toral-abundance} for split groups $G$ of type $D_{2N+1}$. In this appendix, by ``the proof'' we will always refer to the proof of Proposition \ref{prop:0-toral-abundance}.
  We maintain the notation from there. Recall that it suffices to exhibit an elliptic maximal torus $T\subset G$ such that for every $n\in \Z_{\ge 1}$, there exists a $G$-generic element $X\in \mathfrak t^*$ of depth $-r$ with $n<r\le n+1$. Put $s:=2N+1$ in favor of simpler notation.

  Our strategy is similar to the other cases of the proof. We take $E/F$ to be the unramified extension of degree $2s-2=4N$ with $\Gal(E/F)=\langle \sigma \rangle$,  where $\sigma$ denotes the (arithmetic) Frobenius automorphism. Recall that $T^{\textup{sp}}$ is a split maximal torus in $G$.

  Let $\{e_1,...,e_s\}$ be a basis for $X_*(T^{\textup{sp}})\otimes_{\Z}\R$. Without loss of generality, i.e. by changing the basis if necessary, we let the simple coroots be
  $$\check\alpha_i=e_i-e_{i+1},\quad 1\le i\le s-1,\qquad \check\alpha_s=e_{s-1}+e_s,$$
  and the coroots $\check\Phi=\{ \pm( e_i\pm e_j): 1\le i<j\le s\}$.
  Take the Coxeter element $w=s_{\check\alpha_1}s_{\check\alpha_2}\hdots s_{\check\alpha_s}$, where $s_{\check\alpha_i}$ denotes the reflection on $X_*(T^{\textup{sp}})\otimes_{\Z}\R$ corresponding to $\check\alpha_i$.
The order of $w$ is equal to $\Cox(G)=2s-2=4N$. An easy computation shows:
	\begin{eqnarray*}
		w(\check \alpha_i)&=& \check \alpha_{i+1},\qquad 1\le i\le s-3, \\
		w(\check \alpha_{s-2})&=& \check \alpha_1+\check \alpha_2+\cdots+\check \alpha_{s}\\
		w(\check \alpha_{s-1})&=& -(\check \alpha_1+\check \alpha_2+\cdots+\check \alpha_{s-1})\\
		w(\check \alpha_{s})&=& -(\check \alpha_1+\check \alpha_2+\cdots+\check \alpha_{s-2}+\check \alpha_{s})
	\end{eqnarray*}
  As in the earlier proof, we define $T$ over $F$ from the cocycle $\ol f:\Gal(E/F)\ra W$ sending $\sigma$ to $w$. (A cocycle here is a homomorphism as $\Gal(E/F)$ acts trivially on $W$.) We may and will identify $T_E$ with $T^{\textup{sp}}_E$ and fix $T$ henceforth. Let $r\in \Z$. Then giving $X\in \ft^*(E)_{-r}$ is equivalent to assigning $a_1,...,a_s\in \cO_E$ such that
  $$X(H_{\varpi_F^r\check\alpha_1})=a_1,~...,~X(H_{\varpi_F^r\check\alpha_s})=a_s$$
  since $\{\varpi_F^r\check\alpha_i\}_{1\le i\le s}$ is an $\cO_E$-basis of $\ft(E)_{r}$.
  (To see this, it is enough to check it for $r=0$. In this case, $\{\check\alpha_i\}_{1\le i\le s}$ generates a subgroup of the free $\Z$-module $X_*(T_E)$ with index coprime to $p$, since $p>\Cox(G)$ and thus $p$ does not divide the order of the Weyl group. It follows that $\{\check\alpha_i\}_{1\le i\le s}$ indeed generates $\ft(E)_0$ over $\cO_E$. Since $\ft(E)_0$ is a free $\cO_E$-module of rank $s$, linear independence follows.)

  The $G$-genericity means that we need
  \begin{equation}\label{eq:Dodd-1}
  v(X(H_{\check\beta}))=-r,\qquad \forall \check\beta\in \check\Phi.
  \end{equation}
  On the other hand, $X$ descends to an $\cO_F$-linear functional on $\ft_{r}$ if
   \begin{equation}\label{eq:Dodd-2}
    X(\sigma(H_{\varpi_F^r\check\alpha_i}))=\sigma(a_i),\qquad \forall 1\le i\le s.
   \end{equation}
 Hence in order to prove Proposition \ref{prop:0-toral-abundance} for split groups $G$ of type $D_{2N+1}=D_s$, it suffices to find $X$ satisfying \eqref{eq:Dodd-1} and \eqref{eq:Dodd-2}. Moreover, it suffices to find such an $X$ only when $r=0$, since the case $n<r\le n+1$ follows by multiplying $X$ with $\varpi_F^{-n-1}$.
  Thus we set $r=0$ from now on. Then \eqref{eq:Dodd-2} can be rewritten as
\begin{equation}\label{eq:Dodd-3}
\begin{array}{r@{}l}
		\sigma(a_i)&{}= a_{i+1},\qquad 1\le i\le s-3, \\
		\sigma(a_{s-2})&{}= a_1+a_2+\cdots+a_{s},\\
		\sigma(a_{s-1})&{}= -(a_1+a_2+\cdots+a_{s-1}),\\
		\sigma(a_{s})&{}= -(a_1+a_2+\cdots+a_{s-2}+a_{s+1}).
\end{array}
\end{equation}

  Take $\zeta$ to be a primitive $({q^{2s-2}-1})$-th root of unity in $\cO_E$, where $q$ is the residue field cardinality of $F$ (so $k_E=\F_{q^{2s-2}}$).
     Set
   $$a=\zeta^{\frac{q^{s-1}+1}{2}},\qquad b=\zeta^{\frac{q^{2(s-1)}-1}{2(q-1)}}.$$
  We would like to verify that the following solution to the system of equations \eqref{eq:Dodd-3} works:
 $$ a_i=\sigma^{i-1}(a),\quad 1\le i\le s-2,$$
 \begin{eqnarray}
  a_{s-1}&=&\frac12 (b-(a+\sigma(a)+\cdots+\sigma^{s-3}(a))+\sigma^{s-2}(a)),\nonumber\\
   a_{s}&=&\frac12 (-b-(a+\sigma(a)+\cdots+\sigma^{s-3}(a))+\sigma^{s-2}(a)).\nonumber
  \end{eqnarray}
   A simple computation shows
   \begin{equation}\label{eq:property-ab-typeD}
   \sigma^{s-1}(a)=-a,\qquad \sigma(b)=-b,
   \end{equation}
   $$a_{s-1}-a_s=b,\qquad a_{s-1}+a_s=-(a+\sigma(a)+\cdots+\sigma^{s-3}(a))+\sigma^{s-2}(a).$$
   Using this, it is elementary to check that \eqref{eq:Dodd-3} is satisfied.

   It remains to prove that \eqref{eq:Dodd-1} holds with $r=0$.
   As a preparation, observe that the reduction $\ol{a}\in k_E$ of $a$ generates $k_E$, namely
   \begin{equation}\label{eq:stabilizer-of-a-typeD}
   \sigma^j(\ol{a})=\ol{a},\quad j\in \Z_{\ge 1}\qquad \Rightarrow \qquad (2s-2)|j.
   \end{equation}
 When each coroot $\check\beta$ is written (uniquely) as $\check\beta=\sum_{i=1}^s \lambda_i \check\alpha_i$ with $\lambda_i\in \Z$,
 the condition imposed by \eqref{eq:Dodd-1} is that $v(\sum_{i=1}^s \lambda_i X( H_{\check\alpha_i}))=0$.
  As $\check\beta$ runs through all coroots (enough to consider positive coroots), the following are the conditions to check:
\begin{enumerate}
  \item $v(a_{i+1}+\cdots + a_j)=0$ for $0\le i<j\le s$.
  \item $v(a_{i+1}+\cdots + a_s + (a_{j+1}+\cdots + a_{s-2}))=0$ for $0\le i\le j\le s-2$.
  \item $v(a_{i+1}+\cdots + a_{s-2}+a_s)=0$ for $0\le i < s-2$.
\end{enumerate}
Divide Case (1) into (1a) $j\le s-2$, (1b) $j=s-1$, and (1c) $j=s$.
To check an element of $\cO_E$ has valuation zero, it suffices to show that the reduction is nonzero in $k_E$.

\vspace{.1in}

\paragraph{\textbf{Cases (1a), (1c), and (2)}}
  In these cases, the condition to be checked has the form
  \begin{equation}\label{eq:Case-1a,2}
  \sigma^{i'}(\ol{a})+\cdots + \sigma^{j'}(\ol{a})\neq 0,\qquad 1\le j'-i' \le s-1.
  \end{equation}
  Indeed this is clear in Case (1a) with $i'=i$ and $j'=j-1$. In Case (1c), it suffices to show that
  $-\sigma(\ol{a}_{i+1}+\cdots + \ol{a}_s)\neq 0$ but the left-hand side equals, via \eqref{eq:property-ab-typeD},
  $$-\sigma\left(\sigma^{s-2}(\ol{a})-(\ol{a}+\sigma(\ol{a})+\cdots + \sigma^{i-1}(\ol{a}))\right)
  = \ol{a}+\sigma(\ol{a})+ \cdots + \sigma^{i}(\ol{a}).$$
  So this case corresponds to showing \eqref{eq:Case-1a,2} with $i'=0$ and $j'=i$.
  In Case (2),
  $$\ol{a}_{i+1}+\cdots + \ol{a}_s + (\ol{a}_j+\cdots + \ol{a}_{s-2})
  =-a-\sigma(\ol{a})-\cdots-\sigma^{i-1}(\ol{a})+\sigma^{j}(\ol{a})+\sigma^{j+1}(\ol{a})+\cdots+\sigma^{s-2}(\ol{a})$$
  $$=\sigma^{j}(\ol{a})+\sigma^{j+1}(\ol{a})+\cdots+\sigma^{s-2+i}(\ol{a})$$
  so the condition that this expressions is non-zero amounts to \eqref{eq:Case-1a,2} with $i'=j$ and $j'=s-2+i$. Note that $j'-i'=s-2+i-j\le s-2$.

  Now the verification of \eqref{eq:Case-1a,2} is the same as for type A in the proof of Proposition \ref{prop:0-toral-abundance}, based on \eqref{eq:stabilizer-of-a-typeD}.

\paragraph{\textbf{Case (1b)}}
  To prove the claim by contradiction, we suppose that
  $$2(\ol{a}_{i+1}+\ol{a}_{i+2}+\cdots + \ol{a}_{s-1})=
  - (\ol{a}+\sigma(\ol{a})+\cdots + \sigma^{i-1}(\ol{a})) + (\sigma^i(\ol a)+\cdots+\sigma^{s-2}(\ol{a}))+\ol{b}=0.$$
  We apply $\sigma^2$ to the equation and subtract it from the original equation, recalling that $\sigma^2(\ol{b})=\ol{b}$. Then
$$ -\ol{a}-\sigma(\ol a) + 2( \sigma^i(\ol a)+\sigma^{i+1}(\ol a)) -\sigma^{s-1}(\ol a)-\sigma^s(\ol a)=0.$$
  This is simplified via \eqref{eq:property-ab-typeD} as $\sigma^i(\ol a)+\sigma^{i+1}(\ol a)=0$, also using that $q$ is odd.
  Hence $\sigma(\ol a)=-\ol a$ and therefore $\sigma^2(\ol{a})=\ol{a}$, which contradicts \eqref{eq:stabilizer-of-a-typeD} (as $s\ge 5$).

\vspace{.1in}

\vspace{.1in}

\paragraph{\textbf{Case (3)}}

 Again suppose that
  $$2(\ol{a}_{i+1}+\cdots + \ol{a}_{s-2}+\ol{a}_s)=
  - (\ol{a}+\sigma(\ol{a})+\cdots + \sigma^{i-1}(\ol{a})) + (\sigma^i(\ol a)+\cdots+\sigma^{s-2}(\ol{a}))-\ol{b}=0.$$
  The equation is the same as in Case (1b) except that the coefficient of $\ol{b}$ has opposite sign, which does not affect the argument, so we reach contradiction in the same way as before.

\section{Calculations for $E_6$} \label{app:E6}

  Here we prove Proposition \ref{prop:0-toral-abundance} for a split group $G$ of type $E_6$.
  That is, we exhibit an elliptic maximal torus $T\subset G$, and a $G$-generic element $X\in \mathfrak t^*$ of depth $-r$ for the same $T$ (as $r$ varies). The notation in the proof of the proposition is maintained.

	We denote the simple coroots of $E_6$ by $\check \alpha_1, \check\alpha_2, \hdots, \check\alpha_6$ as shown in Figure \ref{Figure:E6}.
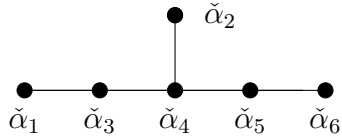
\begin{figure}[ht]
	\centering
	\begin{tikzpicture}
	
	\draw (0,0) -- (3.7,0);
	\draw (0,0) -- (4,0);
	\draw (2,0) -- (2,1);
	
	\draw[fill=black] (0,0) circle(.1);
	\draw[fill=black] (1,0) circle(.1);
	\draw[fill=black] (2,0) circle(.1);
	\draw[fill=black] (2,1) circle(.1);
	\draw[fill=black] (2,1) circle(.1);
	\draw[fill=black] (3,0) circle(.1);
	\draw[fill=black] (4,0) circle(.1);
	
	\node at (0,-.4) {$\check\alpha_1$};
	\node at (2.6,1) {$\check\alpha_2$};
	\node at (1,-.4) {$\check\alpha_3$};
	\node at (2,-.4) {$\check\alpha_4$};
	\node at (3,-.4) {$\check\alpha_5$};	
	\node at (4,-.4) {$\check\alpha_6$};
	
	\end{tikzpicture}
	\caption{Dynkin diagram for (the dual of) $E_6$}	
	\label{Figure:E6}	
\end{figure}

Then the positive coroots of $E_6$ are\footnote{The list coincides with the one in \cite[Planche~V]{Bourbaki-Lie-Ch456} up to identifying roots with coroots. This is fine as the root system of type $E_6$ is self-dual.}
\begin{enumerate}
	\item $\check\alpha_1+\hdots +\check\alpha_j-\check\alpha_2$ for $1 < j  \leq  6$
	\item $\check\alpha_i+\hdots +\check\alpha_j$ for $3 \le i \leq j \leq  6$
	\item $\check\alpha_2+\hdots +\check\alpha_j-\check\alpha_3$ for $3 \leq  j \leq  6$
	\item $\check\alpha_i+\hdots +\check\alpha_4$ for $1 \le i \leq 2$
	\item $\check\alpha_i+\hdots +\check\alpha_j$ for $1 \le i \leq 2 \leq 5 \leq j \leq 6$
	\item $\check\alpha_i+\hdots +\check\alpha_j+\check\alpha_4$ for $1 \le i \leq 2 \leq 5 \leq j \leq 6$
	\item $\check\alpha_2+\check\alpha_3+2\check\alpha_4+2\check\alpha_5+\check\alpha_6$
	\item $\check\alpha_1+\check\alpha_2+\check\alpha_3+2\check\alpha_4+2\check\alpha_5+\check\alpha_6$
	\item $\check\alpha_1+ \check\alpha_2+2\check\alpha_3+2\check\alpha_4+\check\alpha_5$
	\item $\check\alpha_1+ \check\alpha_2+2\check\alpha_3+2\check\alpha_4+\check\alpha_5+\check\alpha_6$
	\item $\check\alpha_1+ \check\alpha_2+2\check\alpha_3+2\check\alpha_4+2\check\alpha_5+\check\alpha_6$
	\item $\check\alpha_1+ \check\alpha_2+2\check\alpha_3+3\check\alpha_4+2\check\alpha_5+\check\alpha_6$
	\item $\check\alpha_1+ 2\check\alpha_2+2\check\alpha_3+3\check\alpha_4+2\check\alpha_5+\check\alpha_6$
\end{enumerate}
(Note that (1)--(4) correspond to coroots of subroot systems of type $A$.)
	
Let $w_h = s_2s_3s_5s_1s_4s_6$, where $s_i$ denotes the reflection corresponding to $\check\alpha_i$. Then $w_h$ is a Coxeter element, hence it has order $\Cox(G)=12$ and its eigenvalues when acting on the complex vector space spanned by the coroots are $\zeta_{12},\zeta_{12}^{4}, \zeta_{12}^{5}, \zeta_{12}^{7}, \zeta_{12}^{8}, \zeta_{12}^{11}$, where $\zeta_{12}$ denotes a primitive (complex) twelfth root of unity (see \cite[3.7.~Table~1 and 3.19.~Theorem]{Hum90}). Hence $w:=w_h^{4}$ is an elliptic element, i.e. does not have any nonzero fixed vector when acting on the above complex vector space. One easily calculates that
	\begin{eqnarray*}
		w(\check\alpha_1)&=& -\check\alpha_1-\check\alpha_2-\check\alpha_3-\check\alpha_4 \\
		w(\check\alpha_2)&=& \check\alpha_1 + \check\alpha_3 +\check\alpha_4 + \check\alpha_5 + \check\alpha_6\\
		w(\check\alpha_3)&=& \check\alpha_1 + \check\alpha_2 + \check\alpha_3 + 2\check\alpha_4 + \check\alpha_5 \\
		w(\check\alpha_4)&=& -\check\alpha_1-\check\alpha_2-2\check\alpha_3-3\check\alpha_4-2\check\alpha_5-\check\alpha_6 \\
		w(\check\alpha_5)&=& \check\alpha_2 + \check\alpha_3 +2\check\alpha_4 + \check\alpha_5 + \check\alpha_6\\
		w(\check\alpha_6)&=& -\check\alpha_2-\check\alpha_4-\check\alpha_5 -\check\alpha_6
	\end{eqnarray*}

	Let $E$ be a cubic Galois extension of $F$. Fix a generator $\sigma$ of $\Gal(E/F)$ and define $f: \Gal(F^{\textup{sep}}/F) \twoheadrightarrow \Gal(E/F) \ra W$ by sending $\sigma\in\Gal(E/F)$ to $w$. As in the proof of Proposition \ref{prop:0-toral-abundance}, $f$ gives rise to (the conjugacy class of) a maximal torus $T$ of $G$. Since $w$ is elliptic, the torus $T$ is elliptic.
	We divide into two cases.

\vspace{.05in}
	
\paragraph{\textbf{Case (1)}: when $F$ does not contain a nontrivial third root of unity.}
 In this case we let $E$ be the unramified cubic extension of $F$.	Let $a \in E$ be an element of valuation zero such that $a+\sigma(a)+\sigma^2(a)=0$ and such that the image  $\bar a$ of $a$ in the residue field $k_E$ of $E$ is a generator for the field extension $k_E/k_F$ (see Lemma \ref{Lemma-a}), and define
$$a_1=a_2=a_4=\sigma(a), \quad a_3=a-2\sigma(a), \quad
		a_5=a+\sigma(a), \quad a_6=-3a-2\sigma(a).
$$
	Then
	\begin{eqnarray*}
		\sigma(a_1)&=& -a_1-a_2-a_3-a_4 ,\\
		\sigma(a_2)&=& a_1 + a_3 +a_4 + a_5 + a_6,\\
		\sigma(a_3)&=& a_1 + a_2 + a_3 + 2a_4 + a_5, \\
		\sigma(a_4)&=& -a_1-a_2-2a_3-3a_4-2a_5-a_6, \\
		\sigma(a_5)&=& a_2 + a_3 +2a_4 + a_5 + a_6,\\
		\sigma(a_6)&=& -a_2-a_4-a_5 -a_6.
	\end{eqnarray*}
	Hence the linear functional $X$ on $\ft(E)_{y,n+1}$ defined by $X(\varpi_F^{n+1}H_{\check \alpha_i})=a_i$ descends to a linear functional on $\ft_{y,n+1}$.

	We claim that $\{1, \bar a, \sigma(\bar a)\}$ is a $k_F$-basis for $k_E$. To this end, suppose that there exist $c_1, c_2, c_3 \in k_F$ such that $c_1+ c_2 \bar a + c_3 \sigma(\bar a) =0$. Applying $\sigma$, we have $c_1+c_2\sigma(\bar a)+c_3(-\bar a - \sigma(\bar a))=0$. Taking the difference of the two equations, we obtain
$$(c_2-2c_3)\sigma(\bar a)=(c_2+c_3)\bar a.$$ If $c_2-2c_3\neq 0$ then $\sigma(\bar a)=c\bar a$ with $c=(c_2+c_3)/(c_2-2c_3)$, thus $\bar a=\sigma^3(\bar a)=c^3 \bar a$.
Since $F$ does not contain any nontrivial third root of unity, neither does $k_F$, and hence $c=1$. But then $\sigma(\bar a)=\bar a$, contradicting $k_E=k_F(\bar a)\neq k_F$.
Therefore $c_2-2c_3=0$, thus also $c_2+c_3=0$ as $\bar a\neq 0$. Since $p>\Cox(G)>3$, it follows that $c_2=c_3=0$. Hence $c_1=0$ as well, proving the desired linear independence of $\{1, \bar a, \sigma(\bar a)\}$ over $k_F$.
	
	Now  using the linear independence of $\bar a$ and $\sigma(\bar a)$ together with the explicit formulas for the (positive) coroots of $E_6$ above, it is easy to check that $v(X(H_{\check \alpha}))=-(n+1)$ for all coroots $\check \alpha$ of $G_E$ with respect to $T_E$. Hence $X$ is $G$-generic of depth $n+1$.
	
\vspace{.05in}

	\paragraph{\textbf{Case (2)}: when $F$ contains a nontrivial third root of unity $\zeta$.}
	In this case let $E$ be the totally ramified extension $F(\varpi_E)$ for a root $\varpi_E$ of the equation $x^3-\varpi_F=0$. As our choice of $\zeta$ and $\varpi_E$ is flexible, we may assume that $\sigma(\varpi_E)=\zeta \varpi_E$. 	
	We set
	$$a_1=2, \quad a_2=a_4=a_5=1, \quad a_3=-4-2\zeta, \quad a_6=3\zeta .$$
	Then
	\begin{eqnarray*}
		\zeta a_1 &=& -a_1-a_2-a_3-a_4, \\
		\zeta a_2  &=& a_1 + a_3 +a_4 + a_5 + a_6,\\
		\zeta a_3 &=& a_1 + a_2 + a_3 + 2a_4 + a_5, \\
		\zeta a_4 &=& -a_1-a_2-2a_3-3a_4-2a_5-a_6, \\
		\zeta a_5 &=& a_2 + a_3 +2a_4 + a_5 + a_6,\\
		\zeta a_6 &=& -a_2-a_4-a_5 -a_6.
	\end{eqnarray*}
	Thus the linear functional $X$ on $\ft(E)_{y,n+1/3}$ defined by $X(\varpi_E^{3n+1}H_{\check \alpha_i})=a_i$ descends to a linear functional on $\ft_{y,n+1/3}$. It remains to check that $v(X(H_{\check\alpha}))=-(n+\frac{1}{3})$ for all coroots $\check\alpha$ of $G_E$ with respect to $T_E$. Note that it suffices to consider the positive coroots. Using the explicit formulas above, we obtain that for all positive coroots $\check \alpha$,
	\begin{eqnarray*}
X(\varpi_E^{3n+1}H_{\check\alpha}) & \in &
\{ 1, 2, 3, -2-4 \zeta \}\cup \{ i-2 \zeta \, | \,  -4 \leq i \leq 1\} \cup  \{ i- \zeta \, | \,  -3 \leq i \leq 1\}\\
&  & \cup \{ i + \zeta \, | \,  -2 \leq i \leq 3\} \cup \{ i+3 \zeta \, | \,  0 \leq i \leq 3 \} \, .
\end{eqnarray*}
	If $\zeta \notin \bF_p$, then $p>\Cox(G)=12 >4$  implies that the image $\overline{X(\varpi_E^{3n+1}H_{\check\alpha})}$ of $X(\varpi_E^{3n+1}H_{\check\alpha}) \in \cO_E$ in the residue field $k_E$ is non-zero, hence $v(X(\varpi_E^{3n+1}H_{\check\alpha}))=0$ as desired.
	If $\zeta \in \bF_p$, then one can treat each case separately and show that $\overline{X(\varpi_E^{3n+1}H_{\check\alpha})} \neq 0$ using that $p>12$. More precisely it is obvious that $\overline{X(\varpi_E^{3n+1}H_{\check\alpha})} \neq 0$ when the value of $X(\varpi_E^{3n+1}H_{\check\alpha})$ is an integer, a multiple of $\zeta$, or of the form $\pm 1 \pm \zeta$. In the remaining cases, we have the form $X(\varpi_E^{3n+1}H_{\check\alpha})=c_1 + c_2 \zeta$, and one verifies that $c_1^3 \not\equiv -c_2^3 \mod p$ in each case so that $\overline{c_1 + c_2 \zeta}\neq 0$.
	
We conclude that $X$ is $G$-generic of depth $n+\frac{1}{3}$.

\section{A note on Galois representations associated to automorphic forms (by Vytautas Pa\v{s}k\=unas)}\label{app:Galois-reps}

\markright{\textsc{Appendix C. A note on Galois representations associated to automorphic forms}}
\markleft{\textsc{Vytautas Pa\v{s}k\=unas}}

 The aim of this note is to explain how to deduce Theorem 3.3.3 (in the main article) by replacing the use of Theorem 3.2.1 in its proof
by the density results proved in \cite{EP}.  We put ourselves in the setting of the proof of Theorem 3.3.3. In particular,  $F$ is a totally real field, $E$ is a totally imaginary quadratic extension of $F$,
 $S$ is a finite set of places of $F$ containing all the places above $p$ and $\infty$, $S_E$ is the set of places of $E$ above $S$,  $E_S$ is the maximal extension of $E$ unramified outside $S$,
   $G$ is a unitary group over $F$ which is an outer form of $\GL_n$ with respect to the quadratic extension $E/F$ such that $G$ is quasi-split at all finite places and anisotropic at all infinite places. We assume that all the places of $F$ above $p$ split in $E$. This implies that $G(F\otimes_{\QQ} \Qp)$
is isomorphic to a product of $\GL_n(F_v)$ for $v\mid p$. Let $U^p$ be a compact  open subgroup of $G(\mathbb A^{\infty, p}_F)$. If $U_p$ is a compact open subgroup  of $G(F\otimes_{\QQ} \Qp)$ then the double coset space
  $$Y(U^p U_p):= G(F)\backslash G(\mathbb A_F)/ U^p U_p G(F \otimes_{\QQ} \RR)^{\circ}$$
is a finite set. From our point of view the key objects are the completed cohomology
$$ \Htilde(U^p):= \varprojlim_{m} \varinjlim_{U_p} H^0(Y(U^p U_p), \ZZ/p^m \ZZ),$$
where the inner limit is taken over all open compact subgroups $U_p$ and the big Hecke algebra:
\begin{equation}\label{defnT}
\mathbb T^S_{\mathrm{Spl}}(U^p):= \varprojlim_{m, U_p} \mathbb T^S_{\mathrm{Spl}}(U^p U_p, \ZZ/p^m \ZZ),
\end{equation}
where  $\mathbb T^S_{\mathrm{Spl}}(U^p U_p, \ZZ/p^m \ZZ)$ is the image the algebra $\mathbb T^S_{\mathrm{Spl}}$, defined in the proof
of Theorem 3.3.3, in $\End_{\Zp}( H^0(Y(U^p U_p), \ZZ/p^m \ZZ))$. This algebra is denoted by $\mathbb T'$ in \cite[\S 5.2]{EP}.
We use the projective limit to define the topology on $\mathbb T^S_{\mathrm{Spl}}(U^p)$, which makes it into a profinite ring. We also note that
$H^0(Y(U^p U_p), \ZZ/p^m \ZZ)$ is just a space of $\ZZ/p^m \ZZ$-valued functions on $Y(U^pU_p)$ and so coincides with $M(U^p U_p, \ZZ/p^m \ZZ)$
in \S3.

We will show that $\mathbb T^S_{\mathrm{Spl}}(U^p)$ is a noetherian semi-local ring and will attach a Galois representation
of $\Gal(E_S/E)$ to each maximal ideal of $\mathbb T^S_{\mathrm{Spl}}(U^p)[1/p]$, assuming the result of Clozel recalled in Theorem 3.3.1.
This is slightly more general than Theorem 3.3.3, as we allow maximal ideals, which do not correspond to the classical automorphic forms, and we do
not have a restriction on the prime $p$, as we work with Bushnell--Kutzko types.

We may identify  $\Htilde(U^p)$ (resp. $\Htilde(U^p)_{\Qp}$)  with the space of continuous $\Zp$-valued (resp. $\Qp$-valued) functions on the profinite set $$Y(U^p):=G(F)\backslash G(\mathbb A_F)/ U^p  G(F \otimes_{\QQ} \RR)^{\circ}\cong\varprojlim_{U_p} Y(U^p U_p).$$
 The action of $G(F\otimes_{\QQ} \Qp)$ on $Y(U^p)$ makes $\Htilde(U^p)_{\Qp}$ into an admissible unitary
$\Qp$-Banach space representation of $G(F\otimes_{\QQ} \Qp)$ with unit ball equal to $\Htilde(U^p)$. We fix an open pro-$p$ subgroup $K$ of $G(F\otimes_{\QQ} \Qp)$, which acts freely on $Y(U^p)$ with finitely many orbits. This enables us to identify
$\Htilde(U^p)$ as a representation of $K$ with a direct sum of  finitely many copies of $\mathcal C(K, \Zp)$, the space of continuous functions from $K$ to
$\Zp$ on which $K$ acts via right translations. Thus the Schikhof dual $M:=\Hom^{\cont}_{\Zp}( \Htilde(U^p), \Zp)$ is a free $\Zp[[K]]$-module of finite rank
and we may apply the results of \cite{EP} to $M$.
We point out that it follows from the Schikhof duality that $\Htilde(U^p)_{\Qp}$ is  isometric to the Banach space
representation denoted by $\Pi(M)$ in \cite[Lem.\,2.3]{EP}.

In \cite[\S 3.4]{EP} we construct a countable family $\{V_i\}_{i\in \mathbb N}$ of smooth absolutely irreducible representations of $K$
defined over a finite extension of $\Qp$ such that if $\pi$ is a smooth irreducible $\Qpbar$-representation of $G(F\otimes_{\QQ} \Qp)$ and
$\Hom_K(V_i, \pi)$ is non-zero then $\pi$ is a supercuspidal representation.
(We actually work with $\GL_n(F_v)$ with $v\mid p$, but the argument readily adapts to the product of such groups by taking tensor products.) Moreover, the
evaluation map
\begin{equation}\label{dense}
\ev: \bigoplus_{i\ge 1} \Hom_K(V_i, \Htilde(U^p)_{\Qp})\otimes V_i \rightarrow \Htilde(U^p)_{\Qp}
\end{equation}
has dense image, see Proposition 3.26, together with Lemmas 2.3 and 2.17 in \cite{EP}. This density result is the key input in this note.

The representations  $V_i$ are obtained as direct summands of inductions of certain characters of open subgroups of $K$, analogous to the characters $(U_{p,m}, \psi\circ \lambda_m)$ in the main text, but  constructed using Bushnell--Kutzko theory of types. This theory is not available for every reductive group, but it does not impose any restrictions on the prime $p$.
To orientate the reader we point out that  we may identify
$\Hom_K(V_i, \Htilde(U^p)_{\Qp})$ with $M(U^p K, V_i^*)$ in \S3 by sending $\varphi$ to the function $f:Y(U^p K)\rightarrow V^*_i$, which satisfies $f(x)(v)= \varphi(v)(x)$, for all $x\in Y(U^p K)$ and $v\in V_i$. As explained in \S3, see also \cite[Prop.\,3.2.4]{em0},
the space $\Hom_K(V_i, \Htilde(U^p)_{\Qp})$ is related to the space of automorphic forms on $G$, the action of $\mathbb T^S_{\mathrm{Spl}}(U^p)[1/p]$
on this finite dimensional vector space is semi-simple and the maximal ideals in the support of $\Hom_K(V_i, \Htilde(U^p)_{\Qp})$ correspond
to certain classical automorphic forms, such that the associated automorphic representations are supercuspidal at places above $p$.

For $k\ge 1$ let $A_k$ be the image of $\oplus_{i=1}^k \Hom_K(V_i, \Htilde(U^p)_{\Qp})\otimes V_i$ in  $\Htilde(U^p)_{\Qp}$and
let $A_{\infty}$ be the image of \eqref{dense}.  Let $\mathfrak a_k$ be the $\mathbb T^S_{\mathrm{Spl}}(U^p)$-annihilator of $A_k$.
Each $\Zp$-algebra homomorphism $x: \mathbb T^S_{\mathrm{Spl}}(U^p)/\mathfrak a_k \rightarrow \Qpbar$ will correspond to a set
of Hecke eigenvalues appearing in $\Hom_K(V_i, \Htilde(U^p)_{\Qp})\otimes\Qpbar$ for some $1\le i \le k$, and hence will correspond to a classical
automorphic form which, by construction of $V_i$, will be supercuspidal at all places above $p$. Hence, to such
$x$ we may attach a Galois representation $\rho_x: \Gal(E_S/E)\rightarrow \GL_n(\Qpbar)$, using the results of Clozel. Moreover,
$(\mathbb T^S_{\mathrm{Spl}}(U^p)/\mathfrak a_k)[1/p]$ is semi-simple.

\begin{lemma}\label{open_max_ideals} If  $U_p$ is  an open pro-$p$ subgroup of
$G(F\otimes_{\QQ} \Qp)$,  then the open maximal ideals of $\mathbb T^S_{\mathrm{Spl}}(U^p)$ coincide with  the maximal ideals of $\mathbb T^S_{\mathrm{Spl}}(U^pU_p, \ZZ/p\ZZ)$. In particular,  $\mathbb T^S_{\mathrm{Spl}}(U^p)$
has only finitely many open maximal ideals.
\end{lemma}
\begin{proof} The transition maps in \eqref{defnT} are surjective and thus a maximal ideal $\mathfrak m$ of  $\mathbb T^S_{\mathrm{Spl}}(U^p)$ is open if and only if it
is equal to the preimage of a maximal ideal of  $\mathbb T^S_{\mathrm{Spl}}(U^pU'_p, \ZZ/p^m\ZZ)$ under the
surjection $\mathbb T^S_{\mathrm{Spl}}(U^p)\twoheadrightarrow \mathbb T^S_{\mathrm{Spl}}(U^pU_p', \ZZ/p^m\ZZ)$, for some open subgroup
$U'_p$ and $m\ge 1$. Using the surjectivity of the transition maps we may assume that $U_p'$ is contained in $U_p$.
Since the action of  $\mathbb T^S_{\mathrm{Spl}}(U^pU_p', \ZZ/p^m\ZZ)$ on $H^0(Y(U^p U_p'), \ZZ/p^m \ZZ)$ is faithful by definition,
we conclude that the localisation $H^0(Y(U^p U'_p), \ZZ/p^m \ZZ)_{\mm}$ is non-zero. Since the module is $p$-torsion and $U_p$ is pro-$p$
the $U_p$-invariants of its reduction modulo $p$ are non-zero. Since these operations commute with localisation, we conclude that
$H^0(Y(U^p U_p), \ZZ/p \ZZ)_{\mm}$ is non-zero and so $\mm$ is a maximal ideal of $\mathbb T^S_{\mathrm{Spl}}(U^pU_p, \ZZ/p\ZZ)$.
\end{proof}

If $\mm$ is an open maximal ideal of $\mathbb T^S_{\mathrm{Spl}}(U^p)$ we let $\Htilde(U^p)_{\mm}$ and $\mathbb T^S_{\mathrm{Spl}}(U^p)_{\mm}$
be the $\mm$-adic completions of $\Htilde(U^p)$ and $\mathbb T^S_{\mathrm{Spl}}(U^p)$, respectively.
It follows from the Chinese remainder theorem applied at each finite level that
\begin{equation}\label{Ch_rem}
\Htilde(U^p)\cong \prod_{\mm} \Htilde(U^p)_{\mm}, \quad \mathbb T^S_{\mathrm{Spl}}(U^p)\cong \prod_{\mm} \mathbb T^S_{\mathrm{Spl}}(U^p)_{\mm},
\end{equation}
where the (finite) product is taken over all open maximal ideals of $\mathbb T^S_{\mathrm{Spl}}(U^p)$.
\begin{remark} It follows from \eqref{Ch_rem} that the completion of $\Htilde(U^p)$ and $\mathbb T^S_{\mathrm{Spl}}(U^p)$ at an open maximal ideal $\mm$ coincides with the localisation,
because inverting an element of $\mathbb T^S_{\mathrm{Spl}}(U^p)$, which maps to $1$ in  $\mathbb T^S_{\mathrm{Spl}}(U^p)_{\mm}$ and to $0$ in other completions, kills off the other factors.
\end{remark}
For $k\ge 1$ let $A_k^0:= A_{k} \cap \Htilde(U^p)$ and let $A_{\infty}^0:=A_{\infty} \cap \Htilde(U^p)$. Then $A^0_k/p^m$ injects into $A^0_{\infty}/p^m$ and, since $A_{\infty}$ is dense in $\Htilde(U^p)_{\Qp}$, we have
\begin{equation}\label{reduce}
\Htilde(U^p)/ p^m \cong A_{\infty}^0/ p^m \cong \varinjlim_{k\ge 1}A^0_k/p^m , \quad \forall m\ge 1.
\end{equation}
 In particular, there exists $k$, such that $A^0_k/p$ contains $H^0(Y(U^p K), \ZZ/p\ZZ)$.  For such $k$, $\mathfrak a_k$ will also annihilate $H^0(Y(U^p K), \ZZ/p\ZZ)$.
 Hence, there is a surjection  $\mathbb T^S_{\mathrm{Spl}}(U^p)/\mathfrak a_k \twoheadrightarrow  \mathbb T^S_{\mathrm{Spl}}(U^pK, \ZZ/p\ZZ)$.
 It follows from Lemma \ref{open_max_ideals} that the maximal ideals $\mathbb T^S_{\mathrm{Spl}}(U^p)/\mathfrak a_k$ coincide with the open maximal ideals of $\mathbb T^S_{\mathrm{Spl}}(U^p)$. We note that, since $A_k^0$ is a finite free $\Zp$-module, the same applies to $\mathbb T^S_{\mathrm{Spl}}(U^p)/\mathfrak a_k$ and to its localisation $(\mathbb T^S_{\mathrm{Spl}}(U^p)/\mathfrak a_k)_{\mm}$.
In particular, the quotient topology on  $\mathbb T^S_{\mathrm{Spl}}(U^p)/\mathfrak a_k$ coincides
with the $p$-adic one and every maximal ideal of $\mathbb T^S_{\mathrm{Spl}}(U^p)/\mathfrak a_k$ is open.

 \begin{lemma}\label{correct} Let  $\rho_x: \Gal(E_S/E)\rightarrow \GL_n(\Qpbar)$ be the Galois representation corresponding to  a $\Zp$-algebra homomorphism $x: \mathbb (\mathbb T^S_{\mathrm{Spl}}(U^p)/\mathfrak a_k)_{\mm} \rightarrow \Qpbar$. Then the function $D_x:\Zp[\Gal(E_S/E)]\rightarrow \Qpbar$, $a\mapsto
 \det(\rho_x(a))$ takes values in the image of $x$. Moreover, there is a semi-simple Galois representation $\rhobar: \Gal(E_S/E)\rightarrow \GL_n(\Fpbar)$
  such that the function $$\overline{D}: \Zp[\Gal(E_S/E)]\rightarrow \Fpbar, \quad a\mapsto \det(\rhobar(a))$$ takes values in the residue field $\kappa(\mm)$ and
  $$ \overline{D}(a)\equiv D_x(a)\pmod{\mm}, \quad \forall a \in \Zp[\Gal(E_S/E)]$$
  and for all $\Zp$-algebra homomorphisms $x: \mathbb (\mathbb T^S_{\mathrm{Spl}}(U^p)/\mathfrak a_k)_{\mm} \rightarrow \Qpbar$.
 \end{lemma}
\begin{proof} Let $\rhobar$ be the semisimplification of the reduction modulo $p$ of a $\Gal(E_S/E)$-stable lattice in $\rho_x$, for some $x$.
As explained in the proof of Theorem 3.3.3 using density arguments it is enough to check the assertions for $a=1+ t \Frob_w$ with $w\not\in S_E$ split over $F$. The assertion then follows from equation (3.3.6), which expresses the characteristic polynomial of $\rho_x(\Frob_w)$ in terms of
Hecke operators.
\end{proof}

The function  $\overline{D}: \Zp[\Gal(E_S/E)]\rightarrow \kappa(\mm)$  is a continuous $n$-dimensional determinant in the sense of Chenevier \cite{chen}. The universal deformation ring $R_{\overline{D}}$ of $\overline{D}$ is a complete local noetherian
 algebra over the ring of Witt vectors of $\kappa(\mm)$ by \cite[Prop.\,3.3, 3.7, Ex.\,3.6]{chen}. It follows from Lemma \ref{correct} that $D_x$ is a deformation of $\overline{D}$ and hence induces a map $R_{\overline{D}}\rightarrow \Qpbar$.
 By taking the product over all $\Zp$-algebra homomorphisms $x: (\mathbb T^S_{\mathrm{Spl}}(U^p)/\mathfrak a_k)_{\mm} \rightarrow \Qpbar$
we obtain a continuous map
 \begin{equation}\label{map}
 R_{\overline{D}} \rightarrow \prod_{x} \Qpbar\cong (\mathbb T^S_{\mathrm{Spl}}(U^p)/\mathfrak a_k)_{\mm}\otimes_{\Zp} \Qpbar,
 \end{equation}
 where the last isomorphism follows since $(\mathbb T^S_{\mathrm{Spl}}(U^p)/\mathfrak a_k)_{\mm}[1/p]$ is semi-simple and finite over $\Qp$.
 \begin{lemma}\label{image} The map \eqref{map} induces a surjection $R_{\overline{D}}\twoheadrightarrow (\mathbb T^S_{\mathrm{Spl}}(U^p)/\mathfrak a_k)_{\mm}$.
 \end{lemma}
 \begin{proof} This is proved in the course of the proof of Theorem 3.3.3 - let $R'$ be the image of \eqref{map} and let $D'$ be the tautological deformation of $\overline{D}$ to $R'$.
 Then $R'$ is equal to the closure of the subring generated by the
 coefficients  of $D'(1+t \Frob_w)$ for all  $w\not\in S_E$
 and these are contained in  $(\mathbb T^S_{\mathrm{Spl}}(U^p)/\mathfrak a_k)_{\mm}$. Since the coefficients of $D'(1+t \Frob_w)$ can be expressed in terms of
 Hecke operators, they  are contained in $R'$. Since these generate $(\mathbb T^S_{\mathrm{Spl}}(U^p)/\mathfrak a_k)_{\mm}$ the map
 is surjective. We note that  in the case of modular forms  the analogous argument appears in \cite[\S 2.2]{carayol}.
 \end{proof}

 \begin{theorem} The maps $R_{\overline{D}}\twoheadrightarrow (\mathbb T^S_{\mathrm{Spl}}(U^p)/\mathfrak a_k)_{\mm}$ for $k\ge 1$ induce a surjection
 $R_{\overline{D}}\twoheadrightarrow \mathbb T^S_{\mathrm{Spl}}(U^p)_{\mm}$. In particular, $\mathbb T^S_{\mathrm{Spl}}(U^p)_{\mm}$ is noetherian and
 for every $\Zp$-algebra homomorphism $x: \mathbb T^S_{\mathrm{Spl}}(U^p)_{\mm}\rightarrow \Qpbar$ there is a continuous semi-simple representation
 $\rho_x: \Gal(E_S/ E) \rightarrow \GL_n(\Qpbar)$ such that
\begin{equation}\label{char_poly_frob}
 \det( 1+ t \rho_x(\Frob_w))= \sum_{i=0}^n t^i N(w)^{i(i-1)/2} x(T_w^{(i)}), \quad \forall w\not \in S_E.
 \end{equation}
 \end{theorem}
 \begin{proof} Using Lemma \ref{image} we obtain a continuous action of $R_{\overline{D}}$ on
 $A^0_{k, \mm}$ compatible with the inclusions
 $A^0_{k, \mm}\subset A^0_{k+1, \mm}$, for $k\ge 1$. Thus for each $m\ge 1$ we obtain a continuous action of $R_{\overline{D}}$ on
 $$\varinjlim_{k\ge 1} (A^0_{k}/ p^m)_{\mm} \cong (A^0_{\infty}/ p^m)_{\mm} \cong (\Htilde(U^p)/ p^m)_{\mm},$$
 where the last isomorphism follows  from \eqref{reduce}.
 By passing to the projective limit we obtain a continuous action of $R_{\overline{D}}$ on $\Htilde(U^p)_{\mm}$, which factors through the action of
 $\mathbb T^S_{\mathrm{Spl}}(U^p)_{\mm}$. Let $R$ be the image of the map $R_{\overline{D}}\rightarrow \mathbb T^S_{\mathrm{Spl}}(U^p)_{\mm}$.
 Since $R_{\overline{D}}$ is a complete local  noetherian ring with a finite residue field it is compact. Since $\mathbb T^S_{\mathrm{Spl}}(U^p)_{\mm}$ is profinite it is Hausdorff and hence $R$ is closed in $\mathbb T^S_{\mathrm{Spl}}(U^p)_{\mm}$.  But $R$ is also dense, since by construction $R$ surjects onto $\mathbb T^S_{\mathrm{Spl}}(U^pU_p, \ZZ/p^m \ZZ)_{\mm}$
for all $U_p$ and $m\ge 1$. Hence, $R= \mathbb T^S_{\mathrm{Spl}}(U^p)_{\mm}$.

 Let $D_R$ be the tautological deformation of $\overline{D}$ to $R$. Then
$$D_R( 1+ t \Frob_w )= \sum_{i=0}^n t^i N(w)^{i(i-1)/2} T_w^{(i)}, \quad \forall w\not\in S_E,$$
as this relation holds for all $D_x$ and hence in
$(\mathbb T^S_{\mathrm{Spl}}(U^p)/\mathfrak a_k)_{\mm}$ by construction. If $x: R\rightarrow \Qpbar$
is a homomorphism of $\Zp$-algebras then by \cite[Thm.\,2.12]{chen} there is a unique semi-simple representation $\rho_x: \Gal(E_S/E)\rightarrow \GL_n(\Qpbar)$
such that
$$x(D_R(1+t g))=\det (1+ t\rho_x(g)), \quad \forall g\in \Gal(E_S/E).$$
Since $x\circ D_R$ is continuous, the representation $\rho_x$ is continuous by \cite[Ex.\,2.34]{chen}.
 \end{proof}
 \begin{remark} It follows from \eqref{Ch_rem}  and the theorem above that $\mathbb T^S_{\mathrm{Spl}}(U^p)$ is noetherian and we may attach a Galois representation satisfying \eqref{char_poly_frob} to any $\Zp$-algebra homomorphism $x:\mathbb T^S_{\mathrm{Spl}}(U^p)\rightarrow \Qpbar$.
 \end{remark}

 \textbf{Acknowledgements}. I would like to thank the organisers  Brandon Levin, Rebecca Bellovin, Matthew Emerton and David Savitt for inviting me to the  Banff Workshop \textit{Modularity and Moduli Spaces}, which took place in Oaxaca in October 2019, where I first heard Jessica Fintzen talk about her joint work with Sug Woo Shin. I would like to thank Jessica Fintzen and Sug Woo Shin for kindly agreeing to include this note as an appendix to their paper and
 for their comments on earlier drafts.

\renewcommand\refname{References for Appendix \ref{app:Galois-reps}}

	\section{A non-explicit proof of the existence of omni-supercuspidal types (by Rapha\"el Beuzart-Plessis)}

\markright{\textsc{Appendix D. A non-explicit proof of the existence of omni-supercuspidal types}}
\markleft{\textsc{Rapha\"el Beuzart-Plessis}}

	The goal of this appendix is to give another proof of Theorem C on the existence of (sufficiently many) omni-supercuspidal types for reductive groups over local fields of characteristic zero. The argument is very close to that in the paper \cite{BP}, whose purpose was only to show the existence of one supercuspidal representation. It doesn't use type theory but instead a notion of cusp forms due to Harish-Chandra \cite{HC}. As in \cite{BP} and contrary to the proof given in the main article, this alternative approach is non-explicit in that we do not exhibit any omni-supercuspidal type. On the other hand, one advantage of the method is that we are able to bypass the assumption on the residual characteristic made in Theorem C. We end this appendix by explaining why this theorem cannot hold for local fields of positive characteristic.
	
	We now fix some notations and recall the result. Let $F$ be a local non-Archimedean field of characteristic zero with residual characteristic $p$, $\cO\subset F$ be its ring of integers and $\cP\subset \cO$ be the maximal ideal. We fix a uniformizer $\varpi\in \cP$. Let $G$ be a connected reductive group defined over $F$. For our purpose, it will be convenient to adopt a slightly more general definition of {\em omni-supercuspidal type} than in the main body of the paper. More precisely, in this appendix, an {\em omni-supercuspidal type} will be a pair $(U,\lambda)$ where $U$ is a compact-open subgroup of $G(F)$ and $\lambda:K\to A$ is a smooth character valued in an (arbitrary) abelian group $A$ such that for every nontrivial character $\chi: A\to \mathbb{C}^\times$, $(U,\chi\circ \lambda)$ is a supercuspidal type. Note that if $(U,\lambda)$ is an omni-supercuspidal type with $\lambda: U\to A$ and $\mu:A\to B$ is a surjective morphism of abelian groups then $(U,\mu\circ \lambda)$ is also an omni-supercuspidal type. Thus, if $B=\mathbb{Z}/p^m\mathbb{Z}$ for some $m\geqslant 1$ then $(U,\mu\circ \lambda)$ is an omni-supercuspidal type of level $p^m$ as defined in the paper. The goal of this appendix is to prove the following result.
	
	\begin{theorem}\label{theo1 appendix}
		For every open subgroup $V\subset G(F)$ and $m\geqslant 1$, there exists an omni-supercuspidal type $(U_m,\lambda_m)$ of level $p^m$ with $U_m\subset V$.
	\end{theorem}
	
	\subsection{Proof of Theorem \ref{theo1 appendix}}
	
	Let $\mathfrak{g}$ be the Lie algebra of $G$, $\mathfrak{g}^*$ be the dual of $\mathfrak{g}$ and $\langle .,.\rangle$ for the canonical pairing between $\mathfrak{g}$ and $\mathfrak{g}^*$. An element $Y\in \mathfrak{g}^*(F)$ is called {\emph elliptic} if for every proper parabolic subalgebra $\mathfrak{p}\subsetneq \mathfrak{g}$ with nilpotent radical $\mathfrak{n}$, we have $Y\notin \mathfrak{n}^\perp(F)$ where $\mathfrak{n}^\perp$ stands for the orthogonal of $\mathfrak{n}$ in $\mathfrak{g}^*$. We denote by $\mathfrak{g}^*(F)_{\elli}$ the open subset of elliptic elements in $\mathfrak{g}^*(F)$.
	
	\begin{remark}
	Choosing a $G$-invariant nondegenerate bilinear form $B:\mathfrak{g}\times \mathfrak{g}^*\to \mathbb{G}_a$, we get an equivariant isomorphism $\mathfrak{g}^*\simeq \mathfrak{g}$ identifying $\mathfrak{g}^*(F)_{\elli}$ with the usual subset of elliptic elements in $\mathfrak{g}(F)$ that is: elements $X\in \mathfrak{g}(F)$ that do not belong to any proper parabolic subalgebra. In particular, if $T\subset G$ is an elliptic maximal torus (that is; anisotropic modulo the center of $G$), regular elements in the Lie algebra $\mathfrak{t}(F)$ are elliptic in this sense. Since such a maximal torus always exists \cite[Theorem 6.21]{PR}, we deduce that $\mathfrak{g}^*(F)_{\elli}$ is nonempty. 
	\end{remark} 
	
	Let $\psi: F\to \mathbb{C}^\times$ be a continuous additive character which is trivial on $\cO$ but not on $\cP^{-1}$. By a lattice in a finite-dimensional $F$-vector space $V$, we mean a finitely generated $\cO$-submodule $L\subset V$ such that $L\otimes_{\cO} F=V$. For every lattice $L\subset \mathfrak{g}(F)$ we denote by $L^\perp\subset \mathfrak{g}^*(F)$ the lattice of elements $Y\in \mathfrak{g}^*(F)$ such that $\langle Y,X\rangle\in \cO$ for every $X\in L$. We define similarly $(L^*)^\perp\subset \mathfrak{g}(F)$ for every lattice $L^*\subset \mathfrak{g}^*(F)$.
	
	\begin{lemma}\label{lem1 appendix}
		Let $Y_1\in \mathfrak{g}^*(F)_{\elli}$. Then, there exists a lattice $L_0\subset \mathfrak{g}(F)$ such that
		$$Y_1+L_0^\perp\subset \mathfrak{g}^*(F)_{\elli} \;\;\; \mbox{ and } \;\;\; \langle Y_1,L_0\rangle=\cP^{-1}.$$
	\end{lemma}
	
	\begin{proof}
		First, consider the case where $G/Z(G)$ is anisotropic. Then, $\mathfrak{g}^*(F)_{\elli}=\mathfrak{g}^*(F)$ and any lattice $L_0\subset \mathfrak{g}(F)$ such that $\langle Y_1,L_0\rangle=\cP^{-1}$ satisfies the requirement.
		
		Assume now that $G/Z(G)$ is not anisotropic. In particular, $0\in \mathfrak{g}^*(F)$ is not elliptic. As $\mathfrak{g}^*(F)_{\elli}$ is open, there exists a lattice $L_1^*\subset \mathfrak{g}^*(F)$ such that $Y_1+L_1^*\subset \mathfrak{g}^*(F)_{\elli}$. Set $L_2^*=\cP Y_1+L_1^*$ (again a lattice in $\mathfrak{g}^*(F)$) and $L_0=(L_2^*)^\perp$. Then, we have
		$$\displaystyle Y_1+L_0^\perp=Y_1+L_2^*=(1+\cP)\cdot (Y_1+L_1^*)\subseteq \mathfrak{g}^*(F)_{\elli}.$$
		On the other hand, since $\cP Y_1\subset L_2^*=L_0^\perp$, we have $\langle Y_1,L_0\rangle\subset \cP^{-1}$ and it cannot happen that $\langle Y_1,L_0\rangle\subset \cO$ as otherwise $0\in Y_1+L_0^\perp$ and $0$ is not elliptic. Therefore, $\langle Y_1,L_0\rangle=\cP^{-1}$ and we see that $L_0$ has the required property.
	\end{proof}
	
	From now on, we choose $Y_1$ and $L_0$ as in the previous lemma. For every integer $n\geqslant 0$, we set $L_n=\varpi^n L_0$ and $Y_n=\varpi^{-n+1}Y_1$. Let $\omega$ be a neighborhood of $0$ in $\mathfrak{g}(F)$ on which the exponential map $\exp: \omega\to G(F)$ is well-defined and one-to-one. Then, there exists an integer $n_0\geqslant 1$ such that for every $n'\geqslant n\geqslant n_0$, $L_n \subset \omega$, $K_n=\exp(L_n)$ is a compact-open subgroup of $G(F)$ and $K_{n'}$ is normal in $K_n$. For $m\geqslant 1$ and $n\geqslant n_0$, we define a map
	$$\displaystyle \lambda_{n,m}:K_n\to \cP^{-m}/\cO$$
	by
	$$\displaystyle \exp(X)\mapsto \langle Y_{n+m},X\rangle+\cO,\;\; X\in L_n.$$
	
	From now on we fix $m\geqslant 1$.
	
	\begin{lemma}
		There exists $n_1\geqslant n_0$ such that for every $n\geqslant n_1$, $\lambda_{n,m}$ is a character.
	\end{lemma}
	
	\begin{proof}
		By the Baker-Campbell-Hausdorff formula, there exists $r\geqslant 0$ such that for every $n\geqslant n_0$ and $X,Y\in L_n$ we have
		$$\displaystyle \exp(X)\exp(Y)\in \exp(X+Y+\varpi^{2n-r}L_0).$$
		To conclude, it suffices to choose $n_1\geqslant n_0$ such that
		$$\displaystyle \langle Y_m,\varpi^{n_1-r}L_0\rangle\subset \cO.$$
	\end{proof}
	
	As the sequence $(K_n)_{n\geqslant n_0}$ form a decreasing basis of neighborhoods of $1$ and there exists a surjective homomorphism $\cP^{-m}/\cO\to \bZ/p^{m'}\bZ$ (e.g. induced from the trace of $F$ over $\mathbb{Q}_p$), we see that the following proposition implies Theorem \ref{theo1 appendix}.
	
	\begin{proposition}\label{prop1 appendix}
		There exists $n_2\geqslant n_1$ such that for $n\geqslant n_2$, the pair $(K_n,\lambda_{n,m})$ is an omni-supercuspidal type.
	\end{proposition}
	
	Before proving the proposition, we make some reductions.
	
	First, since all the nontrivial characters of $\cP^{-m}/\cO$ are of the form $y\mapsto \psi_x(y):=\psi(xy)$ for some $x\in \cO - \cP^m$, we may as well fix $x\in \cO - \cP^m$ and prove the existence of $n_2\geqslant n_1$ such that for $n\geqslant n_2$, the pair $(K_n,\psi_x\circ \lambda_{n,m})$ is a supercuspidal type.
	
	Secondly, following Harish-Chandra, we call a function $f\in C_c^\infty(G(F))$ a {\em cusp form} if for every proper parabolic subgroup $P=MN\subsetneq G$ we have
	$$\displaystyle \int_{N(F)} f(gu) du=0, \mbox{ for every } g\in G(F).$$
	Similarly, a function $\varphi\in C_c^\infty(\mathfrak{g}(F))$ is said to be a {\em cusp form} if for every proper parabolic subgroup $P=MN\subsetneq G$ we have
	$$\displaystyle \int_{\fn(F)} \varphi(X+U) dU=0, \mbox{ for every } X\in \mathfrak{g}(F)$$
	where $\fn$ stands for the Lie algebra of the unipotent radical $N$ of $P$.
	
	The following characterization of supercuspidal types is probably well-known but for lack of a proper reference, we sketch the argument.
	
	\begin{lemma}\label{lem2 appendix}
		Let $K\subset G(F)$ be a compact-open subgroup and $\chi:K\to \mathbb{C}^\times$ a smooth character. Then, the pair $(K,\chi)$ is a supercuspidal type if and only if the function
		$$\displaystyle f_{K,\chi}: g\in G(F)\mapsto \left\{\begin{array}{ll}
			\chi(g) \mbox{ if } g\in K, \\
			0 \mbox{ otherwise,}
		\end{array}\right.$$
		is a cusp form.
	\end{lemma}
	
	\begin{proof}
		We know that $(K,\chi)$ is a supercuspidal type if and only if the compactly induced representation $\cind_K^G(\chi)$ is itself supercuspidal\footnote{By this, we just mean that all of the proper Jacquet modules of $\cind_K^G(\chi)$ vanish. Since this representation is not always admissible (e.g. when the center of $G(F)$ is non-compact), certain authors might prefer the terminology {\em quasi-cuspidal}.}. Thus, it suffices to check that $f_{K,\chi}$ is a cusp form if and only if all the proper Jacquet modules of $\cind_K^G(\chi)$ are zero. Let $P=MN$ be a proper parabolic subgroup of $G$. There is a natural embedding of $\cind_K^G(\chi)$ in the right regular representation $R$ on $C_c^\infty(G(F))$. Moreover, the image of this embedding is generated (as a module over $G(F)$) by the function $f_{K,\chi}$. Therefore, by exactness of the functor $(.)_N$ of $N(F)$-coinvariants, we just need to check that, for $g\in G(F)$, the image of $R(g)f_{K,\chi}$ in $C_c^\infty(G(F))_N$ is zero if and only if 
		$$\displaystyle \int_{N(F)} f_{K,\chi}(g'ug) du=0,\mbox{ for every } g'\in G(F).$$
		But this readily follows from the fact that the map
		$$\displaystyle C_c^\infty(G(F))\to C_c^\infty(G(F)/N(F)),\; f\mapsto \left(gN(F) \mapsto \int_{N(F)} f(gu)du\right)$$
		induces an isomorphism $C_c^\infty(G(F))_N\simeq C_c^\infty(G(F)/N(F))$.
	\end{proof}
	
	Set $f_{n,m,x}=f_{K_n,\psi_x\circ \lambda_{n,m}}$. By the above lemma, we are reduced to showing that $f_{n,m,x}$ is a cusp form for $n$ sufficiently large. For every $n\geqslant 0$, we let $\varphi_{n,m,x}\in C_c^\infty(\mathfrak{g}(F))$ be the function defined by
	$$\displaystyle \varphi_{n,m,x}(X)=\left\{\begin{array}{ll}
		\psi(x\langle Y_{n+m}, X\rangle) \mbox{ if } X\in L_n, \\
		0 \mbox{ otherwise,}
	\end{array} \right.$$
	for every $X\in \mathfrak{g}(F)$. Note that
	$$\displaystyle \varphi_{n,m,x}(X)=\varphi_{m,x}(\varpi^{-n}X)$$
	where for simplicity of notation we have set $\varphi_{m,x}=\varphi_{0,m,x}$, and
	$$\displaystyle f_{n,m,x}=\exp_* \varphi_{n,m,x}$$
	where for $\varphi\in C_c^\infty(\mathfrak{g}(F))$, we write $\exp_* \varphi$ for the function on $G(F)$ given by
	$$\displaystyle (\exp_*\varphi)(g)=\left\{\begin{array}{ll}
		\varphi(X) \mbox{ if } g=\exp(X) \mbox{ with } X\in \omega, \\
		0 \mbox{ otherwise,}
	\end{array} \right.$$
	for $g\in G(F)$.

	The following lemma is \cite[(5)]{BP}.
	
	\begin{lemma}
		Let $\varphi\in C_c^\infty(\mathfrak{g}(F))$ be a cusp form and set $\varphi_\lambda(X)=\varphi(\lambda^{-1}X)$ for $\lambda\in F^\times$, $X\in \mathfrak{g}(F)$. Then, for $\lambda$ sufficiently close to $0$ the function $f_\lambda=\exp_*\varphi_\lambda$ is a cusp form on the group.
	\end{lemma}
	
	By the above, we are finally reduced to checking that $\varphi_{m,x}$ is a cusp form. We define a Fourier transform $C_c^\infty(\mathfrak{g}(F))\to C_c^\infty(\mathfrak{g}^*(F))$, $\varphi\mapsto \widehat{\varphi}$, by
	$$\displaystyle \widehat{\varphi}(Y)=\int_{\mathfrak{g}(F)} \varphi(X) \psi(\langle Y,X\rangle) dX,\; \varphi\in C_c^\infty(\mathfrak{g}(F)), Y\in \mathfrak{g}^*(F)$$
	where $dX$ is a Haar measure on $\mathfrak{g}(F)$, the precise normalization of which does not really matter. We have the following characterization of cusp forms on $\mathfrak{g}(F)$.
	
	\begin{lemma}
		A function $\varphi\in C_c^\infty(\mathfrak{g}(F))$ is a cusp form if and only if $\widehat{\varphi}$ is supported in $\mathfrak{g}^*(F)_{\elli}$.
	\end{lemma}
	
	\begin{proof}
		This is essentially \cite[Proof of Lemma 1]{BP}. Let $P=MN\subsetneq G$ be a proper parabolic subgroup and $\mathfrak{n}$ be the Lie algebra of $N$. We have a commutative diagram
		$$\displaystyle \xymatrix{ C_c^\infty(\mathfrak{g}(F)) \ar[r]^{\cF_{\mathfrak{g}}} \ar[d]^{\int_{\fn}} & C_c^\infty(\mathfrak{g}^*(F)) \ar[d]^{\Res} \\ C_c^\infty((\mathfrak{g}/\fn)(F)) \ar[r]^{\cF_{\mathfrak{g}/\fn}} & C_c^\infty((\mathfrak{g}/\fn)^*(F))}$$
		where $\cF_{\mathfrak{g}}$ denotes the Fourier transform already introduced, $\cF_{\mathfrak{g}/\fn}$ is the Fourier transform on $(\mathfrak{g}/\fn)(F)$ defined in a similar way (with respect to the measure quotient of the Haar measures on $\mathfrak{g}(F)$ and $\fn(F)$), the right vertical arrow is the restriction map to $\fn^\perp(F)=(\mathfrak{g}/\fn)^*(F)$ and the left vertical arrow sends a function $\varphi\in C_c^\infty(\mathfrak{g}(F))$ to the function
		$$\displaystyle X\in \mathfrak{g}(F)/\fn(F)\mapsto \int_{\fn(F)} \varphi(X+U) dU.$$
		By definition, a function $\varphi\in C_c^\infty(\mathfrak{g}(F))$ is a cusp form if and only if its image by this left vertical arrow is zero for every proper parabolic subgroup $P$. By injectivity of the Fourier transform, and commutativity of the above diagram, this is equivalent to asking that the Fourier transform $\widehat{\varphi}$ vanishes on $\fn^\perp(F)$ for every proper parabolic subgroup $P=MN\subsetneq G$. But, this last condition is just saying that $\widehat{\varphi}$ is supported in $\mathfrak{g}^*(F)_{\elli}$.
	\end{proof}
	
	An easy computation show that $\widehat{\varphi_{m,x}}=\vol(L_0) \mathbf{1}_{-xY_m+L_0^\perp}$ and therefore we aim to check that $-xY_m+L_0^\perp\subset \mathfrak{g}^*(F)_{\elli}$. As
	$$\displaystyle -xY_m+L_0^\perp=-x\varpi^{-m+1}Y_1+L_0^\perp=-x\varpi^{-m+1}(Y_1+\varpi^{m-1}x^{-1}L_0^\perp)$$
	it is equivalent to showing that $Y_1+\varpi^{m-1}x^{-1}L_0^\perp\subset \mathfrak{g}^*(F)_{\elli}$. Since $x\in \cO-\cP^m$, we have $\varpi^{m-1}x^{-1}\in \cO$ and therefore
	$$\displaystyle Y_1+\varpi^{m-1}x^{-1}L_0^\perp\subset Y_1+L_0^\perp \subset \mathfrak{g}^*(F)_{\elli}$$
	by the choice of $Y_1$ and $L_0$ (see Lemma \ref{lem1 appendix}). This ends the proof of Proposition \ref{prop1 appendix} and therefore also of Theorem \ref{theo1 appendix}.

	\subsection{On the case of positive characteristic}
	
	We keep the previous notations except that we assume now that $F$ is a local non-Archimedean field of positive characteristic. Then, we claim that the obvious analog of Theorem \ref{theo1 appendix} does not hold anymore in this setting, at least when the group $G$ is not anisotropic modulo the center. More precisely, we have:
	
	\begin{proposition}
	Assume that $G$ is not anisotropic modulo the center. Then, there exists $n\geqslant 1$ such that $G$ has no omni-supercuspidal type of level $p^n$.
	\end{proposition}

\begin{proof}
This all come from the following lemma.

\begin{lemma}
Let $N$ be a unipotent algebraic group defined over $F$ with nilpotent index $m$, $U_N\subset N(F)$ be a compact-open subgroup and $\chi$ be a smooth character of $U_N$. Then, $\chi^{p^m}$ is trivial.
\end{lemma}

\begin{proof}
Indeed, let $N=N_0\supset N_1\supset \ldots \supset N_m=\{ 1 \}$ be a central series such that $N_i/N_{i+1}$ is abelian for $0\leqslant i\leqslant m-1$. Then, for every $i$, $N_i(F)/N_{i+1}(F)$ is an abelian group killed by $p$. Hence, by an easy induction, for every $u\in N(F)$ and $1\leqslant i\leqslant m$, we have $u^{p^i}\in N_i(F)$. Setting $i=m$ gives the result.
\end{proof}

Let $P=MN\subsetneq G$ be a proper parabolic subgroup and $m$ be the nilpotent index of $N$. We claim that $G$ has no omni-supercuspidal type of level $p^{m+1}$. Indeed, assume by way of contradiction that $(U,\lambda)$ is such an omni-supercuspidal type of level $p^{m+1}$. Let $\chi: \mathbb{Z}/p^{m+1}\mathbb{Z}\to \mathbb{C}^\times$ be a character of order $p^{m+1}$. Then, by the above lemma, $\chi^{p^m}\circ \lambda$ is trivial on $U_N:=U\cap N(F)$. In particular, we have
$$\displaystyle \int_{N(F)} f_{U,\chi^{p^m}\circ\lambda}(u)du=\vol(U_N)\neq 0$$
where $f_{U,\chi^{p^m}\circ\lambda}$ is the function associated to the pair $(U,\chi^{p^m}\circ\lambda)$ as in the previous section. It follows that the function $f_{U,\chi^{p^m}\circ\lambda}$ is not a cusp form and consequently, by Lemma \ref{lem2 appendix}, that $(U,\chi^{p^m}\circ\lambda)$ is not a supercuspidal type. As $\chi^{p^m}$ is non trivial, this contradicts the assumption that $(U,\lambda)$ was an omni-supercuspidal type.
\end{proof}
	
\begin{remark}
From the proof, we actually see that if $G$ has a proper (maximal) parabolic subgroup with an abelian unipotent radical then there is even no omni-supercuspidal type of level $p^2$. There are many examples of groups satisfying this condition, including all the quasi-split classical groups, and a classification of such (in the split case) is given in \cite[Remark 2.3]{RRS}.
\end{remark}

 \textbf{Acknowledgements}. The project leading to this publication has received funding from Excellence Initiative of Aix-Marseille University-A*MIDEX, a French ``Investissements d'Avenir" programme.

\renewcommand\refname{References for Appendix D}

\markright{\textsc{Congruences and supercuspidal representations}}
\markleft{\textsc{Jessica Fintzen and Sug Woo Shin}}

\renewcommand\refname{References}

\end{document}